\documentclass[12pt]{article}

\usepackage{amsmath,amssymb,amsthm,latexsym,vmargin,tikz}
\usepackage{mathtools}
\usepackage{overpic}
\usepackage{comment}
\usepackage{enumitem}
\usepackage{scalerel}
\usepackage{authblk}
\usetikzlibrary{decorations.markings}

\usetikzlibrary{hobby}

\newtheorem{theorem}{Theorem}

\newtheorem{lemma}[theorem]{Lemma}
\newtheorem{proposition}[theorem]{Proposition}
\newtheorem{corollary}[theorem]{Corollary}

\theoremstyle{definition}
\newtheorem{definition}[theorem]{Definition}
\newtheorem{remark}[theorem]{Remark}

\newtheorem{assumptions}[theorem]{Assumptions}
\newtheorem{rhproblem}[theorem]{RH problem}

\numberwithin{equation}{section}
\numberwithin{theorem}{section}

\DeclareMathOperator{\Bal}{Bal}
\DeclareMathOperator{\divis}{div}
\DeclareMathOperator{\diag}{diag}
\DeclareMathOperator{\Prob}{Prob}

\DeclareMathOperator{\supp}{supp}
\DeclareMathOperator{\tr}{tr}

\DeclareMathOperator{\adj}{adj}
\DeclareMathOperator{\Disc}{Disc}
\newcommand{\ds}{\displaystyle}

\let\Re\undefined
\let\Im\undefined
\DeclareMathOperator{\Re}{Re}
\DeclareMathOperator{\Im}{Im}

\title{Matrix valued orthogonal polynomials arising from  hexagon tilings with $3\times 3$-periodic weightings }

\author{Arno B.J. Kuijlaars 
{\small  \authorcr \ \authorcr
Department of Mathematics,  Katholieke Universiteit Leuven, \authorcr 
	Celestijnenlaan 200 B bus 2400, 3001 Leuven, Belgium
\authorcr arno.kuijlaars@kuleuven.be}}
\date{}                     %% if you don't need date to appear

\begin{document}
	\maketitle

\begin{abstract}
	Matrix valued orthogonal polynomials (MVOP) appear in
	the study of doubly periodic tiling models. Of particular
	interest is their limiting behavior as the degree tends
	to infinity.	
	In recent years, MVOP associated with doubly periodic 
	domino tilings of the Aztec diamond have  
	been successfully analyzed. 
	The MVOP related to doubly periodic lozenge tilings of a hexagon are more complicated. In this paper we focus
	on a special subclass of hexagon tilings with $3 \times 3$ periodicity. 
	The special subclass leads to a genus one spectral curve  with additional symmetries that allow us to 
	find an equilibrium measure in an external field explicitly. 
	The equilibrium measure gives the asymptotic distribution for the zeros of the determinant of the MVOP. 
	The associated $g$-functions appear in the strong
	asymptotic formula for the MVOP that we obtain
	from a steepest descent analysis of the Riemann-Hilbert problem for MVOP. 
\end{abstract}

\section{Introduction and statement of results}

We study the matrix valued orthogonal polynomials (MVOP) that are
associated with  random tilings 
of a regular hexagon with doubly periodic weightings of period three. We restrict to a special class with additional
symmetries that allow us to perform a
full  asymptotic analysis of the MVOP.

\subsection{Matrix valued orthogonal polynomials}

Matrix valued orthogonal polynomials play a role in the analysis
of tiling models with doubly periodic weights. This connection was found in \cite{DK21} and used in subsequent works \cite{BD19,Cha21a,Cha21b,CDKL20,GK21,KP24+}.

Here we study the MVOP that come from a lozenge tiling model
of a regular hexagon with $3 \times 3$-periodic weights. 
The model depends on positive parameters $a_{jk}$, $b_{j k}$,
for $j,k = 1,2,3$ that are collected in three matrix valued functions
\begin{equation} \label{Tj3} 
	T_j(z) = \begin{pmatrix} a_{j1} & b_{j1} & 0 \\ 0 & a_{j2} & b_{j2} 
		\\ b_{j3} z & 0 & a_{j3} \end{pmatrix},
	\qquad j=1, 2, 3. \end{equation}
We put
\begin{equation} \label{Wz}
	W(z) = T_1(z) T_2(z) T_3(z).
	\end{equation}
See Appendix \ref{appendixA} for the connection with doubly periodic lozenge tilings. 

We take the size of the hexagon to be $3N \times 3BN \times 3CN$, 
with positive integers $N$, $BN$ and $CN$, so that it
fits nicely with the periodicity. Then the MVOP associated with this model  is a monic polynomial $P_N$, matrix valued of size $3 \times 3$ and having degree $N$, such that
\begin{equation} \label{PNortho}
		\frac{1}{2\pi i} \oint_{\gamma} P_N(z) \frac{W(z)^{(B+C)N}}{z^{(1+C)N}} z^j dz = 0_3,
			\qquad j=0,1, \ldots, N-1.
			\end{equation}
The contour $\gamma$ in \eqref{PNortho}
is any simple closed contour going around $0$ once in the positive direction, and the integral in \eqref{PNortho} is taken entrywise.
Thus, the right-hand side is the zero matrix of size $3 \times 3$.

The orthogonality \eqref{PNortho} is not related to a positive definite scalar product. Hence, existence and uniqueness of the MVOP does not follow from general theory of MVOP, see e.g.\ \cite{DPS08}. However, because of the connection to the random tiling model, the MVOP with the specific orthogonality \eqref{PNortho} does uniquely exist \cite{DK21,GK21}. 

Of interest is the asymptotic behavior of $P_N$ as $N \to \infty$.
Ideally, we would like to understand this for general parameters, for general dimensions of the hexagon, and also
for higher periodicities. In this paper we restrict to
a special subclass, namely we assume $B=C=1$ (i.e., the regular hexagon), and 
\begin{align} \label{Eq1ab} 
	a_{1k} a_{2k} a_{3k} & = 1, \qquad k=1,2,3, \\
	\label{Eq2ab} 
	b_{11} b_{22} b_{33}  =  b_{12} b_{23} b_{31} = b_{13} b_{21} b_{32}  & = 1, \\ 
	\label{Eq3ab} 
	\frac{a_{j1}a_{j2} a_{j3}}{b_{j1} b_{j2} b_{j3}} & = 1,
	\qquad j =1,2,3, \end{align}
	and
\begin{equation} \label{Ineq3ab}
	a_{11} b_{21} \neq a_{22} b_{11} \quad \text{ or } \quad 
	a_{12} b_{22} \neq a_{23} b_{12}. \end{equation}	
For a concrete model satisfying the conditions \eqref{Eq1ab}--\eqref{Ineq3ab} one may take
\begin{equation} \label{Tjconcrete} 
	T_1(z) = \begin{pmatrix} \alpha_1 & 1 & 0 \\ 0 & \alpha_2^{-1} & 1 \\ z & 0 & \alpha_1^{-1} \alpha_2 \\
	\end{pmatrix}, \
	T_2(z) = \begin{pmatrix} \alpha_1^{-1} & 1 & 0 \\ 0 & \alpha_2 & 1 \\ z & 0 & \alpha_1 \alpha_2^{-1} \\
	\end{pmatrix}, \
	T_3(z) = \begin{pmatrix} 1 & 1 & 0 \\ 0 & 1 & 1 \\ z & 0 & 1 \\
	\end{pmatrix}, \end{equation}
with positive parameters $\alpha_1, \alpha_2 >0$,
$(\alpha_1, \alpha_2) \neq (1,1)$. 

The conditions \eqref{Eq1ab}--\eqref{Ineq3ab} come from
the structure of the associated Riemann surface (spectral
curve) that we discuss in the next subsection.

\begin{remark}
The assumptions \eqref{Eq1ab}, \eqref{Eq2ab},
	\eqref{Eq3ab} reduce the number of
	parameters $a_{jk}$, $b_{jk}$ in the model
	from $18$ to $10$, as only eight of the nine equations \eqref{Eq1ab}--\eqref{Eq3ab} are linearly independent.	 There
	is a further reduction, since different
	sets of parameters in the matrices
	\eqref{Tj3} may give rise
	to essentially the same weight matrices \eqref{Wz}.
This happens if $D_1, D_2, D_3$ are diagonal matrices with positive diagonal entries, and $c_1, c_2, c_3$
are positive constants. Then putting
$D_4 = D_1$, and
\[ \widehat{T}_k = c_k D_k T_k D_{k+1}^{-1},
	\quad \text{for } k =1,2,3, \]
we obtain  matrices of the form \eqref{Tj3} such
that 
\[ \widehat{W}(z) = \widehat{T}_1(z)
	\widehat{T}_2(z) \widehat{T}_3(z)
		= c_1 c_2 c_3 D_1 W(z) D_1^{-1}. \]
The monic MVOP $\widehat{P}_N$ associated with $\widehat{W}$  
is then directly related to $P_N$ since  
\[ \widehat{P}_N(z) = D_1 P_N(z) D_1^{-1} \]
and the two models can be considered equivalent. 
It turns out that, under this equivalence, the number of parameters
further reduces to $2$, and each model is equivalent
to \eqref{Tjconcrete} for some choice
of $\alpha_1, \alpha_2$. We will not prove this here, as we will not use it in what follows.
\end{remark}

\subsection{The spectral curve}

The characteristic equation of \eqref{Wz}, i.e.,
\begin{equation} \label{Wcharpoly} 
		P(z,\lambda) :=	\det\left(\lambda I_3 - W(z)\right) = 0, \end{equation}
plays an important role in this work. It is referred to as the spectral curve.
The compact Riemann surface associated with \eqref{Wcharpoly}
is denoted by $\mathcal R$. Generic points on $\mathcal R$
are denoted by $p$ and $q$. Each $p \in \mathcal R$
has a $z$ and a $\lambda$ coordinate, that we denote by
$z(p)$ and $\lambda(p)$.

The matrix valued orthogonality is known to be related to scalar orthogonality on the spectral curve \eqref{Wcharpoly}, see e.g., \cite{Ber23, BGK23, Cha21b, GK21}. We will exploit this relation extensively, as we will use several notions coming from the
spectral curve, in particular the equilibrium measure, see section \ref{section15} below.  However, we do not deal with the
scalar orthogonality directly.
See \cite{Ber21,Ber22,DIP24+, DLR23+,FOX23} for  other recent works on orthogonality on a Riemann surface covering varying aspects
of asymptotic analysis, including steepest descent of Riemann-Hilbert problems
on a spectral curve \cite{Ber22, DIP24+}.

The spectral curve \eqref{Wcharpoly} for the weight matrix \eqref{Wz} is known to be a \textit{Harnack curve}, see  \cite{KO06,KOS06},
which means that the \textit{amoeba map} 
$A : p  \mapsto (\log |z(p)|, \log |\lambda(p)|)$ is at most $2$-to-$1$ when restricted to the curve. 
Since $P$ has real coefficients, the Harnack property means that 
$A(z,\lambda) = A(z',\lambda')$
with $P(z,\lambda) = P(z',\lambda') = 0$ implies that either $(z',\lambda') = (z,\lambda)$ or 
$(z',\lambda') = (\overline{z},\overline{\lambda})$.

We will not work with the amoeba in this paper, but rather view $\mathcal R$
as a three-fold covering of the complex $z$-plane. 
There is a canonical way to do so, because of the
following very special property of the eigenvalues of 
$W(z)$. If we order the three eigenvalues according
to decreasing absolute value, then for $z \in \mathbb C \setminus \mathbb R$ there is strict inequality, i.e.,
\begin{equation} \label{lambdajordered} 
	|\lambda_1(z)| > |\lambda_2(z)| > |\lambda_3(z)| > 0,
	\quad z \in \mathbb C \setminus \mathbb R. \end{equation}
This is a consequence of being a Harnack curve \cite{KOS06}.
In particular, $W(z)$ has simple spectrum if $z \in \mathbb C \setminus \mathbb R$, and we will write
\begin{equation} \label{Wdecomp}
	W(z) = E(z) \Lambda(z) E(z)^{-1}, 
	\quad \Lambda(z)  = \diag(\lambda_1(z), \lambda_2(z), \lambda_3(z)),
\end{equation} 
for $z \in \mathbb C \setminus \mathbb R$,
where $E(z)$ is a matrix containing the eigenvectors of $W(z)$.
Later we will make a specific choice for the 
matrix of eigenvectors, see Definition \ref{def:E}.

\begin{figure}[t]
	\begin{center}
		\begin{tikzpicture}[scale=0.8](15,10)(0,0)
			% First sheet
			\filldraw[gray!50!white] (6,3) --++(10,0) --++(2,2) --++(-10,0) --++(-2,-2);
			% Second sheet
			\filldraw[gray!50!white] (6,0) --++(10,0) --++(2,2) --++(-10,0) --++(-2,-2);
			% Third sheet
			\filldraw[gray!50!white] (6,-3) --++(10,0) --++(2,2) --++(-10,0) --++(-2,-2);
			
			% First sheet points
		\filldraw (7.5,4)  circle (2pt);	 
		\filldraw (10,4)  circle (2pt);	 
		\filldraw (11,4)  circle (2pt);	 
		\filldraw (17,4)  circle (2pt);	 
		\draw (7.5,4) node[below] {$z_1$};
		\draw (10,4) node[below] {$z_2$};
		\draw   (11,4) node[below] {$0$};
		\draw   (17,4) node[below] {$\infty$};
			% First sheet cuts
		\draw[very thick,black] (7.5,4)--++(2.5,0);
			% First sheet oval 
		\draw[dashed,ultra thick,purple] (7,4)--++(0.5,0);
		\draw[dashed,ultra thick,purple] (10,4)--++(7,0);
	
	% Second sheet points
	\filldraw (7.5,1)  circle (2pt);	 
	\filldraw (10,1)  circle (2pt);	 
	\filldraw (11,1)  circle (2pt);	 
	\filldraw (11.5,1)  circle (2pt);	 
	\filldraw (12.5,1)  circle (2pt);	 
	\filldraw (13.5,1)  circle (2pt);	 	 
	\filldraw (16.5,1) circle (2pt);
	\filldraw (17,1)  circle (2pt);	 
	\draw (7.5,1) node[below] {$z_1$};
	\draw (10,1) node[below] {$z_2$};
	\draw   (11,1) node[below] {$0$};
	\draw   (17,1) node[below] {$\infty$};
	\draw (11.5,1) node[below] {$z_3$};
	\draw (12.5,1) node[below] {$z_4$};
	\draw (13.5,1) node[below] {$z_5$};
	\draw (16.5,1) node[below] {$z_6$};
	% Second sheet cuts
	\draw[very thick,black] (7.5,1)--++(2.5,0);
	\draw[very thick,black] (11.5,1)--++(1,0);
	\draw[very thick,black] (13.5,1)--++(3,0);
	% Second sheet oval 
	\draw[dashed,ultra thick,purple] (7,1)--++(0.5,0);
	\draw[dashed,ultra thick,purple] (10,1)--++(1.5,0);
	\draw[dashed,ultra thick,red] (12.5,1)--++(1,0);
	\draw[dashed,ultra thick,purple] (16.5,1)--++(0.5,0);
		
	% Third sheet points	 	 
	\filldraw (9.2,-2) circle (2pt);
	\filldraw (9.5,-2) circle (2pt);
	\filldraw (9.8,-2) circle (2pt);
	\filldraw (11,-2)  circle (2pt);	 
	\filldraw (11.5,-2)  circle (2pt);	 
	\filldraw (12.5,-2)  circle (2pt);	 
	\filldraw (13.5,-2)  circle (2pt);	 
	\filldraw (16.5,-2) circle (2pt);
	\filldraw (17,-2)  circle (2pt);	 
	\draw   (11,-2) node[below] {$0$};
	\draw  (9.5,-2) node[below] {$\lambda = 0$};
	\draw   (17,-2) node[below] {$\infty$};
	\draw (11.5,-2) node[below] {$z_3$};
	\draw (12.5,-2) node[below] {$z_4$};
	\draw (13.5,-2) node[below] {$z_5$};
	\draw (16.5,-2) node[below] {$z_6$};
	% Third sheet cuts
	\draw[very thick,black] (11.5,-2)--++(1,0);
	\draw[very thick,black] (13.5,-2)--++(3,0);
	% Third sheet oval 
	\draw[dashed,ultra thick,purple] (7,-2)--++(4,0);
	\draw[dashed,ultra thick,purple] (10,-2)--++(1.5,0);
	\draw[dashed,ultra thick,red] (12.5,-2)--++(1,0);
	\draw[dashed,ultra thick,purple] (16.5,-2)--++(0.5,0);	
		
	% Vertical line between first and second sheet
	\draw[dashed,help lines,purple] (7.5,1)--(7.5,4);
	\draw[dashed,help lines,purple] (10,1)--(10,4);	
	
	% Vertical line between first and second sheet
			
	\draw[dashed,help lines,purple] (11.5,-2)--(11.5,1);
	\draw[dashed,help lines,red] (12.5,-2)--(12.5,1);
	\draw[dashed,help lines,red] (13.5,-2)--(13.5,1);	
	\draw[dashed,help lines,purple] (16.5,-2)--(16.5,1);			
		\end{tikzpicture}
	\end{center}
\caption{Sheet structure of the Riemann surface for generic parameters. \label{fig:sheets}}
\end{figure}
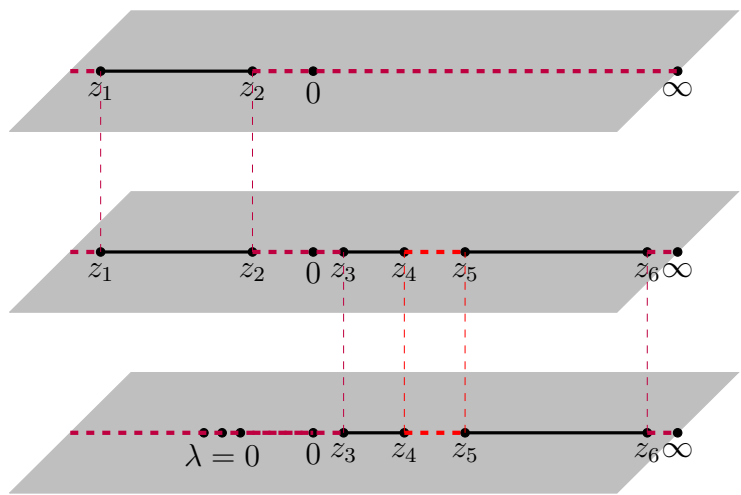

The ordering \eqref{lambdajordered} provides the Riemann surface
associated with \eqref{Wcharpoly} with 
a natural sheet structure, namely we put $\lambda = \lambda_k$ on the $k$th sheet for $k=1,2,3$.
Figure \ref{fig:sheets} shows the sheet 
structure of the Riemann surface for generic parameters.
There are six branch points on the real line,
say with $z$-coordinates $z_1, \ldots, z_6$, such that 
\begin{equation} \label{zjordering}
	-\infty \leq z_1 < z_2 \leq 0 \leq z_3 < z_4 \leq z_5 < z_6 \leq +\infty. \end{equation}
Then  $W(x)$ has three real eigenvalues if and only
if $x \in (-\infty,z_1] \cup [z_2, z_3] \cup [z_4,z_5]
\cup [z_6,\infty)$. 
For $x \in (z_1,z_2) \cup (z_3,z_4) \cup (z_4,z_5)$
there is one real eigenvalue, and a pair of non-real
complex conjugate eigenvalues, namely $\lambda_3(x)$
is the only real eigenvalue for $x \in (z_1,z_2)$
and $\lambda_1(x)$ is the only real eigenvalue for $x \in (z_3,z_4) \cup (z_5,z_6)$. The branch cut 
 $[z_1,z_2]$ on the negative real axis 
 connecting the first and second
 sheets and two branch cuts $[z_3,z_4]$ and
 $[z_5,z_6]$ on the positive real axis connect the second and third sheets. This is all a consequence of being a 
 Harnack curve.
 
  The real part of \eqref{Wcharpoly} consists
   of two ovals. The unbounded oval stretches out over all three sheets,
  as shown in purple (dashed) in Figure \ref{fig:sheets}. 
  The positive real axis
  on the first sheet, and the negative real axis on the third sheet
  are fully on the unbounded oval.
  It also  contains all points
 where $z$ or $\lambda$  are $0$ or $\infty$.  
 At these points we observe the following
 \begin{itemize}
 	\item  For $z=0$, the matrix $W(z)$ is upper triangular, and
 the eigenvalues are its diagonal entries
 \begin{equation} \label{z0values} 
 	a_{1k} a_{2k} a_{3k}, \quad k=1,2,3. 
 	\end{equation}
 \item   $\lambda=0$ is an eigenvalue of $W(z)$ if and only if
 one of the factors \eqref{Tj3} is singular.
 Since $\det T_j(z) = a_{j1} a_{j2} a_{j3} + b_{j1} b_{j2} b_{j3} z$, this  happens
 for the three negative  values
 \begin{equation} \label{lambda0values}
 	- \frac{a_{j1} a_{j2} a_{j3}}{b_{j1} b_{j2} b_{j3}}, \quad j = 1,2,3. \end{equation}
  The corresponding points are on the third sheet of
 the Riemann surface, see also Figure \ref{fig:sheets}.
 \item As $z \to \infty$, also $\lambda_k(z) \to 0$ 
 and the limits  $\lim\limits_{z \to \infty} \lambda_k(z)/z$ exists.
 This can be seen from the fact that \eqref{Wz}
 has the form 
 \[ W(z) = z L + U \]
 with a lower triangular matrix $L$ and an  upper triangular
 matrix $U$. The diagonal entries of $L$ give us the
 possible limits of $\lambda_k(z)/z$ as $z \to \infty$, 
 and these numbers are
 \begin{equation} \label{zinftyvalues} 
 	b_{11} b_{22} b_{33}, 	\quad b_{12} b_{23} b_{31}, \quad b_{13} b_{21} b_{32}.
  \end{equation}
 \end{itemize}

The bounded oval is a closed contour that consists
of the interval $[z_4,z_5]$ on the second and third sheets
of the Riemann surface, it is shown in red (also dashed)
in Figure \ref{fig:sheets}. If $z_4 = z_5$ then 
the bounded oval degenerates to a node. In that
case the Riemann surface has genus zero, otherwise the
genus is one.  

The assumptions for the present paper are the following.
\begin{assumptions} \label{assump12} We assume
\begin{itemize}
	\item[(a)] $0$ is a branch point of the Riemann surface that connects all three sheets and
	$\lambda_k(z) \to 1$ as $z \to 0$ for every $k=1,2,3$;
	
	\item[(b)] $\infty$ is a branch point of the Riemann surface that connects
	all three sheets and $\frac{\lambda_k(z)}{z} \to 1$ as $z \to \infty$ for every $k=1,2,3$;
	
	\item[(c)]  $\lambda = 0$ is taken at $z=-1$ only, i.e., $\lambda_3$
	has a triple zero at $z=-1$.
	
	\item[(d)] The
	bounded oval does not degenerate to a node (i.e., the
	Riemann surface has genus one).
\end{itemize}
\end{assumptions}
The assumptions (a) and (b) imply that $z_1 = -\infty$,
$z_2 = 0 = z_3$ and $z_6 = +\infty$ in \eqref{zjordering},
and we have the sheet structure as in Figure \ref{fig:sheets2}.
Due to (a)-(c) there is only one point on the Riemann surface with $z=0$, one point with $z=\infty$ and one point with $\lambda=0$. 
We call these three points
$P_0$, $P_{\infty}$, and $P_1$, respectively,
as also indicated in Figure \ref{fig:sheets2}.

Furthermore, assumption (a) tells us that the numbers \eqref{z0values} 
are equal to $1$, while (c) gives that the
numbers \eqref{lambda0values} are equal to $-1$,
and (b) gives that the numbers \eqref{zinftyvalues} 
are equal to $1$. That is, the equalities listed in \eqref{Eq1ab}, \eqref{Eq2ab}, and \eqref{Eq3ab} hold,
and they are equivalent to Assumptions (a)-(c).

	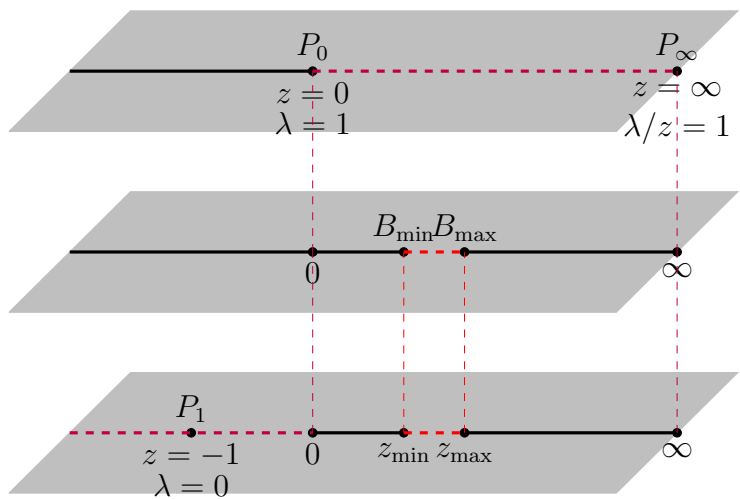
\begin{figure}[t]
		\begin{center}
			\begin{tikzpicture}[scale=0.8](15,10)(0,0)
				
				% First sheet
				\filldraw[gray!50!white] (6,3) --++(10,0) --++(2,2) --++(-10,0) --++(-2,-2);
				% Second sheet
				\filldraw[gray!50!white] (6,0) --++(10,0) --++(2,2) --++(-10,0) --++(-2,-2);
				% Third sheet
				\filldraw[gray!50!white] (6,-3) --++(10,0) --++(2,2) --++(-10,0) --++(-2,-2);
				
				% First sheet points
				%\filldraw (7.5,4)  circle (2pt);	 
				%\filldraw (10,4)  circle (2pt);	 
				\filldraw (11,4)  circle (2pt);	 
				\filldraw (17,4)  circle (2pt);	
				\draw   (11,4) node[above] {$P_0$};
				\draw  (11,4) node[below] {$z = 0$};
				\draw  (11,3.5) node[below] {$\lambda=1$};
				\draw   (17,4) node[above] {$P_\infty$};
				\draw   (17,4) node[below] {$z=\infty$}; 
				\draw   (17,3.5) node[below] {$\lambda/z =1$};
				% First sheet cuts
				\draw[very thick,black] (7,4)--++(4,0);
				% First sheet oval 
				\draw[dashed,very thick,purple] (11,4)--++(6,0);
				
				% Second sheet points
				%	\filldraw (7.5,1)  circle (2pt);	 
				%	\filldraw (10,1)  circle (2pt);	 
				\filldraw (11,1)  circle (2pt);	 
				%	\filldraw (11.5,1)  circle (2pt);	 
				\filldraw (12.5,1)  circle (2pt);	 
				\filldraw (13.5,1)  circle (2pt);	 	 
				%	\filldraw (16.5,1) circle (2pt);
				\filldraw (17,1)  circle (2pt);	 
				\draw   (11,1) node[below] {$0$};
				\draw   (17,1) node[below] {$\infty$};
				\draw   (12.5,1) node[above] {$B_{\min}$};
			%	\draw   (12.5,1) node[below] {$z_{\min}$};
				\draw   (13.5,1) node[above] {$B_{\max}$};
			%	\draw   (13.5,1) node[below] {$z_{\max}$};
				% Second sheet cuts
				\draw[very thick,black] (7,1)--++(5.5,0);
				\draw[very thick,black] (13.5,1)--++(3.5,0);
				% Second sheet oval 
				%			\draw[very thick,purple] (10,1)--++(1.5,0);
				\draw[dashed, very thick,red] (12.5,1)--++(1,0);
				
				% Third sheet points	 	 
				\filldraw (9,-2) circle (2pt);
				\filldraw (11,-2)  circle (2pt); 
				\filldraw (12.5,-2)  circle (2pt);	 
				\filldraw (13.5,-2)  circle (2pt);	 
				\filldraw (17,-2)  circle (2pt);	 
				\draw   (11,-2) node[below] {$0$};
				\draw   (9,-2) node[above] {$P_1$};
				\draw  (9,-2.5) node[below] {$\lambda = 0$};
				\draw  (9,-2) node[below] {$z=-1$};
				\draw   (17,-2) node[below] {$\infty$};
				\draw   (12.5,-2) node[below] {$z_{\min}$};
				\draw   (13.5,-2) node[below] {$z_{\max}$};
				% Third sheet cuts
				\draw[very thick,black] (11,-2)--++(1.5,0);
				\draw[very thick,black] (13.5,-2)--++(3.5,0);
				% Third sheet oval 
				\draw[dashed,very thick,purple] (7,-2)--++(4,0);
				\draw[dashed,very thick,red] (12.5,-2)--++(1,0);
				
				% Vertical line between first and second sheet
				di			%\draw[dashed,help lines,purple] (7,1)--(7,4);
				\draw[dashed,help lines,purple] (11,1)--(11,4);	
				\draw[dashed,help lines,purple] (17,1)--(17,4);	
				
				% Vertical line between first and second sheet
				%	\draw[dashed,help lines,purple] (7,-2)--(7,1);
				\draw[dashed,help lines,purple] (11,-2)--(11,1);
				\draw[dashed,help lines,red] (12.5,-2)--(12.5,1);
				\draw[dashed,help lines,red] (13.5,-2)--(13.5,1);	
				\draw[dashed,help lines,purple] (17,-2)--(17,1);			
			\end{tikzpicture}
		\end{center}
		\caption{Sheet structure for the Riemann surface
			under Assumptions \ref{assump12}. 
			The Riemann
			surface has the equation \eqref{Plamz} with $\beta > 0$. \label{fig:sheets2}}
	\end{figure}
	
In addition to the branch points at $P_0$ and $P_{\infty}$
that connect all three sheets, the Riemann surface
has branch points at $B_{\min}$ and $B_{\max}$ that connect the second and third
sheets. We write
\begin{equation} \label{BminBplus} 
	B_{\min} = (z_{\min}, \lambda_{\min}), \quad B_{\max} = (z_{\max}, \lambda_{\max}), \end{equation}
with $0 < z_{\min} < z_{\max}$. The coordinates can be found explicitly, see \eqref{zminzmax} below, 
and it turns out that $z_{\min} \leq 1 \leq z_{\max}$. Note that $z_{\min} = z_4$ and $z_{\max} = z_5$ in our earlier notation
from \eqref{zjordering}.
The assumption (d) means that 
\begin{equation} \label{zminzmaxorder} z_{\min} < 1 < z_{\max} 
	\end{equation}
and this is equivalent to \eqref{Ineq3ab}.

\subsection{Asymptotic formula for MVOP}

Under the Assumptions \ref{assump12}, or equivalently
\eqref{Eq1ab}--\eqref{Ineq3ab}, we have a
strong asymptotic formula for the MVOP \eqref{PNortho}
in case $B=C=1$.
We use $\mathbb T = \{ z \in \mathbb C \mid |z| = 1 \}$ to denote the unit circle in the complex plane.

\begin{theorem} \label{theorem13}
	Suppose $B=C=1$ and assume that \eqref{Eq1ab}--\eqref{Ineq3ab}
	hold. Then there exist
	matrix valued functions $G$  and $A_N$, a constant matrix
	$L$ (all of size $3 \times 3$) and a constant $c > 0$ 	such that 
	\begin{equation} \label{PNasymp} 
		P_N(z) = L^N \left(A_N(z) + \mathcal{O}\left( \frac{e^{-cN}}{1+|z|}\right) \right) 	G(z)^N 
		\quad \text{ as } N \to \infty,
	\end{equation}
	uniformly for $z$ in compact subsets of $\overline{\mathbb C} \setminus \mathbb T$. Moreover,
	\begin{enumerate}
		\item[\rm (a)] $G$ is defined and analytic on $\mathbb C \setminus (\mathbb T \cup [z_{\min}, z_{\max}])$, and it takes the form
		\begin{equation} \label{Gform} 
		G(z) = E(z) \begin{pmatrix} e^{g_1(z)} & 0 & 0 \\
			0 & e^{g_2(z)} & 0 \\
			0 & 0 & e^{g_3(z)} \end{pmatrix}  E(z)^{-1} 
		\end{equation}
		where $E(z)$ is the matrix of eigenvectors of $W(z)$,
		and
		with three $g$-functions that behave like 
		\begin{equation} \label{gjasymp} 
			g_j(z)  =  \log z + \mathcal{O}(z^{-1/3}) \quad \text{ as } z \to \infty, \end{equation}
		\item[\rm (b)]
		$L$ is a unit lower triangular matrix (i.e., lower triangular with ones on the diagonal) with
		the property that for each $N$,
		\begin{equation} \label{GNasymp} L^N G(z)^N = z^N I_3 + \mathcal{O}(z^{N-1}) 
		\quad \text{ as } z \to \infty, \end{equation}
		\item[\rm (c)] $A_N$ only depends on the parity of $N$, 
		i.e., there exist $A_e$ and $A_o$ such that
		\begin{equation} \label{ANparity} 
			A_N(z) = \begin{cases} A_e(z), & \text{ if $N$ is even}, \\
				A_o(z), & \text{ if $N$ is odd},
		\end{cases} \end{equation}	
		where $A_e$ and $A_o$ are defined and analytic in 
		$\mathbb C \setminus (\mathbb T \cup [z_{\min}, z_{\max}])$,
		with
		\begin{equation} \label{ANasymp} 
			\lim_{z \to \infty} A_e(z) = \lim_{z \to \infty} A_o(z) = I_3. \end{equation}
		\item[\rm (d)] $G$ and $A_N$ have boundary values on
		$[z_{\min}, 1) \cup (1, z_{\max}]$ that satisfy
		\begin{align} \label{Gjump1}
			G_+  & = G_- E \begin{pmatrix} 1 & 0 & 0 \\ 0 & -1 & 0 \\ 0 & 0 & -1 \end{pmatrix} E^{-1}, \\
			(A_N)_+  & = (A_N)_- E \begin{pmatrix} 1 & 0 & 0 \\ 0 & (-1)^N & 0 \\ 0 & 0 & (-1)^N \end{pmatrix} E^{-1}.			\label{ANjump1}
		\end{align}
	\end{enumerate}	
\end{theorem}	
The subscripts $\pm$ in \eqref{Gjump1} and \eqref{ANjump1}
are used to denoted the $+$ or $-$ boundary values on 
the real line.  For example, 
$G_+(x) = \lim\limits_{z \to x \atop \Im z > 0} G(z)$ and
$G_-(x) = \lim\limits_{z \to x \atop \Im z < 0} G(z)$, if
$x \in \mathbb [z_{\min},1) \cup (1,z_{\max}]$.
From \eqref{Gjump1} and \eqref{ANjump1} it follows that $A_N(z) G^N(z)$ is
analytic across $[z_{\min}, 1) \cup (1,z_{\max}]$.

The proof of Theorem \ref{theorem13} is in section \ref{section101}.

\subsection{Limiting zero distribution of $\det P_N$}
Since $\det L = 1$ (as $L$ is unit lower triangular) 
we find from Theorem \ref{theorem13} that 
\begin{equation} \label{detPNasymp} 
	\det P_N(z) = \left(\det A_N(z) + \mathcal{O} \left(\frac{e^{-cN}}{1+|z|} \right) \right)
		e^{N(g_1(z) + g_2(z) + g_3(z))} 
		\end{equation}
as $N \to \infty$, 
uniformly for $z$ in compact subsets of $\overline{\mathbb C} \setminus \mathbb T$.
This is a strong asymptotic formula for the scalar polynomials 
$\det P_N$ as $N \to \infty$. Observe that $\det P_N$ has degree $3N$.

It will follow from \eqref{detPNasymp} that most
of the zeros of $\det P_N$ tend to the unit circle as $N \to \infty$. The next result deals with the 
limit of the normalized zero counting measures
\begin{equation} \label{nudetPNdef} \nu(\det P_N) := \frac{1}{3N} \sum_{z: \det P_N(z) = 0} \delta_z 
 \end{equation}
as $N \to \infty$. The weak$^*$ limit of the
measures \eqref{nudetPNdef} is a probability
measure on $\mathbb T$, which will arise as
the pushforward of a measure $\mu$ on $\mathcal R$
under the projection map
\begin{equation} \label{pizdef} \pi_z : \mathcal R \to \overline{\mathbb C} : 
	p \mapsto z(p). \end{equation}

\begin{theorem} \label{theorem14}
	Suppose that \eqref{Eq1ab}--\eqref{Ineq3ab} hold.
For $j=1,2,3$, let $\Gamma_j$ denote the unit circle $|z|=1$ on the $j$th sheet of
	the Riemann surface $\mathcal R$ associated
	with the spectral curve \eqref{Wcharpoly}, see also Figure \ref{fig:sheets3}. 
	Then there is a probability measure 
	$\mu$ on $\Gamma_1 \cup \Gamma_2$ with a real analytic and positive
	density such that the following hold. 

\begin{enumerate}
	\item[\rm (a)] For every continuous function $h : \mathcal R \to \mathbb C$
	that is harmonic on $\mathcal R \setminus (\Gamma_1 \cup \Gamma_2)$
	one has 
	\begin{equation} \label{hbal}
		\int h d\mu = h(P_{\infty}) - h(P_1) + h(P_0)
		\end{equation} 
		where $P_0$, $P_1$, and $P_{\infty}$ are the three special points on the Riemann surface, see Figure \ref{fig:sheets2}.
	\item[\rm (b)] 
	Let $\mu_* = (\pi_z)_*(\mu)$ be the pushforward measure of $\mu$ under the projection map \eqref{pizdef}.
	Then $\mu_*$ is a probability measure on $\mathbb T$
	with the property that
\begin{equation} \label{Unu} 
	\Re \left(g_1(z) + g_2(z) + g_3(z) \right) = 3 \int \log |z-s| d\mu_* (s),
	\quad z \in \mathbb C \setminus \mathbb T.
	\end{equation}
	\item[\rm (c)] Let $P_N$ be the MVOP satisfying
	the matrix orthogonality \eqref{PNortho} with $B=C=1$. 
	Then 
$\mu_*$ is the weak$^*$-limit of the normalized zero counting measures \eqref{nudetPNdef}
of the polynomials $\det P_N$ as $N \to \infty$.
\end{enumerate}
\end{theorem}

Part (a) of Theorem \ref{theorem14} is proved in section \ref{section3}, and parts (b) and (c) are proved
in the final sections \ref{section102} and \ref{section103}
at the end of the paper.

\begin{figure}[t]
	\begin{center}
		\includegraphics[trim=0 12cm 2cm 2cm, clip,scale=0.5]{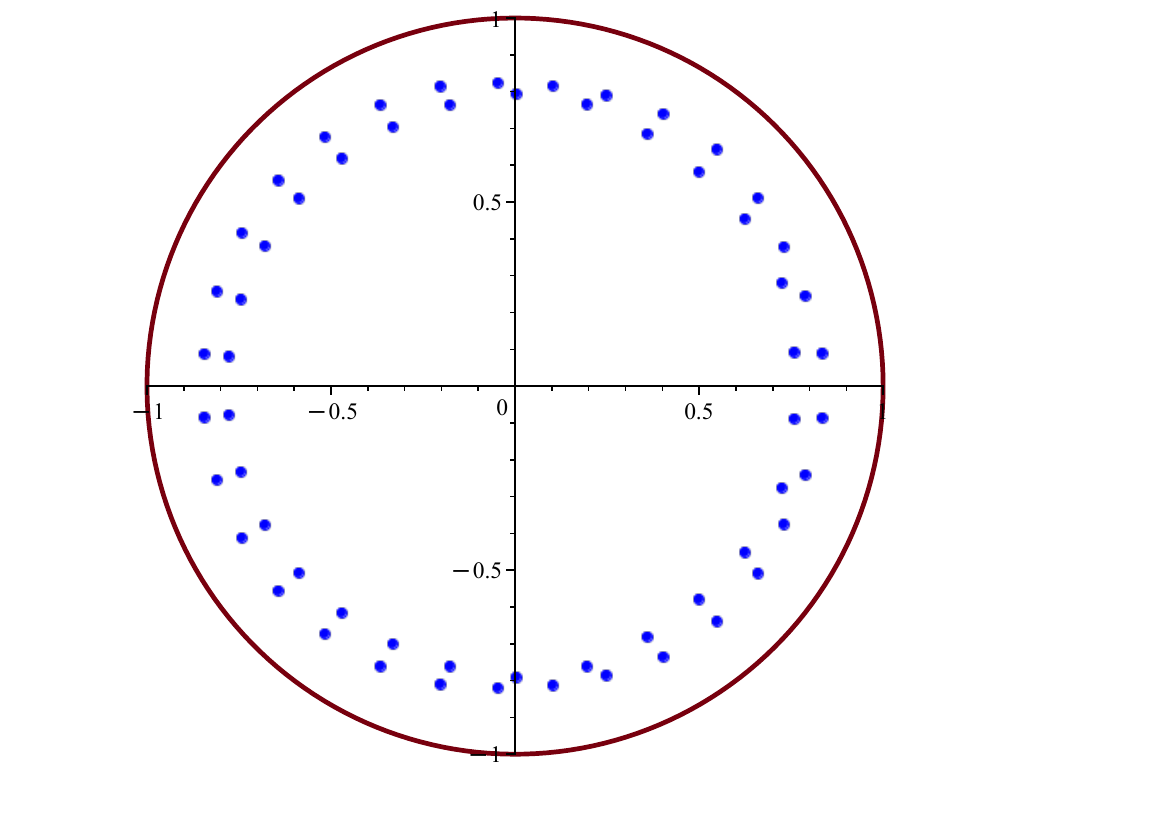}
		\caption{Zeros of $\det P_N$ with $N=20$ 
			coming from the model \eqref{Tjconcrete}
			with parameters $\alpha_1 = 0.2$, $\alpha_2 = 2$.
			The zeros tend to the unit circle, in the sense of
			weak convergence of normalized zero counting measures,
		by Theorem \ref{theorem14}.
			\label{fig:zeros}}
	\end{center}
\end{figure}

\begin{remark} \label{remark15}
The formula \eqref{hbal} determines  $\mu$, since for
a given continuous function $f$ on $\Gamma_1 \cup \Gamma_2$, one may
obtain $h$ as the solution of the Dirichlet problem for harmonic functions 
on $\mathcal R \setminus (\Gamma_1 \cup \Gamma_2)$ with $f$ as
boundary value,
and then $\int f d\mu  = h(P_{\infty}) - h(P_1) + h(P_0)$ by formula \eqref{hbal}.
In this way, $\mu$ is defined as a signed measure. It is part of the result of Theorem \ref{theorem14} that $\mu$ is 
positive, and thus a probability measure, since $\int d\mu = 1$.

The property \eqref{hbal} may equivalently be
expressed in terms of balayage measure of point masses
onto $\Gamma_1 \cup \Gamma_2$. Indeed,
\begin{equation} \label{muBal} 
	\mu = \Bal\left( \delta_{P_{\infty}} - \delta_{P_1} + \delta_{P_0}; \Gamma_1 \cup \Gamma_2 \right). \end{equation}
as, for any $p \in \mathcal R$, the balayage measure
(also known as harmonic measure) $\widehat{\delta}_{p}
	= \Bal(\delta_p; \Gamma_1 \cup \Gamma_2)$ satisfies
\begin{equation} \label{deltaBal}	
	 \int h d \widehat{\delta}_p = h(p) \end{equation}
whenever $h$ is continuous on $\mathcal R$ and harmonic on $\mathcal R \setminus (\Gamma_1 \cup \Gamma_2)$.
\end{remark}

See Figure~\ref{fig:zeros} for a plot of the zeros of 
$\det P_N$ for $N=20$ and for some choice of parameters. 
Note that for $N=20$, the zeros are not that close to the unit circle.
Also observe that the zeros seem to arrange themselves in
two rings. We do not have an explanation for this phenomenon.

\subsection{Equilibrium measure in external field}
\label{section15}

The probability measure $\mu$ from Theorem \ref{theorem14} is an equilibrium measure in an external field on a contour satisfying the $S$-property. 

It was argued in the work \cite{BGK23} of the author with Bertola and Groot 
that,  for MVOP with varying orthogonality \eqref{PNortho}, 
one should be looking for
a system of contours $\Gamma$ on $\mathcal R$, homotopic
in $\mathcal R \setminus \{ P_0, P_{\infty} \}$ to the union of
the unit circles $|z|=1$ on the three sheets,   
and a probability measure $\mu$
on $\Gamma$ with the $S$-property in the presence of the external field
$\Re V$ where
\begin{equation} \label{Vdef}  V = 2 \log z - 2 \log \lambda 
	\quad \text{ on } \mathcal R. \end{equation}
	
Equilibrium measures with external fields in the complex plane are very well-studied, see  \cite{ST97}. Curves with $S$-properties were first studied by Stahl \cite{Sta85,Sta86} in connection with orthogonal polynomials with complex weights. 
Later contributions include   
\cite{AS24+, KS15, MR11, MR16, MS16, Rak12}
where also varying weights and $S$-curves
in an external field were studied.

In the present work we are not in the complex plane, but on a Riemann surface, and the  potential theory is with respect to the bipolar Green's
kernel with one singularity at $P_{\infty}$, here denoted
by $G_{P_{\infty}}$, see \cite{BGK23, Chi18,Chi19,Chi20,Ski15}.

\begin{definition} \label{definition16}
The bipolar Green's kernel $G_{P_{\infty}}$ on $\mathcal R$ with pole
at $P_{\infty}$  is a real valued function defined
on $\{ (p,q) \in \mathcal R \times \mathcal R \mid p \neq q, p \neq P_{\infty}, q \neq P_{\infty} \}$ such that 
\begin{enumerate}
	\item[\rm (a)] for every $q \in \mathcal R \setminus \{ P_{\infty}\}$, the function $p \mapsto G_{P_{\infty}}(p,q)$
	is harmonic on $\mathcal R \setminus \{ P_{\infty}, q\}$,
	\item[\rm (b)] if $z$ is a local coordinate at $q$, then
	\begin{align} \label{bGfatq} G_{P_{\infty}}(p,q) = - \log |z(p)| + \mathcal{O}(1) \quad \text{ as } p \to q, \end{align}
	\item[\rm (c)] if $z_{\infty}$ is a local coordinate at $P_{\infty}$
	then
	\begin{align} \label{bGfatPinfty}  G_{P_{\infty}}(p,q) = \log |z_{\infty}(p)| + \mathcal{O}(1) \quad \text{ as } p \to P_{\infty}, 
		\end{align}
	\item[\rm (d)] $G_{P_{\infty}}(p,q) = G_{P_{\infty}}(q,p)$.
\end{enumerate}
\end{definition}
The properties (a)-(d) define $G_{P_{\infty}}$ uniquely, up to an additive constant. The additive constant is not really important
for what follows. However, we make
a  definite choice in Lemma \ref{lemma41} below.

On a complex torus $\mathbb C \slash (\mathbb Z + \tau \mathbb Z)$ with $\Im \tau > 0$, the bipolar Green's function with pole at $0$ (modulo $\mathbb Z + \tau \mathbb Z$) is known explicitly
in terms of the Jacobi elliptic function $\theta_1$, namely
\begin{equation} \label{G0uv} G_0(u,v) = \log \left| \frac{\theta_1(u) \theta_1(v)}{\theta_1(u-v)} \right| 
	- \frac{2\pi}{\Im \tau} (\Im u) (\Im v), \end{equation}
see e.g.\ Skinner \cite{Ski15}.  
To verify \eqref{G0uv} one needs the well-known facts that
the Jacobi elliptic function $\theta_1= \theta_1( \cdot \mid \tau)$ is an entire function
with simple zeros at the lattice points, and no other
zeros, and it satisfies the quasi-periodicity properties
\begin{equation} \label{theta1period}
	\begin{aligned} 
		\theta_1(u + 1 \mid \tau) & = -\theta_1(u \mid \tau), \\
		\theta_1(u + \tau \mid \tau) &  
		= -e^{-\pi i \tau} e^{-2 \pi i u} \theta_1(u \mid \tau).
	\end{aligned} 
\end{equation}
See \cite[Appendix A]{KM17+}
for explicit formulas in the higher genus case.

The bipolar Green's energy 
of a measure $\mu$ in external field $\Re V$  is 
\begin{equation} \label{Greenenergy} 
	3 \iint G_{P_{\infty}}(p,q) d\mu(p) d\mu(q) + \int \Re V d\mu. \end{equation}
A probability measure that minimizes \eqref{Greenenergy} among all
probability measures on a given set $\Gamma$ is called the equilibrium
measure of $\Gamma$ in the external field $\Re V$. 
We also define 
\begin{equation} \label{Gmudef} 
	U^\mu(p) = 3\int G_{P_{\infty}}(p,q) d\mu(q) \end{equation}
which is analogous to the logarithmic potential of a measure in the complex plane.  
The equilibrium measure is characterized by the Euler-Lagrange variational
conditions. In the present context these conditions are that there 
is a constant $\ell$ such  that  $2 U^{\mu} + \Re V \geq \ell$ on $\Gamma$, with equality on the support of $\mu$.

\begin{theorem} \label{theorem17}
	Let $\mu$ be the probability measure from Theorem \ref{theorem14} and let $V$ be given by \eqref{Vdef}.
	Then there is a constant $\ell$ such that
	\begin{equation} \label{muELvarcon} 
		2 U^{\mu} + \Re V \begin{cases} = \ell,  & \text{ on } \Gamma_1 \cup \Gamma_2, \\
			> \ell, & \text{ on } \Gamma_3,
			\end{cases} \end{equation}
and $\mu$ is the equilibrium measure of $\Gamma = \Gamma_1 \cup \Gamma_2 \cup \Gamma_3$ in the external field $\Re V$.
In addition, the property 
\begin{equation} \label{GreenSprop} 
	\frac{\partial}{\partial n_+} \left(2 U^{\mu}  + \Re V \right) =
	\frac{\partial}{\partial n_-} \left(2 U^{\mu} + \Re V \right) 
	\quad \text{ on } \supp(\mu) = \Gamma_1 \cup \Gamma_2, \end{equation}
is satisfied. Here $\frac{\partial}{\partial n_+}$ and $\frac{\partial}{\partial n_-}$ denote the derivatives in the two normal directions to $\Gamma_1 \cup \Gamma_2$.
\end{theorem}
The conditions \eqref{muELvarcon} imply that the Euler-Lagrange variational
conditions associated with the minimization of \eqref{Greenenergy}
are satisfied. Therefore, $\mu$ is the equilibrium measure of $\Gamma = \Gamma_1 \cup \Gamma_2 \cup \Gamma_3$
in the external field $\Re V$.
The property \eqref{GreenSprop} is the $S$-property in the external field $\Re V$.

	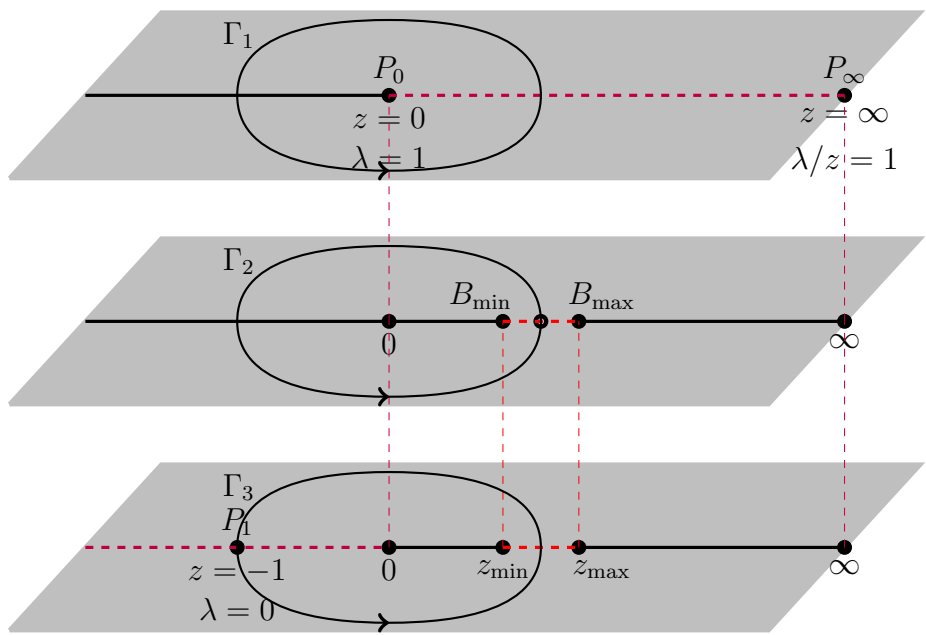
\begin{figure}[t]
	\begin{center}
		\begin{tikzpicture}(15,10)(0,0)
			\begin{scope}[ultra thick,decoration={
					markings,
					mark=at position 0.5 with {\arrow{>}}}
				] 
		% First sheet
			\filldraw[gray!50!white] (6,2.9) --++(10,0) --++(2,2.2) --++(-10,0) --++(-2,-2.2);
			% Second sheet
			\filldraw[gray!50!white] (6,-0.1) --++(10,0) --++(2,2.2) --++(-10,0) --++(-2,-2.2);
			% Third sheet
			\filldraw[gray!50!white] (6,-3.1) --++(10,0) --++(2,2.2) --++(-10,0) --++(-2,-2.2);
			
			% First sheet points
			%\filldraw (7.5,4)  circle (2pt);	 
			%\filldraw (10,4)  circle (2pt);	 
			\filldraw (11,4)  circle (2pt);	 
			\filldraw (17,4)  circle (2pt);	
			\draw   (11,4) node[above] {$P_0$};
			\draw  (11,4) node[below] {$z = 0$};
			\draw  (11,3.5) node[below] {$\lambda=1$};
			\draw   (17,4) node[above] {$P_\infty$};
			\draw   (17,4) node[below] {$z=\infty$}; 
			\draw   (17,3.5) node[below] {$\lambda/z =1$};
			% First sheet cuts
			\draw[very thick,black] (7,4)--++(4,0);
			% First sheet oval 
			\draw[dashed,very thick,purple] (11,4)--++(6,0);
			
			% Second sheet points
			%	\filldraw (7.5,1)  circle (2pt);	 
			%	\filldraw (10,1)  circle (2pt);	 
			\filldraw (11,1)  circle (2pt);	 
			%	\filldraw (11.5,1)  circle (2pt);	 
			\filldraw (12.5,1)  circle (2pt);	 
			\draw (13,1) circle (2pt);
			\filldraw (13.5,1)  circle (2pt);	 	 
			%	\filldraw (16.5,1) circle (2pt);
			\filldraw (17,1)  circle (2pt);	 
			\draw   (11,1) node[below] {$0$};
			\draw   (17,1) node[below] {$\infty$};
			\draw   (12.2,1) node[above] {$B_{\min}$};
		%	\draw   (12.5,1) node[below] {$z_{\min}$};
		%	\draw   (13,1) node[below] {$1$};
			\draw   (13.8,1) node[above] {$B_{\max}$};
		%	\draw   (13.8,1) node[below] {$z_{\max}$};
			% Second sheet cuts
			\draw[very thick,black] (7,1)--++(5.5,0);
			\draw[very thick,black] (13.5,1)--++(3.5,0);
			% Second sheet oval 
			%			\draw[very thick,purple] (10,1)--++(1.5,0);
			\draw[dashed, very thick,red] (12.5,1)--++(1,0);
			
			% Third sheet points	 	 
			\filldraw (9,-2) circle (2pt);
			\filldraw (11,-2)  circle (2pt); 
			\filldraw (12.5,-2)  circle (2pt);	 
			\filldraw (13.5,-2)  circle (2pt);	 
			\filldraw (17,-2)  circle (2pt);	 
			\draw   (11,-2) node[below] {$0$};
			\draw   (9,-2) node[above] {$P_1$};
			\draw  (9,-2.5) node[below] {$\lambda = 0$};
			\draw  (9,-2) node[below] {$z=-1$};
			\draw   (17,-2) node[below] {$\infty$};
			\draw   (12.5,-2) node[below] {$z_{\min}$};
			\draw   (13.8,-2) node[below] {$z_{\max}$};
			% Third sheet cuts
			\draw[very thick,black] (11,-2)--++(1.5,0);
			\draw[very thick,black] (13.5,-2)--++(3.5,0);
			% Third sheet oval 
			\draw[dashed,very thick,purple] (7,-2)--++(4,0);
			\draw[dashed,very thick,red] (12.5,-2)--++(1,0);
			
			% Vertical line between first and second sheet
			di			%\draw[dashed,help lines,purple] (7,1)--(7,4);
			\draw[dashed,help lines,purple] (11,1)--(11,4);	
			\draw[dashed,help lines,purple] (17,1)--(17,4);	
			
			% Vertical line between first and second sheet
			%	\draw[dashed,help lines,purple] (7,-2)--(7,1);
			\draw[dashed,help lines,purple] (11,-2)--(11,1);
			\draw[dashed,help lines,red] (12.5,-2)--(12.5,1);
			\draw[dashed,help lines,red] (13.5,-2)--(13.5,1);	
			\draw[dashed,help lines,purple] (17,-2)--(17,1);

			% \Gamma_1
			\coordinate (a) at (9,4);
			\coordinate (b) at (11,3);
			\coordinate (c) at (13,4);
			\coordinate (d) at (11,5);
			\coordinate (e) at (9,4);             
			\path[draw, thick, use Hobby shortcut,closed=true]  (a)..(b)..(c)..(d)..(e); 
			\draw [postaction={decorate}] (11,3) node{};		
				 
			\draw   (9.4,4.8) node[left] {$\Gamma_1$};
			
			% \Gamma_2
			\coordinate (a) at (9,1);
			\coordinate (b) at (11,0);
			\coordinate (c) at (13,1);
			\coordinate (d) at (11,2);
			\coordinate (e) at (9,1);               
			\path[draw, thick, use Hobby shortcut,closed=true]  (a)..(b)..(c)..(d)..(e); 
			\draw [postaction={decorate}] (11,0) node{};		
			
			\draw   (9.4,1.8) node[left] {$\Gamma_2$};
			% \Gamma_3
			\coordinate (a) at (9,-2);
			\coordinate (b) at (11,-3);
			\coordinate (c) at (13,-2);
			\coordinate (d) at (11,-1);
			\coordinate (e) at (9,-2);               
			\path[draw, thick, use Hobby shortcut,closed=true]  (a)..(b)..(c)..(d)..(e); 
			\draw [postaction={decorate}] (11,-3) node{};			
			\draw   (9.4,-1.2) node[left] {$\Gamma_3$};
			\end{scope}
		\end{tikzpicture}
	\end{center}
	\caption{The unit circles $\Gamma_1, \Gamma_2$ and $\Gamma_3$
		on the three sheets of the Riemann surface.
	 \label{fig:sheets3}}
\end{figure}

\begin{remark}
The results in Theorems \ref{theorem13} and \ref{theorem14}
are for the case $B=C=1$ in \eqref{PNortho}. It is natural to ask whether the approach of this paper can be extended to more general $B, C > 0$.

The Riemann surface $\mathcal R$ depends on $W$ but not on
the parameters $B,C$. For general $B,C >0$ the external field
\eqref{Vdef} has to be modified to
	\begin{equation} \label{VdefBC}  V = (1+C) \log z - (B+C) \log \lambda 
		\quad \text{ on } \mathcal R. \end{equation}
If one then uses 
 \begin{equation} \label{muBalBC} 
 	\mu =   \Bal\left( \tfrac{1+B}{2} \delta_{P_{\infty}} -
	\tfrac{B+C}{2} \delta_1 
	+ \tfrac{1+C}{2} \delta_{P_0} ; \Gamma_1 \cup \Gamma_2 \right) 
	\end{equation}
instead of \eqref{muBal}, then the equality
$2 U^{\mu} + \Re V = \ell$ on $\Gamma_1 \cup \Gamma_2$,
in \eqref{muELvarcon} is still satisfied, for an appropriate constant $\ell$. 
However, there are three additional properties
that \eqref{muBalBC} would have to satisfy (cf.\ Theorem \ref{theorem17}), namely
\begin{itemize}
	\item[(1)] the signed measure \eqref{muBalBC} has to be positive, and thus a probability measure,
	\item[(2)] the inequality in \eqref{muELvarcon} should hold, but not necessarily in the strict sense,
	i.e., we need $2 U^{\mu} + \Re V \geq \ell$ on $\Gamma_3$,
	\item[(3)] the $S$-property in the external field $\Re V$ \eqref{GreenSprop} should hold. 
\end{itemize}
If these three properties hold, then  \eqref{muBalBC} is the minimizer of
\eqref{Greenenergy} among probability measures on $\Gamma = \Gamma_1 \cup \Gamma_2 \cup \Gamma_3$, and
the $S$-property on the support of $\mu$ is satisfied.

The property (3) holds in case $B=C$ in \eqref{muBalBC}, 
because of the symmetries in the model. 
If $B \neq C$, then one cannot expect that the $S$-property in external field holds on $\Gamma = \Gamma_1 \cup \Gamma_2 \cup \Gamma_3$. A major issue would then be
to determine an appropriate deformation 
of $\Gamma$
whose equilibrium measure in external field satisfies the $S$-property \eqref{GreenSprop}.

In case $B=C$, then the property (1) holds for $C=1$ (as we show in section \ref{subsec321} below) and then also for $0 < C \leq 1$. However, it will fail if $C$ is big enough. In this case the 
equilibrium measure of $\Gamma$ in the external field $\Re V$ 
is most likely  supported on a subset of $\Gamma_1 \cup \Gamma_2$ (and not on the
full $\Gamma_1 \cup \Gamma_2$).

Again, in case $B=C$, the property (2) holds for $C=1$ (as we show in section \ref{subsec322}), and probably also for $C \geq 1$,
but it will fail for small $C > 0$. Indeed, in the extreme case
$C=0$, the equilibrium measure is the balayage
of $\frac{1}{2} \delta_{P_{\infty}} + \frac{1}{2} \delta_{P_0}$
onto $\Gamma_1 \cup \Gamma_2 \cup \Gamma_3$, and $\Gamma_3$
is in the support. For small $C > 0$ 
part of $\Gamma_3$ will be in the support of
the equilbrium measure. 
\end{remark}

\begin{remark}
	One may also wonder what can be done in case the
	parameters $a_{jk}$, $b_{jk}$ in the model do not
	satisfy the equations \eqref{Eq1ab}--\eqref{Eq3ab}.
	Then the Riemann surface has the more general structure
	shown in Figure~\ref{fig:sheets}. Generically, there
	 are three points at infinity, say $P_{\infty}^{(1)}$,
	 $P_{\infty}^{(2)}$, $P_{\infty}^{(3)}$,
	 and one would have to change the bipolar Green's energy
	 to
	 \begin{equation} \label{Greenenergy2} \sum_{j=1}^3
	 	\iint G_{P_{\infty}^{(j)}}(p,q) d\mu(p) d\mu(q)
	 	+ \int \Re V d\mu \end{equation}
	with $V = 2 \log z - 2 \log \lambda$ in case
	of the regular hexagon, or given by \eqref{VdefBC}
	in the more general case. One would also have
	to change the definition of $U^{\mu}$ in \eqref{Gmudef}
	accordingly. An equilibrium measure on $\Gamma$ in the external field is now a minimizer of \eqref{Greenenergy2}
	among probability measures on $\Gamma$, and
	the task would be to look for $\Gamma$, that
	is homotopic to the union of three unit circles on
	the three sheets, such that the $S$-property in
	the external field is satisfied.
	
	This generalization also explains why we chose to have the factor $3$ in \eqref{Greenenergy} and \eqref{Gmudef}.
\end{remark}

\subsection{Overview} 

The asymptotic results in Theorem \ref{theorem13} and Theorem \ref{theorem14}(c)  are obtained from a Deift-Zhou steepest descent analysis of
the Riemann-Hilbert problem (RH problem) that characterizes the MVOP $P_N$. The equilibrium measure $\mu$ from Theorems \ref{theorem14}
and \ref{theorem17} serves
as a main ingredient in the analysis.

The RH problem for MVOP first appeared in \cite{CM12,GIM11}, as a generalization
of the RH problem for orthogonal polynomials, due to Fokas, Its, and Kitaev~\cite{FIK92}. A notable feature
is that the matrix valued function in the RH problem 
has size $2r \times 2r$, in case the MVOP have size $r \times r$.
Thus in our case, the size is $6 \times 6$.

The steepest descent analysis of RH problems was introduced by Deift and Zhou in \cite{DZ93}. It was first applied to orthogonal polynomials with varying weights in \cite{BI99, DKMVZ99}, see also \cite{Dei99}, where the RH problem is of size $2 \times 2$.
%Steepest descent on larger size RH problems were done in connection with multiple orthogonal polynomials, see the surveys \cite{AK11, Kui10a,Kui10b}, and matrix valued orthogonality \cite{DKR23}.
The RH problem corresponding to the matrix valued orthogonality \eqref{PNortho} takes the form of RH problem \ref{rhpforY} below.
As usual in RH problems, if $\gamma$ is an oriented contour,
and $f$ is analytic in a neighborhood of $\gamma$, but not on $\gamma$ itself,
then we use $f_+$ ($f_-$) to denote limiting
values, if they exist, from the left (right), when the contour is traversed
according to its orientation.

\begin{rhproblem} \label{rhpforY} \
	\begin{description}
		\item[RHP-Y1] 
		$Y : \mathbb C \setminus \Sigma_Y \to \mathbb C^{6 \times 6}$ is analytic,
		where $\Sigma_Y = \gamma$ is a simple closed contour going once around
		$0$ in counterclockwise direction.
		\item[RHP-Y2]  On $\gamma$ we have a jump $Y_+ = Y_- J_Y$
		where 
		\begin{equation} \label{Yjump} 
			J_Y(z) = \begin{pmatrix} I_3 & \frac{W(z)^{(B+C)N}}{z^{(1+C)N}}  \\
				0_3 & I_3 \end{pmatrix}, \qquad z \in \gamma. 
		\end{equation}
		\item[RHP-Y3]  As $z \to \infty$, we have 
		\begin{equation} \label{Yasymp} 
			Y(z) = \left(I_6 + O(z^{-1})\right) \begin{pmatrix} z^N I_3 & 0_3 \\ 0_3 & z^{-N} I_3 \end{pmatrix} \quad \text{ as } z \to \infty. 
		\end{equation}
	\end{description}
\end{rhproblem}
The solution to the Riemann Hilbert problem contains the MVOP $P_N$
in its left upper $3\times 3$ block, i.e.,
\begin{equation} \label{PNinY}
	P_N(z) = \begin{pmatrix} I_3 & 0_3 \end{pmatrix}
	Y(z) \begin{pmatrix} I_3 \\ 0_3 \end{pmatrix}.
\end{equation}
The full solution takes the form
\[ Y(z) = \begin{pmatrix} P_N(z) & \ds \frac{1}{2\pi i} \oint_{\gamma}
	\frac{P_N(s) W(s)^{(B+C)N}}{s^{(1+C)N} (s-z)} ds \\
	Q_{N-1}(z) & \ds \frac{1}{2\pi i} \oint_{\gamma}
	\frac{Q_{N-1}(s) W(s)^{(B+C)N}}{s^{(1+C)N} (s-z)} ds \end{pmatrix},
	\quad z \in \mathbb C \setminus \gamma,
\]
where $Q_{N-1}$ is a certain matrix valued polynomial of degree $\leq N-1$, see \cite{GIM11} for details.
For later use, we remark that
\begin{equation} \label{Ydet} 
	\det Y(z) = 1,  \qquad \text{ for } z \in \mathbb C \setminus \gamma.
\end{equation} It is a consequence of the fact
that the jump matrix in \eqref{Yjump} has determinant one
for every $z \in \gamma$,
and $\det Y(z) \to 1$ as $z \to \infty$,
because of \eqref{Yasymp}, see e.g.\ \cite{Dei99}.

Throughout the rest of the paper we take $B=C=1$ and we choose for $\gamma$ the unit circle $\mathbb T$ in the complex plane. 
The steepest descent analysis of the
RH problem \ref{rhpforY} consists of four transformations
\[ Y \mapsto X \mapsto T \mapsto S \mapsto R. \]

The first transformation $Y \mapsto X$ is a preliminary transformation based
on the eigenvectors of $W$. This transformation was also
done in the paper \cite{DK21} on $2$ periodic
tilings of the Aztec diamond. However, the further analysis is different. The MVOP and the associated RH problem in \cite{DK21} are remarkably simple, as the steepest descent analysis leads to
an exact formula for the MVOP at finite $N$, see also \cite{KP24+}. This is not the case in the current situation.
The preliminary transformation $Y \mapsto X$ 
is not needed when dealing
with matrix valued 
orthogonality with respect to a positive definite matrix weight 
on $[-1,1]$, see \cite{DKR23}.

The second transformation $X \mapsto T$ relies on the equilibrium measure $\mu$ with
the properties that are listed in Theorems \ref{theorem14} and
\ref{theorem17} above. It gives rise to the three $g$-functions,
$g_1, g_2, g_3$ that will be defined in Definition \ref{definition64} and whose properties are studied in
Section~\ref{section61}.

The third transformation $T \mapsto S$ is the
opening of lenses around $\mathbb T$. Since the measure
$\mu$ has full support on $\Gamma_1 \cup \Gamma_2$ with
a positive real analytic density, we can open the lens
in such a way that its boundary is fully separated from
$\mathbb T$. This is similar to what happens for 
orthogonal polynomials on the unit circle with positive analytic
weight, see \cite{MMS06}. It implies that
there is no need for local parametrices. It also ultimately implies
that we find asymptotic formulas with exponential small error terms.

The global parametrix is constructed in Section \ref{section8}.
We give explicit formulas with Jacobi theta functions
associated with a double cover of the Riemann surface.
With the global parametrix we define the final 
transformation $S \mapsto R$ in section~\ref{section9}.
It leads to a matrix valued function $R$ that is 
exponentially close to the identity matrix as $N \to \infty$.

The asymptotic formulas from Theorem \ref{theorem13} 
and Theorem \ref{theorem14} are proved in the final section
\ref{section10} of the paper.

In the appendix we describe the connection with doubly periodic lozenge tilings of the hexagon. This connection is the main motivation to
study the particular MVOP from \eqref{PNortho}. The steepest descent analysis that we do in this paper  has implications for the lozenge tilings, but we turn to do this in a separate future paper. 

\section{Spectral curve}
 The main result of this section is Proposition \ref{prop24} below.
To prepare for this we first find an explicit equation for
the spectral curve under the Assumptions \ref{assump12}.

\subsection{Explicit equation}

\begin{lemma} \label{lemma21} 
The spectral curve \eqref{Wcharpoly} takes the form
\begin{equation} \label{Plamz} 
	P(z,\lambda) = 	(\lambda-z-1)^3 - 27(1+\beta) \lambda z = 0
\end{equation}
for a certain $\beta > 0$.
\end{lemma}
\begin{proof} 
	From \eqref{Tj3} and \eqref{Wz} it
	follows that $W(z) = zL + U$ with a lower triangular matrix $L$ and an upper triangular $U$. The parts (a) and (b)
	of Assumption \ref{assump12} imply that the diagonal
	entries of $L$ and $U$ are equal to $1$. Thus
	\begin{equation} \label{Wzstructure} W(z) = \begin{pmatrix} z+1 & w_{12} & w_{13} \\
			w_{21} z & z+1 & w_{23} \\
			w_{31} z & w_{32} z & z+1 \end{pmatrix}
	\end{equation}
	with certain positive constants $w_{jk}$ that depend on
	the $a_{jk}$ and $b_{jk}$. From this structure
	it follows that 
	\[ \det\left(\lambda I_3 - W(z)\right) = (\lambda -z-1)^3
		-Q \lambda z - R(z) \] 
	with a constant $Q$ and a polynomial $R(z)$ of
	degree $\leq 2$ in $z$. Assumption \ref{assump12} (c)
	tells us that  $\lambda = 0$ can only be an eigenvalue
	of $W(z)$ for $z=-1$. This means that
	$\det W(z) = (z+1)^3 + R(z)$ has $z=-1$ as its only zero. 
	Since $R$ has degree $\leq 2$, we conclude that $R(z) \equiv 0$ and therefore
	\[ P(z, \lambda ) = (\lambda -z-1)^3 - Q \lambda z \] 
	for some $Q$.

	Assisted by Maple we find that 
	$Q = \frac{(A+B+C)^3}{ABC}$ and
	$A = a_{11} a_{12} b_{21} b_{22}$, $B = a_{11} a_{23} b_{12} b_{21}$
	and $C = a_{22} a_{23} b_{11} b_{12}$.
	The inequality between the arithmetic and geometric means%\footnote{Thanks to Mateusz for this remark}
	$\frac{A + B + C}{3} \geq \sqrt[3]{ABC}$ then gives us that $Q \geq 27$. 
	Then we can write $Q$ in the form $27(1+\beta)$ with $\beta \geq 0$ 
	and thus obtain \eqref{Plamz}. 
	
	For $\beta=0$ the spectral curve \eqref{Plamz} has the  explicit solution 
	$\lambda = (1+z^{1/3})^3$ 	and the genus is zero. Thus, by Assumptions \ref{assump12} (d) we have $\beta > 0$. 
\end{proof}
	For the special parameters \eqref{Tjconcrete} we have
	$ \beta = \frac{(1 + \alpha_1 + \alpha_2)^3}{27 \alpha_1 \alpha_2} - 1$.
\begin{remark} \label{remark22}
	The discriminant of \eqref{Plamz} with respect to $\lambda$ is
	\begin{equation} \label{DiscP} 
		\Disc_{\lambda} P = -19683 (1+\beta)^2 z^2
		(z^2-(2+4\beta) z + 1) \end{equation}
	which 	gives the two branch points
	\begin{equation} \label{zminzmax}
		z_{\min} = 1+ 2\beta - 2 \sqrt{\beta(1+\beta)}, \qquad
		z_{\max} = 1+ 2\beta + 2 \sqrt{\beta(1+\beta)}.
	\end{equation}
	which for $\beta > 0$ indeed satisfy
	\[ 0 < z_{\min} < 1 < z_{\max} = z_{\min}^{-1} < \infty. \] 
	We use $B_{\min}$, $B_{\max}$ to denote the corresponding points on
	the Riemann surface, see also Figure \ref{fig:sheets2}. 
	
	The polynomial $\lambda \mapsto
	P(\lambda,z_{\min})$ thus has a double
	zero, which turns out to be
	$-1-\beta + \sqrt{\beta(1+\beta)}$
	and one simple zero, which is 
	$8 (1+\beta - \sqrt{\beta(1+\beta)})$.
	It is easy to see that the simple zero is larger
	in absolute value than the double zero. 
	We thus find
	\begin{align*} 
		\lambda_1(z_{\min}) & = 8 \left(1+\beta - \sqrt{\beta(1+\beta)}\right), \\
		\lambda_2(z_{\min}) = \lambda_3(z_{\min}) & =
		-1-\beta + \sqrt{\beta(1+\beta)}.
	\end{align*}
	Similarly,
	\begin{align*} 
		\lambda_1(z_{\max}) & = 8 \left(1+\beta + \sqrt{\beta(1+\beta)}\right), \\
		\lambda_2(z_{\max}) = \lambda_3(z_{\max}) & =
		-1-\beta - \sqrt{\beta(1+\beta)}.
	\end{align*}
	Thus we have the explicit $\lambda$-coordinates 
	\begin{equation} \label{B1B2}
		\begin{aligned}
		z(B_{\min}) = z_{\min}, \quad & 	\lambda(B_{\min}) = -1-\beta + \sqrt{\beta(1+\beta)}, \\
		z(B_{\max}) = z_{\max}, \quad & 	\lambda(B_{\max}) = -1-\beta - \sqrt{\beta(1+\beta)}.
		\end{aligned}
	\end{equation}
	We also independently verified that the
	branch points $B_{\min}$ and $B_{\max}$ connect
	the second and third sheets.
	
	As $\beta \to 0+$, the bounded oval shrinks to a node at $(z,\lambda) = (1, -1)$.
\end{remark}

Observe that \eqref{Plamz} has the symmetries
\begin{equation} \label{involutions} 
	(z,\lambda) \mapsto (-\lambda, -z), \quad (z,\lambda) \mapsto (\tfrac{1}{z}, \tfrac{\lambda}{z} ), 
\end{equation}
which are holomorphic involutions, in addition to the anti-holomorphic involution \begin{equation} \label{involutionanti} 
	(z,\lambda) \mapsto (\overline{z},\overline{\lambda}).
	\end{equation}

For future reference we note the following.

\begin{lemma} \label{lemma23}
	\begin{enumerate}
		\item[\rm (a)] 
		The functions $\lambda_1$, $\lambda_2$, $\lambda_3$ are
		 analytic in $\mathbb C \setminus \mathbb R$
		 with
		 \begin{equation} \label{lambdajstrict} 
		 	|\lambda_1(z)| > |\lambda_2(z)| > |\lambda_3(z)| > 0,
		 	\quad z \in \mathbb C \setminus \mathbb R. \end{equation}
	\item[\rm (b)] The functions have boundary values on the real line,
 such that
	$\Lambda = \diag(\lambda_1, \lambda_2, \lambda_3)$ satisfies
	\begin{equation} \label{Lambdajump} 
	 \Lambda_+ =  \begin{cases}
	 	\sigma_{12} \Lambda_- \sigma_{12}  & \text{on } (-\infty,0], \\
	 	\sigma_{23} \Lambda_- \sigma_{23}  & \text{on } [0,z_{\min}] \cup [z_{\max},\infty), \\
	 \Lambda_- & \text{on } [z_{\min}, z_{\max}]
	 	\end{cases} \end{equation}
 	where $\sigma_{12}$ and $\sigma_{23}$ are used to denote the permutation matrices
 	\begin{equation} \label{sigmadef} 
 		\sigma_{12} = \begin{pmatrix} 0 & 1 & 0 \\
 		1 & 0 & 0 \\
 		0 & 0 & 1 \end{pmatrix},
 		\quad \sigma_{23} = \begin{pmatrix} 1 & 0 & 0  \\
 			0 & 0 & 1 \\
 			0 & 1 & 0 \end{pmatrix}. \end{equation}
 	\item[\rm (c)] 
 	With the constant $c_{\lambda} = 3 (1+\beta)^{1/3} > 0$ we have 
 	(we use principal branches of the fractional powers)
\begin{align} \label{lambda1asymp}
	\lambda_1(z) & = z + c_{\lambda} z^{2/3} + \frac{c_{\lambda}^2}{3} z^{1/3} + \mathcal{O}(1), \\ \label{lambda2asymp}
	\lambda_2(z) & = z + c_{\lambda} \omega^{\pm 1} z^{2/3} + \frac{c_{\lambda}^2}{3} \omega^{\mp 1} z^{1/3} + \mathcal{O}(1), \\
	\label{lambda3asymp}
	\lambda_3(z) & = z + c_{\lambda} \omega^{\mp 1} z^{2/3} + 
	\frac{c_{\lambda}^2}{3} \omega^{\pm 1} z^{1/3} +  \mathcal{O}(1), 
\end{align}
as $z \to \infty$ with $\pm \Im z > 0$ and $\omega = e^{2\pi i/3}$.
\end{enumerate}
\end{lemma}
\begin{proof} Parts (a) and (b) are consequences of the fact that we are
	dealing with a Harnack curve with the sheet structure as in
	Figure \ref{fig:sheets2}.

	For part (c) we use the spectral curve equation \eqref{Plamz}.
	For $z \to \infty$, any of the three solutions behaves like
   $z + c_{1} z^{2/3} + c_{2} z^{1/3} + \mathcal{O}(1)$
	with $c_{1}^3 = 3c_1^2 c_2 = 27(1+\beta)$.
	Since $\lambda_1(z)$ is real for real $z > 0$, we have the
	expansion \eqref{lambda1asymp} with real coefficients.
	The other solutions have expansions with coefficients
	that are either $c_1 = c_{\lambda} \omega$ and $c_2 = \frac{c_{\lambda}^2}{3} \omega^{-1}$, or  $c_1 = c_{\lambda} \omega^{-1}$
	and $c_2 = \frac{c_{\lambda}^2}{3} \omega$.
	Taking into the  ordering \eqref{lambdajstrict} of absolute values,
	we find \eqref{lambda2asymp} and \eqref{lambda3asymp}.
\end{proof} 	

\subsection{Abel map and complex torus}

The Riemann surface $\mathcal R$ has genus one. We choose
a canonical homology basis $\{ \mathbf{a}, \mathbf{b} \}$ where
$\mathbf{a}$ is the unbounded oval, oriented from left to right,
and $\mathbf{b}$ is $\Gamma_1 \cup \Gamma_2$, with counterclockwise
orientation. 
There is a unique holomorphic differential $\omega$ on $\mathcal R$ with $\oint_{\mathbf{a}} \omega =1$.  Then
$\oint_{\mathbf{b}} \omega = \tau \in i \mathbb R^+$. We write
\[ L = \mathbb Z + \tau \mathbb Z. \] 
The Abel map
\begin{equation} \label{AbelAdef} 
	\mathcal A : \mathcal R \to \mathbb C \slash L
	: p \mapsto \frac{1}{2} + \int_{P_1}^p \omega 
	\quad \mod L
	\end{equation} 
is a conformal map from $\mathcal R$ to the complex torus 
$\mathbb C \slash L$. 
For convenience we choose $P_1 = (-1,0)$ as base point for the integration
in \eqref{AbelAdef}, and we add the shift by $\tfrac{1}{2}$ so that
$\mathcal A(P_1) \equiv \frac{1}{2}$ modulo the lattice $L$.

The following result is illustrated in Figure \ref{fig:Abelmap}.

\begin{figure}[t]
	\begin{center}
		\begin{tikzpicture}[scale=0.75](15,10)(0,0)
			% rectangle
			\filldraw[gray!50!white] (-6,0) --(6,0) --(6,6) --(-6,6) --(-6,0);
			%\draw[very thick, black] (-6,0) --(6,0) --(6,6) --(-6,6) --(-6,0);
			
			\draw (-6,0) node[below] {$0$};
			\draw (6,0) node[below] {$1$};
			\draw (-6,6) node[left] {$\tau$};
			\draw (6,6) node[right] {$1+\tau$};
			
			% images of special points
			\filldraw (0,0)  circle (2pt);	 
			\filldraw (4,0)  circle (2pt);	 
			\filldraw (-4,0)  circle (2pt);	 	 
			\draw (0,0) node[below] {$\mathcal A(P_1) \equiv 1/2$};
			\draw (4,0) node[below] {$\mathcal A(P_0) \equiv 5/6$};
			\draw (-4,0) node[below] {$\mathcal A(P_{\infty}) \equiv 1/6$};
			
			\filldraw (0,6)  circle (2pt);	 
			\filldraw (4,6)  circle (2pt);	 
			\filldraw (-4,6)  circle (2pt);	 	 
			%		\draw (0,6) node[above] {$\mathcal A(P_1) \equiv 1/2 + \tau$};
			%		\draw (4,6) node[above] {$\mathcal A(P_0) \equiv 5/6 + \tau$};
			%		\draw (-4,6) node[above] {$\mathcal A(P_{\infty}) \equiv 1/6 + \tau$};
			\draw (0,6) node[above] {$\mathcal A(P_1)$};
			\draw (4,6) node[above] {$\mathcal A(P_0)$};
			\draw (-4,6) node[above] {$\mathcal A(P_{\infty})$};
			
			% image of circles
			\draw[dashed,black] (0,0)--(0,6);
			\draw[dashed,black] (-6,0)--(-6,6);
			\draw[dashed,black] (6,0)--(6,6);
			\draw (-6,3) node[left] {$\mathcal A(\Gamma_1 \cup \Gamma_2)$};
			\draw (0,3) node[right] {$\mathcal A(\Gamma_3)$};
			\draw (6,3) node[right] {$\mathcal A(\Gamma_1 \cup \Gamma_2)$};
		\end{tikzpicture}
	\end{center}
	\caption{Complex torus $\mathbb C \slash (\mathbb Z + \tau \mathbb Z)$ and the image of the Riemann surface
		under  the Abel map \eqref{AbelAdef}. \label{fig:Abelmap}}
\end{figure}
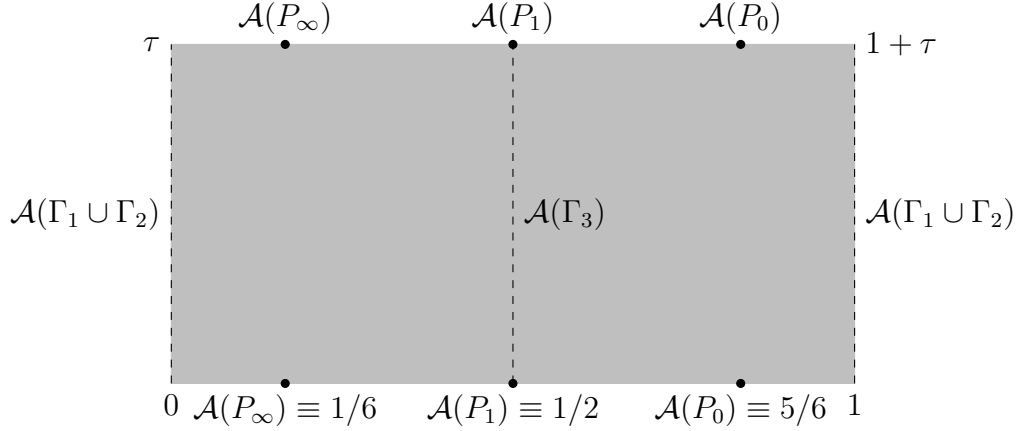

\begin{proposition} \label{prop24}
	\begin{enumerate}
		\item[\rm (a)] 
	The unbounded oval (i.e., the $\bf a$-cycle) is mapped to $[0,1]$ modulo $L$ with 
	\begin{equation} \label{Abelvalues} 
		\mathcal A(P_{\infty}) \equiv \frac{1}{6}, \quad 
		\mathcal A(P_1) \equiv \frac{1}{2}, \quad
	\mathcal A(P_0) \equiv \frac{5}{6}.
	\end{equation}
	
		\item[\rm (b)] The union $\Gamma_1 \cup \Gamma_2$ of
		the two unit circles on the first two sheets (i.e., the $\bf b$-
		cycle) is mapped by the Abel map to $[0,\tau]$ modulo $L$,
		and the unit circle $\Gamma_3$ on the third sheet is
		mapped to $\frac{1}{2} + [0,\tau]$ modulo $L$.
		
		\item[\rm (c)] For the point $B_0 = (1, \lambda_1(1))$ we
		have $\mathcal A(B_0) \equiv 0$ modulo $L$. 
	\end{enumerate}
\end{proposition}
\begin{proof}	
(a)
Since $\mathbf{a}$ is fixed under the anti-holomorphic involution \eqref{involutionanti}
the holomorphic differential $\omega$ will be of the form $\omega = v dz$
with a meromorphic $v$ that is real and positive on $\mathbf{a}$. Then
it follows from \eqref{AbelAdef} that $\mathcal A$
maps $\mathbf{a}$ to $[0,1]$ (modulo $L$) with $\mathcal A(P_1) \equiv \frac{1}{2}$.

$p = (z,\lambda) \mapsto z$ is a meromorphic function with
a triple zero at $P_0$ and a triple pole at $P_{\infty}$,
and no other zeros and poles, i.e., $\divis(z) = 3 P_0 - 3  P_{\infty}$.
Similarly $\divis(\lambda) = 3 P_1 - 3 P_{\infty}$.
 Abel's theorem then implies that $3 \mathcal A(P_0) \equiv 3 \mathcal A(P_{\infty})
\equiv 3 \mathcal A(P_1)$. We note that $\mathcal A(P_1), \mathcal A(P_0), \mathcal A(P_{\infty})$
are real (modulo $L$) and distinct with $\mathcal A(P_1) \equiv \frac{1}{2}$.
This leads to the values \eqref{Abelvalues}, since
$P_1, P_0, P_{\infty}$ are in that order on the $\mathbf{a}$ cycle.

\medskip
(b) The second holomorphic involution in  \eqref{involutions}  maps
$P_1$ to $P_1$, $P_0$ to $P_{\infty}$, and $P_{\infty}$ to $P_0$. 
On the complex torus it thus corresponds to $u \mapsto -u$ (modulo $L$).
The anti-holomorphic involution \eqref{involutionanti} corresponds
to $u \mapsto \overline{u}$ (modulo $L$).
The combination 
\begin{equation} \label{involutionanti2} 
	(z,\lambda) \mapsto \left(\tfrac{1}{\overline{z}}, \tfrac{\overline{\lambda}}{\overline{z}}\right) \end{equation}
 then corresponds
to $u \mapsto -\overline{u}$ (modulo $L$).
This involution has the two closed contours $[0,\tau]$ and $\frac{1}{2} + [0,\tau]$ on $\mathbb C \slash L$ as fixed points, 
while \eqref{involutionanti2} has every point on $\mathcal R$
with $|z| = 1$ as fixed point.
It follows that $\Gamma_1 \cup \Gamma_2$ and $\Gamma_3$
(which are both closed contours on $\mathcal R$) are
 mapped by the Abel map to $[0,\tau]$ and  $\frac{1}{2} + [0,\tau]$. 
 Since $P_1$ is on $\Gamma_3$ and $\mathcal A(P_1) \equiv \frac{1}{2}$,
 we conclude that $\Gamma_3$ is mapped to $\frac{1}{2} + [0,\tau]$.
 Then $\Gamma_1 \cup \Gamma_2$ is mapped to $[0,\tau]$.

\medskip
(c) This follows from parts (a) and (b), as the $\mathbf{a}$ and $\mathbf{b}$ cycles intersect at $B_0$.
\end{proof}

\section{Proof of Theorem \ref{theorem14} (a)} \label{section3}

\begin{proof}
	Let $\mu$ be given by \eqref{muBal}. We have to show that $\mu$ is a
	positive measure with a real analytic and positive density. 
	Because of Proposition \ref{prop24}
	and conformal invariance of balayage measures we have that the
	push-forward of $\mu$ by the Abel map is given by
	\begin{equation} \label{Amu} 
		\mathcal A_*(\mu) =  \Bal\left(\delta_{1/6} - \delta_{1/2} + \delta_{5/6}; \mathcal{A}(\Gamma_1 \cup \Gamma_2)\right) \end{equation}
	with balayage inside $\mathbb C \slash L$, and we have to show
	that \eqref{Amu} has a positive real analytic density.
	
	Let $h$ be a continuous function on $\mathcal R$ that is harmonic
	on $\mathcal R \setminus (\Gamma_1 \cup \Gamma_2)$. Then
	$h \circ \mathcal A^{-1}$ is continuous on 
	 $\mathbb C \slash L$ and 
	harmonic on $(\mathbb C \slash L) \setminus \mathcal A(\Gamma_1 \cup \Gamma_2)$.
	Consider $h \circ \mathcal A^{-1}$ as defined on $\{ u \in \mathbb C \mid 0 \leq \Re u \leq 1,
	0 \leq \Im u \leq \Im \tau \}$.
	Let 
	\[ \widetilde{\Omega} = \{ u \in \mathbb C \mid 0 < \Re u < 1 \} \]
	be the unbounded vertical strip. 
	Then $h \circ \mathcal A^{-1}$ uniquely extends to a continuous function $\widetilde{h}$ on
	the closure of $\widetilde{\Omega}$ with the periodicity
	property $\widetilde{h}(u + \tau) = \widetilde{h}(u)$ for $u \in \widetilde{\Omega}$. 
	Then $\widetilde{h}$ is harmonic in $\widetilde{\Omega}$
	and
	\begin{equation} \label{htildeint} \int h d \mu   
	= \int \widetilde{h} d\widetilde{\mu} \end{equation}
	where
	\begin{equation} \label{mutilde} \widetilde{\mu} = 
	\Bal\left(\delta_{1/6} - \delta_{1/2} + \delta_{5/6}; \partial \widetilde{\Omega} \right). \end{equation}
	Because of \eqref{htildeint} and \eqref{mutilde} 
	it suffices to show that \eqref{mutilde} is a positive measure.
	
	We do this by mapping $\widetilde{\Omega}$ to the upper half plane $\mathbb C^+$
	by means of the conformal map $u \mapsto e^{\pi i u}$. 
	By conformal invariance, we have to show that
	\begin{equation} \label{muBalonR}
		 \Bal\left(\delta_{e^{\pi i/6}} - \delta_{e^{\pi i/2}} + \delta_{e^{5 \pi i/6}} ; \mathbb R\right) 
	\end{equation}	is a positive measure on the real line,
	 with a positive real analytic density.
	This can be done by the following explicit calculation.
	For  $z = a + ib \in \mathbb C^+$
	the balayage of $\delta_z$ onto the real line, is known to
	be the Cauchy measure $\frac{1}{\pi} \frac{b}{b^2 + (x-a)^2} dx$.
	Thus 
	\begin{align*} 
		\Bal \left( \delta_{e^{\pi i/6}}; \mathbb R \right)
			& = \frac{dx}{2 \pi \left(1+x^2- \sqrt{3}x\right)}, \\
		\Bal\left( \delta_{e^{\pi i/2}}; \mathbb R\right) & = \frac{dx}{\pi (1+x^2)}, \\
			\Bal \left( \delta_{e^{5 \pi i/6}}; \mathbb R \right)
		& = \frac{dx}{2 \pi \left(1+x^2 + \sqrt{3}x\right)}, 
		\end{align*} 
	and from these formulas it is straightforward to verify that 	
	\eqref{muBalonR} is indeed a positive measure with a density 
	that is real analytic and positive on $\mathbb R$.
	This completes the proof of Theorem \ref{theorem14} (a).
\end{proof}

\section{Proof of Theorem \ref{theorem17}}

\subsection{Preliminaries on the bipolar Green's kernel}

The bipolar Green's kernel $G_{P_{\infty}}$
has the properties listed in Definition \ref{definition16}
and it is determined up to an additive constant.
In the proof of Theorem \ref{theorem17} it will be convenient
to choose the constant such that the following holds.

\begin{lemma} \label{lemma41}
	There is a choice of additive constant for the bipolar
	Green's kernel such that
	\begin{align} \label{bGfatPinfty2}  G_{P_{\infty}}(p,q) = - \frac{1}{3} \log |z(p)| + o(1) \quad \text{ as } p \to P_{\infty}. 
	\end{align}
\end{lemma}
\begin{proof} Let $\widehat{G}_{P_{\infty}}(p,q)$ be
	any bipolar Green's kernel.
	From \eqref{bGfatPinfty} it follows that for each $q \in \mathcal R \setminus \{ P_{\infty} \}$ there is a constant $c(q)$ such that
	\[ \widehat{G}_{P_{\infty}}(p,q) = - \frac{1}{3} \log |z(p)| + c(q) + o(1),
	\quad \text{ as } p \to P_{\infty}, \]
	since $z(p)^{-1/3}$ is a local coordinate at $P_{\infty}$.
	Thus
	\[ c(q) = \lim_{p \to P_{\infty}}
	\left[ \widehat{G}_{P_{\infty}}(p,q) + \frac{1}{3} \log |z(p)| \right]. 
	\]
	Since $q \to \widehat{G}_{P_{\infty}} (p,q)$ is harmonic
	in $\mathcal R \setminus \{ P_{\infty}, p\}$
	(due to items (a) and (d) in Definition \ref{definition16})
	it follows that $q \to c(q)$ is harmonic
	in $\mathcal R \setminus \{P_{\infty}\}$. 
	
	For $p,q$  in the local coordinate chart around $P_{\infty}$
	we also have that
	\[ H(p,q) = \widehat{G}_{P_{\infty}}(p,q) - \log |z(p)^{-1/3}-z(q)^{-1/3}| \]
	remains bounded, 	as all singularities cancel out.
	Since $c(q) = H(P_{\infty}, q)$, we then have that
	$c(q)$ remains bounded as $q \to P_{\infty}$. Therefore
	$q \mapsto c(q)$ extends to a harmonic function on $\mathcal R$.
	The only harmonic functions on $\mathcal R$ are constants,
	since $\mathcal R$ is compact, hence $c(q) = c_0$ for some $c_0$. and $\widehat{G}_{P_{\infty}}(p,q) - c_0$ is a bipolar Green's kernel
	with the additional property \eqref{bGfatPinfty2}.
\end{proof}

\subsection{Proof of equality on $\Gamma_1 \cup \Gamma_2$}
\label{subsec321}

We first prove the equality in \eqref{muELvarcon} on $\Gamma_1 \cup \Gamma_2$.

\begin{proof}
In the proof, we use the shorter notation 
\[ \widehat{\delta}_{a} = \Bal(\delta_a; \Gamma_1 \cup \Gamma_2) \] 
to denote the balayage of the Dirac point mass at $a \in \mathcal R$ onto $\Gamma_1 \cup \Gamma_2$.

Let $\Omega = \mathcal R \setminus (\Gamma_1 \cup \Gamma_2)$ and
let $G_{\Omega}$ be the usual Green's kernel for $\Omega$. That is, for
each $p \in \Omega$, $q \mapsto G_{\Omega}(p,q)$ is harmonic in
$\Omega \setminus \{p\}$, continuous on $\mathcal R \setminus \{p\}$,
zero on $\partial \Omega$, and (in any local coordinate)
\begin{equation} \label{GreenOmega} 
	G_{\Omega}(p,q) = - \log |p-q| + \mathcal{O}(1), \text{ as } q \to p. 
\end{equation}
Then, by the properties \eqref{bGfatq}, \eqref{bGfatPinfty},
and \eqref{GreenOmega}, one finds 
that for any given $p \in \Omega \setminus \{P_{\infty}\}$,
the function
\[ q \mapsto G_{P_{\infty}}(p,q) - G_{\Omega}(p,q) + G_{\Omega}(P_{\infty}, q) \]
is continuous on $\mathcal R$ and harmonic on $\Omega$.
Hence by the property \eqref{deltaBal}	 of balayage measure
\begin{multline*} \int \left(G_{P_{\infty}}(p,q) - G_{\Omega}(p,q) + G_{\Omega}(P_\infty, q) \right)  d \widehat{\delta}_{a}(q) \\
= G_{P_{\infty}}(p,a) - G_{\Omega}(p,a) + G_{\Omega}(P_\infty, a), \end{multline*}
for any $a \in \mathcal R \setminus ( \Gamma_1 \cup \Gamma_2 \cup \{P_\infty, p\})$.
Since $G_{\Omega}(p,q) = 0$ for $q \in \Gamma_1 \cup \Gamma_2$
and $\widehat{\delta}_{a}$ is supported on $\Gamma_1 \cup \Gamma_2$,
we find
\[ \int G_{P_{\infty}}(p,q) d \widehat{\delta}_{a}(q)  = G_{P_\infty}(p,a) - G_{\Omega}(p,a) + G_{\Omega}(P_\infty, a).  \] 
For $p \in \partial \Omega
= \Gamma_1 \cup \Gamma_2$, this
reduces to
\begin{equation} \label{Gbalhulp} \int G_{P_{\infty}}(p,q) d \widehat{\delta}_{a}(q)  = G_{P_{\infty}}(p,a)  + G_{\Omega}(P_\infty, a),
	\quad \text{ for } p \in \Gamma_1 \cup \Gamma_2.  \end{equation}

The function $\lambda$ is a meromorphic function on $\mathcal R$
with divisor $\divis(\lambda) = 3 P_1 - 3 P_{\infty}$, i.e., there is a triple zero
at $P_1$, a triple pole at $P_{\infty}$ and no other poles or zeros. 
Then $p \mapsto \log | \lambda(p)|$ is harmonic in $\mathcal R \setminus \{P_{\infty}, P_1\}$ with
$\log |\lambda(p)| = 3 \log |z(p)+1| + \mathcal{O}(1)$ as $p \to P_1$
and $\log |\lambda(p)| = - 3 \log |z(p)^{-1/3}| + o(1)$
as $p \to P_{\infty}$, (see also \eqref{lambda1asymp}, \eqref{lambda2asymp}, \eqref{lambda3asymp}), 
where $z(p)+1$ is a local coordinate at $P_1$
and $z(p)^{-1/3}$ is a local coordinate at $P_{\infty}$.
Then it is easy to see that
\[
	G_{P_{\infty}}(p, P_1) = - \frac{1}{3} \log |\lambda(p)|, \quad p \in \mathcal R,
\]	
Taking $a = P_1$ in \eqref{Gbalhulp} we find
\begin{equation} \label{GbalP1}
	\int G_{P_{\infty}}(p,q) d \widehat{\delta}_{P_1}(q)
	 = - \frac{1}{3} \log |\lambda(p)| 
	+ G_{\Omega}(P_{\infty}, P_1), 
	\quad \text{ for } p \in \Gamma_1 \cup \Gamma_2. 
\end{equation}
For $P_0$, we similarly have $G_{P_{\infty}}(p, P_0) = - \frac{1}{3} \log |z(p)|$, and 
\begin{equation} \label{GbalP0}
	\int G_{P_{\infty}}(p,q) d \widehat{\delta}_{P_0}(q)= - \frac{1}{3} \log |z(p)| 
	+ G_{\Omega}(P_{\infty}, P_0), 
	\quad \text{ for } p \in \Gamma_1 \cup \Gamma_2. 
\end{equation}

Next, we want to let $a \to P_{\infty}$ in \eqref{Gbalhulp}.
The logarithmic singularities that the two terms in the
right-hand side have at $P_{\infty}$ cancel out, and the limit exists,
but the limit may depend on $p$.  
To show that the limit does not depend on $p$,
we choose a reference point, say $B_0 \in \Gamma_1 \cup \Gamma_2$,
and we observe that
\[ a \mapsto G_{P_{\infty}}(p,a) - G_{P_{\infty}}(B_0,a) \]
is harmonic on $\mathcal R \setminus \{p,B_0\}$ with logarithmic
singularities at $p$ and $B_0$. It is equal to $G_{B_0}(p,a)$ up to an additive
constant that may depend on $p$ and $B_0$. We write the constant
as $-G_{B_0}(P_\infty,p) + \ell(p,B_0)$ to obtain
\begin{equation} \label{GbalPinf1} 
	G_{P_{\infty}}(p,a) - G_{P_{\infty}}(B_0,a)
	= G_{B_0}(p,a) - G_{B_0}(P_\infty,p) + \ell(p,B_0).  \end{equation}
The function 
\[ p \mapsto G_{P_{\infty}}(p,a)  - G_{B_0}(p,a) + G_{B_0}(P_{\infty},p) \]
is harmonic in $\mathcal R \setminus \{B_0,P_{\infty},a\}$, but it
is easy to check that all logarithmic singularities cancel out. Hence
it is a constant, independent of $p$. Using this in \eqref{GbalPinf1}, we see
that the constant $\ell(p,B_0)$ does not depend on $p$, and we may write
\begin{equation} \label{GbalPinf2} 
	G_{P_{\infty}}(p,a) - G_{P_{\infty}}(B_0,a)
	= G_{B_0}(p,a) - G_{B_0}(P_\infty,p) + \ell(B_0),  \end{equation}
with a constant $\ell(B_0)$ that may depend on $B_0$, but is independent
of $a$ and $p$.
Thus in view of \eqref{Gbalhulp} we have
\[ 	\int G_{P_{\infty}}(p,q) d \widehat{\delta}_{a}(q)
	= 	\int G_{P_{\infty}}(B_0,q) d \widehat{\delta}_{a}(q)
		+ G_{B_0}(p,a) - G_{B_0}(P_\infty,p) + \ell(B_0)
	 \]
for $p \in \Gamma_1 \cup \Gamma_2$. In the limit $a \to P_{\infty}$ we get
\begin{equation} \label{GbalPinfty}
	\int G_{P_{\infty}}(p,q) d \widehat{\delta}_{P_\infty}(q)
	= \ell_{\infty}, \quad \text{ for } p \in \Gamma_1 \cup \Gamma_2. 
\end{equation}
with constant $\ell_{\infty} = 	\int G_{P_{\infty}}(B_0,q) d \widehat{\delta}_{P_\infty}(q)  + \ell(B_0)$
that does not depend on $p$.

Combining \eqref{GbalPinfty}, \eqref{GbalP1}, \eqref{GbalP0}, 
and recalling the definition \eqref{Gmudef} we obtain
\begin{align*} U^{\mu}(p) & =  3 \int G_{P_{\infty}}(p,q) d \left( \widehat{\delta}_{P_{\infty}} - \widehat{\delta}_{P_1} + \widehat{\delta}_{P_0} \right)(q) \\
 	& =  \log |\lambda(p)| - \log |z(p)| + \ell/2,
 	\quad \text{ for } p \in \Gamma_1 \cup \Gamma_2,  \end{align*}
 with yet another constant $\ell = 6 \ell_\infty 
 -6 G_{\Omega}(P_{\infty}, P_1) + 6G_{\Omega}(P_{\infty}, P_0)$.
 This gives the equality \eqref{muELvarcon} on $\Gamma_1 \cup \Gamma_2$ since $\Re V = 2(\log |z| - \log|\lambda|)$ by \eqref{Vdef}.
\end{proof}

\subsection{Proof of strict inequality on $\Gamma_3$}
\label{subsec322}
Strict inequality in \eqref{muELvarcon} holds on $\Gamma_3$,
as follows from the following stronger statement.

\begin{lemma} \label{lemma42}
	Let 
	\begin{equation} \label{hdef} 
		h =  2 U^{\mu} + \Re V - \ell. \end{equation} 
	For a given $p \in \mathcal R$, we let $u \equiv \mathcal A(p)$
	with $u \in [0,1] \times [0,\tau)$. Then the following hold
		\begin{enumerate}
		\item[\rm (a)] $h(p) = 0$ if and only if $\Re u
			\in \{0, \frac{1}{3}, \frac{2}{3}\}$,
	\item[\rm (b)] 	$h(p) < 0$ if and only if $0 < \Re u < \frac{1}{3}$
	or $\frac{2}{3} < \Re u < 1$,
	\item[\rm (c)] $h(p) > 0$ if and only if $\frac{1}{3} < \Re u < \frac{2}{3}$.
	\end{enumerate}
\end{lemma}
\begin{proof}
  The function \eqref{hdef} is harmonic on $\mathcal R \setminus ( \Gamma_1 \cup \Gamma_2 \cup 
 \{P_0, P_1, P_{\infty} \})$. 
 We transfer it to the complex torus by means of the Abel map, and by  slight abuse
 of notation, we call that function $h$ as well. 
 Then $h$ is zero on $[0,\tau]$ modulo $L$,  
 by the equality part of \eqref{muELvarcon} in
 Theorem \ref{theorem17} that we just proved. 
 Also $h$
 is harmonic on $(\mathbb C \slash L ) \setminus [0,\tau]$
 except at $u= \frac{1}{6}, \frac{1}{2}, \frac{5}{6}$ (modulo $L$), 
 as these are the images of $P_{\infty}$,
 $P_1$ and $P_0$, see Figure \ref{fig:Abelmap}. At these special
 points there are logarithmic singularities
 \begin{align} \label{hlogsing}
 	h(u)  = \begin{cases} 6 \log |u-\frac{1}{6}| + O(1),  & \text{as } u \to \frac{1}{6}, \\
 	- 6 \log|u-\frac{1}{2}| + O(1), &   \text{as } u \to \frac{1}{2}, \\
 	6\log|u-\frac{5}{6}| + O(1), &   \text{as } u \to \frac{5}{6}. \end{cases}
 \end{align}
 The behaviors at $\frac{1}{2}$ and $\frac{5}{6}$ comes from the term $\Re V = 2(\log |z| - \log |\lambda|)$ in \eqref{hdef}. For example, $z^{1/3}$ and $u-\frac{5}{6}$
 are both local coordinates  at $z=0$ on $\mathcal R$, and $z^{1/3} \approx C(u-\frac{5}{6})$ for some non-zero constant $C$ as $u \to \frac{5}{6}$
 and this leads to the leading order behavior as $u \to \frac{5}{6}$
 as given in \eqref{hlogsing}. The behavior at $\frac{1}{6}$ comes from the term
 $-2U^{\mu}$, as this is $6 \int G_{P_{\infty}}(p,q) d\mu(q)$ by \eqref{Gmudef} and
 each $G_{P_{\infty}}(p,q)$ with $q \in \Gamma_1 \cup \Gamma_2$ 
 behaves like $\log |u-\frac{1}{6}|$ as $p \to P_{\infty}$ (i.e., $u \to \frac{1}{6}$) by the defining property 
 \eqref{bGfatPinfty} of the bipolar Green's function, while $\mu$ is a probability measure.
    
 The above properties characterize $h$. 
 
 Now we introduce two auxiliary harmonic functions.  
 
 Take $\Omega_1 = \{ u \in \mathbb C \slash L : 0 <  \Re u < \frac{2}{3} \}$,
 and let $h_1$ be harmonic on $\Omega_1$, except 
 for logaritmic singularities at $u=\frac{1}{6}$ and $u= \frac{1}{2}$
 with the same behaviors as in \eqref{hlogsing}
 \begin{align} \label{h1logsing}
 	h_1(u) = \begin{cases} 6\log |u- \frac{1}{6}| + O(1), & \text{as } u \to \frac{1}{6}, \\
  		- 6\log|u-\frac{1}{2}| + O(1), &   \text{as } u \to \frac{1}{2}.
  		\end{cases}
 \end{align}
 Also $h_1$ extends continuously to $\partial \Omega_1$ with $h_1(u) = 0$
 for $u \in \partial \Omega_1$. Such $h_1$ uniquely exists, and it is
 in fact the following combination of Green's functions
 \[ h_1(u) = 6 G_{\Omega_1}(u,\tfrac{1}{2}) - 6 G_{\Omega_1}(u, \tfrac{1}{6}). \]
  Because of symmetry, the function $u \mapsto -h_1(\frac{2}{3}-\overline{u})$ satisfies the
  same properties as $h_1$, and thus $h_1(u) = -h_1(\frac{2}{3} -\overline{u})$
  by uniqueness. In particular $h_1(u) = 0$ for $\Re u = \frac{1}{3}$.
  
 Similarly, we take $\Omega_2 = \{ u \in \mathbb C \slash L \mid \frac{1}{3} < \Re u < 1\}$ and we let $h_2$ be harmonic on $\Omega_2$, except for logaritmic singularities at $u= \frac{1}{2}$, $u= \frac{5}{6}$ where we have 
 \begin{align} \label{h2logsing}
 	h_2(u)  = \begin{cases} -6\log |u-\frac{1}{2}| + O(1),  & \text{as } u \to \frac{1}{2}, \\
 		6\log|u-\frac{5}{6}| + O(1), &   \text{as } u \to \frac{5}{6}.
 	\end{cases}
 \end{align}
 We also want $h_2(u) = 0$ for $u \in \partial \Omega_2$.
 Then as above we have,  
  \[ h_2(u) = 6 G_{\Omega_2}(u,\tfrac{1}{2}) - 6 G_{\Omega_2}(u, \tfrac{5}{6}), \]
 and $h_2(u) = 0$ for $\Re u = \frac{2}{3}$.
 
 Then both $h_1$ and $h_2$ are zero on $\partial (\Omega_1 \cap \Omega_2)$
 and  harmonic on $\Omega_1 \cap \Omega_2$,
 except at $\frac{1}{2}$ with a logarithmic singularity as in \eqref{h1logsing}
 and \eqref{h2logsing}. Thus 
 \[ h_1(u) = h_2(u) = G_{\Omega_1 \cap \Omega_2}(u,\tfrac{1}{2}),
 	\quad \text{for } u \in \Omega_1 \cap \Omega_2. \]
 Then it is easy to see that	
 \begin{equation} \label{h1h2} 
 	u \mapsto \begin{cases} h_1(u), &  0 \leq \Re u < \frac{2}{3}, \\
 	h_2(u), & \frac{1}{3} < \Re u \leq 1,
 \end{cases} \end{equation}
 is well-defined and harmonic on $(\mathbb C \slash L)\setminus 
 ([0,\tau] \cup \{ \frac{1}{6}, \frac{1}{2}, \frac{5}{6}\})$, 
 with logarithmic singularities as in \eqref{h1logsing}
 and \eqref{h2logsing}, and zero boundary conditions.
 Thus it satisfies all the requirements for $h$.
 By uniqueness we conclude that $h$ is given by \eqref{h1h2}. 
 
 Thus $h = 0$ for $\Re u = \frac{1}{3}$ and $\Re u = \frac{2}{3}$
 and it is harmonic for $\frac{1}{3} < \Re u < \frac{2}{3}$,
except for one singularity at $\frac{1}{2}$, with the behavior as in
\eqref{hlogsing}, so that $h$ tends to $+\infty$ as $u \to \frac{1}{2}$.
Then by the minimum principle for harmonic functions, we have that $h(u) > 0$
for $\frac{1}{3} < \Re u < \frac{2}{3}$.

In a similar way it follows that $h(u) < 0$ for $0 < \Re u < \frac{1}{3}$
and for $\frac{2}{3} < \Re u < 1$, and the lemma follows.  
 \end{proof}

 It is now easy to establish the strict inequality in \eqref{muELvarcon} on 
 $\Gamma_3$.
 \begin{proof}[Proof of strict inequality on $\Gamma_3$]
 	If $p \in \Gamma_3$, then $u = \mathcal A(p)$ is such that $\Re u = \frac{1}{2}$  by part (b) of Proposition \ref{prop24}.  
 	Part (c) of Lemma \ref{lemma42}
 	then tells us that $h(p) > 0$. In view of the definition \eqref{hdef} 
 	of $h$, we find the strict inequality in \eqref{muELvarcon}.
 \end{proof}
 
 \subsection{Proof of the $S$-property}
 
 The proof of the identity \eqref{GreenSprop} is straightforward. 

\begin{proof}[Proof of \eqref{GreenSprop}]
We consider again $h$ as in \eqref{hdef}, and as in the proof of Lemma \ref{lemma42}, we view it on the complex torus $\mathbb C \slash L$.
Then $h$ is characterized by the properties
that it is harmonic on $(\mathbb C \slash L) \setminus ([0,\tau] \cup \{ \frac{1}{6}, \frac{1}{2}, \frac{5}{6}\})$ with zero boundary values
on $[0,\tau]$ (modulo $L$) and logarithmic behavior given by \eqref{hlogsing}
at the singularities.

The function $u \mapsto h(1-\overline{u})$ (modulo $L$) has the same properties,
and thus $h(u) = h(1-\overline{u})$ by uniqueness. Then it follows indeed that the two normal derivatives
	\begin{align*} \frac{\partial}{\partial n_+} \left( 2 U^{\mu} + \Re V \right)
	 & =  -\lim_{x \to 1-}  \frac{\partial}{\partial x} h(x+ iy) \\
	 \frac{\partial}{\partial n_-} \left(2 U^\mu + \Re V \right)
	 &	= \lim_{x \to 0+} \frac{\partial}{\partial x} h(x+iy) 
		\end{align*}
agree for every $y \in [0, \Im \tau]$. This proves \eqref{GreenSprop}.
\end{proof}

\section{Eigenvectors and the first transformation}
The matrix with eigenvectors of $W(z)$ is used in the first
transformation of the RH problem. We first study the eigenvectors.
Recall that $W(z)$ is given by \eqref{Wz} and under the Assumptions \ref{assump12}
it takes the form \eqref{Wzstructure}.

\subsection{Eigenvectors of $W(z)$}

We use the spectral decomposition $W(z) = E(z) \Lambda(z) E(z)^{-1}$
as in \eqref{Wdecomp}
with a matrix  $E(z)$ whose columns are the eigenvectors of $W(z)$.
We choose the eigenvectors as follows.  

\begin{definition} \label{def:E}
Let $\vec{e}(z;\lambda)$ be the third column of the
adjugate matrix $\adj(\lambda I_3 - W(z))$. 
%\footnote{Marco Bertola: check Babelon book}
Then we define
\begin{equation} \label{Edef} E(z) = \begin{pmatrix} \vec{e}(z; \lambda_1(z)) & 
	\vec{e}(z;\lambda_2(z)) & \vec{e}(z;\lambda_3(z)) \end{pmatrix}. 
\end{equation} 
\end{definition}

More explicitly, we have the following. Since $W$ has the special structure \eqref{Wzstructure}, we find that the third column of 
$\adj(\lambda I_3 - W(z))$  is  
\begin{equation} \label{ezlambda} \vec{e}(z; \lambda) = 
	\begin{pmatrix} e_1(z,\lambda) \\ e_2(z,\lambda) \\ e_3(z,\lambda) \end{pmatrix} 
	= \begin{pmatrix} w_{13}(\lambda - z-1) + w_{12} w_{23} \\	
		w_{23}(\lambda-z-1) + w_{13} w_{21} z \\
		(\lambda-z-1)^2 - w_{12} w_{21} z
	\end{pmatrix}.
\end{equation} Hence $E(z)$ is equal to
\begin{equation} \label{Edef2}
	\scalebox{0.9}{ $\begin{pmatrix} w_{13}(\lambda_1(z) - z-1) + w_{12} w_{23} & 
	w_{13}(\lambda_2(z) - z-1) + w_{12} w_{23} & %
	w_{13}(\lambda_3(z) - z-1) + w_{12} w_{23} \\ %
	w_{23}(\lambda_1(z)-z-1) + w_{13} w_{21} z & %
		w_{23}(\lambda_2(z)-z-1) + w_{13} w_{21} z & %
			w_{23}(\lambda_3(z)-z-1) + w_{13} w_{21} z \\ %
	(\lambda_1(z)-z-1)^2 - w_{12} w_{21} z & %
		(\lambda_2(z)-z-1)^2 - w_{12} w_{21} z & %
			(\lambda_3(z)-z-1)^2 - w_{12} w_{21} z %
\end{pmatrix}$}.
\end{equation}

\begin{lemma}
	The matrix valued function $E$  defined by \eqref{Edef}
	satisfies the following. 
	\begin{enumerate}
	\item[\rm (a)] $E$ is analytic and invertible on  $\mathbb C \setminus \mathbb R$.
	\item[\rm (b)] The spectral decomposition 
	\eqref{Wdecomp} holds for $z \in \mathbb C \setminus \mathbb R$.
		
	\item[\rm (c)] $E$ satisfies the jump property $E_+ = E_- J_E$ on the real line
	where
	\begin{align} \label{Ejump}
		J_E = \begin{cases} \sigma_{12}, & \text{ on } (-\infty,0], \\
			\sigma_{23}, & \text{ on } (0,z_{\min}) \cup (z_{\max},\infty), \\
			I_3, & \text{ on } [z_{\min},z_{\max}], 
			\end{cases} 
	\end{align}
	where we recall that $\sigma_{12}$ and $\sigma_{23}$ are
	given by \eqref{sigmadef}.
	\item[\rm (d)] As $z \to \infty$ with $\pm \Im z > 0$,
	and $c_{\lambda} > 0$ and $\omega = e^{2\pi i/3}$ as in Lemma \ref{lemma23} (c), one has
	\begin{multline} \label{Easymp} 
		\diag\left(z^{-2/3}, z^{-1}, z^{-4/3}\right)E(z) = 
		\diag\left(w_{13} c_\lambda, w_{13} w_{21}, c_\lambda^2\right)
		\begin{pmatrix} 1 & \omega^{\pm 1} & \omega^{\mp 1} \\
			1 & 1 & 1 \\
			1 & \omega^{\mp 1} & \omega^{\pm 1} \end{pmatrix}  \\
		+  
		\diag\left(\frac{1}{3} w_{13} c_\lambda^2 , w_{23} c_{\lambda}, \frac{2}{3} c_{\lambda}^3 - w_{12} w_{21}\right)
		\begin{pmatrix} 1 & \omega^{\mp 1} & \omega^{\pm 1} \\
			1 & \omega^{\pm 1} & \omega^{\mp 1} \\
			1 & 1 & 1  \end{pmatrix} z^{-1/3}  + \mathcal{O}(z^{-2/3})
	\end{multline} 
	and
	\begin{multline} \label{Einvasymp} 
		E^{-1}(z) \diag\left(z^{2/3}, z, z^{4/3}\right) =
		\frac{1}{3}
		\begin{pmatrix} 1 & 1 & 1 \\
			\omega^{\mp 1} & 1 & \omega^{\pm 1} \\
			\omega^{\pm 1} & 1 & \omega^{\mp 1}
		\end{pmatrix}  
		\diag\left(w_{13} c_\lambda, w_{13} w_{21}, c_\lambda^2\right)^{-1} \\
		-  
		\frac{1}{3}
		\begin{pmatrix} 1 & 1 & 1 \\
			1 & \omega^{\pm 1} & \omega^{\mp 1} \\
			1 & \omega^{\mp 1} & \omega^{\pm 1}
		\end{pmatrix}  
		\diag \left(\frac{w_{23}}{w_{13}^2w_{21}}, \frac{\frac{2}{3} c_{\lambda}^3 - w_{12}w_{21}}{w_{13} w_{21} c_\lambda^2}, \frac{1}{3 c_\lambda} \right)
		z^{-1/3} 
		+ \mathcal{O}(z^{-2/3}).
	\end{multline}  	
		\end{enumerate}
	\end{lemma}
	\begin{proof}
	(a) and (b)		
		By the properties of adjugate matrix one has
		\[ (\lambda I_3 - W(z)) \adj(\lambda I_3 - W(z))
		= \det(\lambda I_3 - W(z)) I_3. \]
		Hence, 
		\[ (\lambda I_3 - W(z)) \vec{e}(z;\lambda) = \det(\lambda I_3 - W(z)) 
		\begin{pmatrix} 0 \\ 0 \\ 1 \end{pmatrix}, \]
		and then by \eqref{Edef} it follows that 
		$W(z) E(z) = E(z) \Lambda(z)$.
		Thus \eqref{Wdecomp} holds whenever $E(z)$ is invertible.
		
	Because of \eqref{ezlambda} and \eqref{Edef2} the first
	row of $E(z)$ is built out of the function $e_1(z,\lambda) = w_{13}(\lambda-z-1) + w_{12} w_{23}$. This is a line in the real $z,\lambda$-plane that
	intersects the bounded oval in two points $Q^*$ and $Q_1^*$,
	see Figure \ref{fig:pointQstar} and also below, and these are the only
	zeros of $e_1$ on $\mathcal R$. Thus, if $z \in \mathbb C \setminus \mathbb R$
	then $e_1(z, \lambda_j(z)) \neq 0$ for every $j=1,2,3$. It then
	follows that $\vec{e}(z,\lambda_j(z))$ cannot be the zero vector,
	and thus it is an eigenvector of $W(z)$ with eigenvalue $\lambda_j(z)$. 
	Since $W(z)$ has simple spectrum for $z \in \mathbb C \setminus \mathbb R$
	the eigenvectors are linearly independent, and it follows that
	$E(z)$ is invertible. The analyticity of $E$ is clear from
	\eqref{Edef2}, and we proved parts (a) and (b) of the Lemma.
		
	\medskip	
	(c) The jump matrices \eqref{Ejump} readily follow from
	\eqref{Edef2} and the jump properties of the $\lambda$-functions
	on the real line 	contained in \eqref{Lambdajump}. 

\medskip
(d) The asymptotic behaviors \eqref{Easymp} and \eqref{Einvasymp}
follow  from \eqref{Edef2} and \eqref{lambda1asymp}--\eqref{lambda3asymp} by 
straightforward  calculations.
\end{proof}

\begin{figure}[t]
\begin{center}
\begin{overpic}[trim=0 12cm 2cm 0cm, clip,scale=0.5]{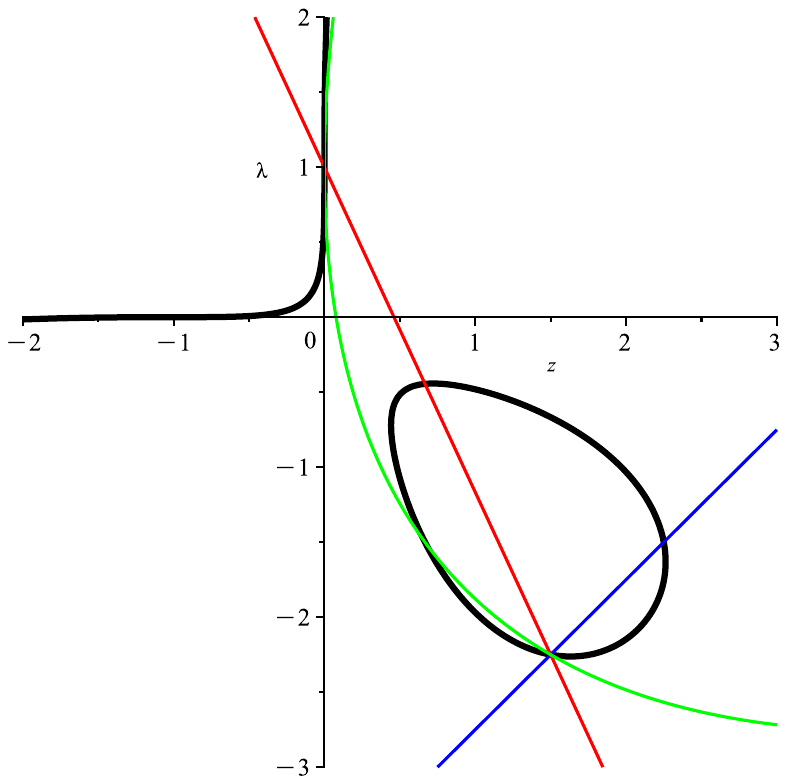}
	\put(55,10){$Q^*$}
	\put(68,22){$Q_1^*$}
	\put(41,22){$Q_3^*$}
	\put(47,40){$Q_2^*$}
	\put(40,57){$P_0$}
	\put(25,45){$P_1$}
\end{overpic}
\caption{Real part of the spectral curve in the $z$-$\lambda$ plane (in black) together with the zero level lines
of the functions that define the matrix $E$. The three zero level lines
intersect the spectral curve at the point $z=z^*$, $\lambda= \lambda^*$ on
the bounded oval. Each zero level line has a second intersection with
the bounded oval. \label{fig:pointQstar}}
\end{center}
\end{figure}

If $z$ is one of the branch points $0,z_{\min}$, or $z_{\max}$
then $E(z)$ is not invertible as two or more columns in
\eqref{Edef2} are identical. There is one 
additional real value $z^*$ for which $E(z)$ is not invertible.

\begin{lemma} \label{lemma53}
	\begin{enumerate}  
	\item[\rm (a)]  We have
	\begin{equation} \label{Edet}	
		\det E = (w_{12} w_{23}^2 - w_{13}^2 w_{21}z) 
		(\lambda_2-\lambda_1)(\lambda_3-\lambda_1)(\lambda_3-\lambda_2). \end{equation}
	\item[\rm (b)]
	$(\det E)^2$ is a polynomial of degree six with
	simple zeros at $z_{\min}$, $z_{\max}$, a double zero at $0$,
	and also a double zero at
	\begin{equation} \label{zstardef} 
		z^* = \frac{w_{12} w_{23}^2}{w_{13}^2 w_{21}}. 
		\end{equation}
	
	\item[\rm (c)]
	There is a value $\lambda^*$, such that
	$(z^*, \lambda^*)$
	with $z^*$ given by \eqref{zstardef}, 
	is a common zero of all three functions $e_1$, $e_2$, $e_3$
	defined in \eqref{ezlambda}.
	The point $Q^*= (z^*, \lambda^*)$ is on the bounded
	oval of the Riemann surface.
	
	\item[\rm (d)]
	Each of the equations $e_j=0$ for $j=1,2,3$ has exactly one other
	zero on the bounded oval, say at $Q_j^*$. The three points $Q_j^* = (z_j^*, \lambda_j^*)$, $j=1,2,3$, 
	are all distinct
	and their images under the Abel map lie at equal distances from each other
	on the line $\Im u = \frac{1}{2} \Im \tau$, namely at
	$\frac{\tau}{2} +x$, $\frac{\tau}{2} + x + \frac{1}{3}$, and $\frac{\tau}{2} + x + \frac{2}{3}$ modulo $L$ for some real $x$.	
	\end{enumerate}
	
\end{lemma}
\begin{proof}
(a) follows from the formula \eqref{Edef2} for $E(z)$, by direct calculation.
\medskip

(b) From \eqref{Edet} and \eqref{zstardef} we obtain
\[ (\det E)^2 = w_{13}^4 w_{21}^2 (z-z^*)^2
	\left[(\lambda_2-\lambda_1)(\lambda_3-\lambda_1)(\lambda_3-\lambda_2)\right]^2 \]
The product $	\left[(\lambda_2-\lambda_1)(\lambda_3-\lambda_1)(\lambda_3-\lambda_2)\right]^2$ is a polynomial in $z$. It is the discriminant
\eqref{DiscP} of the spectral curve. Therefore it is
a degree four polynomial with a double zero at $z=0$ 
and two simple zeros at $z_{\min}$ and $z_{\max}$. 
Part (b) follows.

\medskip

(c) The equations $e_1 = 0$, $e_2 = 0$ give two straight 
lines in the
real  $z$-$\lambda$
	plane that intersect at $(z^*,\lambda^*)$ 
	with $z^*$ given by \eqref{zstardef} and $\lambda^* \in \mathbb R$
	satisfying
	\begin{equation} \label{lambdastar1} 
		\lambda^* - z^* -1 = -\frac{w_{12} w_{23}}{w_{13}} =
		- \frac{w_{13} w_{21}}{w_{23}} z^*. \end{equation}
	Hence
	\[ (\lambda^* - z^* -1)^2 = \left(-\frac{w_{12} w_{23}}{w_{13}} \right) 
	   \left(- \frac{w_{13} w_{21}}{w_{23}}\right) z^* 
	   	= w_{12} w_{21} z^* \] 
	and we see from the expression for $e_3$ in \eqref{ezlambda},  that $e_3(z^*,\lambda^*) = 0$ as well. 
	
	Then $\vec{e}(z^*,\lambda^*)$ is the null vector. Since it is the cross
	product of the first two rows of $\lambda^* I_3 - W(z^*)$, we conclude
	that the first two rows are linearly dependent. Then $\lambda^* I_3 - W(z^*)$ is not invertible, and thus $\det(\lambda^* I_3 - W(z^*))=0$,
	i.e., $Q^* = (z^*,\lambda^*)$ is on the spectral curve, and
	$\lambda^* = \lambda_j(z^*)$ for some $j=1,2,3$.
	
	Because of \eqref{lambdastar1} and $z^* > 0$ we have $\lambda^* < z^*+1$.
	Then due to \eqref{Wzstructure}, 
	\[ \lambda_1(z^*) + \lambda_2(z^*) + \lambda_3(z^*)
		= \tr W(z^*) = 3(z^*+1) > 3 \lambda_j(z^*). \]
	The case $j=1$ would  lead to
	a contradiction with the ordering \eqref{lambdajordered} of absolute
	values, as $\lambda_1(z^*) > 0$. Thus $j=2$ or $j=3$, and in either case 
	the point $Q^*$ is on the bounded oval of the Riemann surface.
	
	\medskip
	
	(d)	
	For $j=1,2,3$ we have that $e_j$, viewed as a function on the spectral curve,
	is meromorphic with pole of order $j+1$ at $P_{\infty}$ and
	no other poles. In addition, $e_2$ has a zero at $P_0$, and 
	$e_3$ has a double zero at $P_0$, as follows from their expressions
	in \eqref{ezlambda}. All three of them have a zero at $Q^*$, as we just proved in part (c).
	Then $e_j$ has one additional zero on $\mathcal R$, say at $Q_j^*$,
	and by Abel's theorem, we have
	\begin{align*} 
		\mathcal{A}(Q^*) + \mathcal A(Q_1^*) & \equiv 2 \mathcal A(P_{\infty}), \\
		\mathcal{A}(Q^*) + \mathcal A(Q_2^*) + \mathcal A(P_0) & \equiv 3 \mathcal A(P_{\infty}), \\
		 \mathcal{A}(Q^*) + \mathcal A(Q_3^*) + 2\mathcal A(P_0) & \equiv 4 \mathcal A(P_{\infty}), \end{align*} 
		modulo $L$.  
	Since $\mathcal A(P_\infty) \equiv \frac{1}{6}$ and $\mathcal A(P_0) \equiv \frac{5}{6}$, the equations are
	\begin{align*} 
		\mathcal A(Q_1^*) \equiv - \mathcal{A}(Q^*) + \frac{1}{3}, \\
		\mathcal A(Q_2^*) \equiv - \mathcal{A}(Q^*) - \frac{1}{3}, \\
		\mathcal A(Q_3^*) \equiv - \mathcal{A}(Q^*) - 1. 
	\end{align*} 
	This shows that the values of the Abel map are indeed of the form
	$x + \frac{\tau}{2} + \frac{j}{3}$, $j=1,2,3$, for some real $x$,
	as claimed in part (d) of the lemma.  
	
	We also conclude that $\mathcal A(Q_j^*) \in \frac{\tau}{2} + [0,1]$,
	and thus $Q_j^*$ is on the bounded oval for $j=1,2,3$,
	see Figure \ref{fig:pointQstar} for an illustration.
\end{proof}

It is possible to write the coordinates of $Q^*$ directly in
terms of the input data $a_{jk}$ and $b_{jk}$, namely
\begin{equation} \label{zstar} z^* = \frac{a_{12} b_{31}}{a_{31} b_{11}}, \quad
	\lambda^* = - \frac{a_{11} b_{21}}{a_{22} b_{11}}. \end{equation}	
From part (d) of the lemma it follows that $z^* \in [z_{\min},z_{\max}]$.
It is possible that $z^*$ coincides with one of the branch points $z_{\min}$
or $z_{\max}$. In such a situation, we should interpret part (b) as saying
that $(\det E)^2$ has a triple zero at that branch point. For generic parameters this is however not the case.

For the special choice of parameters in \eqref{Tjconcrete}
one has $z^* = \alpha_1$, $\lambda^* = - \frac{\alpha_1}{\alpha_2}$.

\begin{center}
	\begin{table} 
		\begin{tabular}{c|c|c|c|c|c|c}
			& $P_0$ & $P_\infty$ & $Q^*$ & $B_{\min}$ & $B_{\max}$ & \\ \hline & & & & & \\
			$e_1$ & $-$ & double pole & simple zero & $-$ & $-$ &   \\
			$e_2$ & simple zero & triple pole & simple zero & $-$ & $-$ &  \\
			$e_3$ & double zero & quadruple pole & simple zero & $-$ & $-$ & \\ \hline & & & & & \\
			$\widetilde{e}_1$ & $-$ & double zero & simple pole & simple pole & simple pole &  \\
			$\widetilde{e}_2$ & simple pole & triple zero & simple pole & simple pole & simple pole &  \\
			$\widetilde{e}_3$ & double pole & quadruple zero & simple pole & simple pole & simple pole &  \\ \hline
			& & & & & & 
		\end{tabular}
		\caption{Poles and zeros of the meromorphic functions
			$e_1, e_2, e_3$, and $\widetilde{e}_1, \widetilde{e}_2, \widetilde{e}_3$. Each of the functions has an additional zero on the bounded oval
			that is not listed in the table. \label{table1}}
	\end{table}
\end{center}

We have seen in \eqref{ezlambda} and \eqref{Edef2}
that $E(z)$ is built out of three meromorphic functions
$e_1, e_2, e_3$ on the Riemann surface. There is a similar
structure for $E(z)^{-1}$, that we note here for future
reference. 
\begin{lemma} \label{lemma54}
	There exist meromorphic functions $\widetilde{e}_1$,
	$\widetilde{e}_2$, $\widetilde{e}_3$ on $\mathcal R$
	such that
	\begin{equation} \label{Einvdef}
		E(z)^{-1} = \begin{pmatrix} 
			\widetilde{e}_k(z, \lambda_j(z)) \end{pmatrix}_{j,k=1}^3. 
	\end{equation}
\end{lemma}

\begin{proof}
	For two meromorphic functions $e$ and $\widetilde{e}$ on
	$\mathcal R$ we write
	\begin{equation} \label{pairing} \langle e, \widetilde{e} \rangle
		= \sum_{l=1}^3 e(z, \lambda_l(z)) \, \widetilde{e}(z, \lambda_l(z)). \end{equation}
	This is a rational function in the $z$-variable. 
	Then, that for any choice of $\widetilde{e}_k$ for $k=1,2,3$
	we will have because of \eqref{Edef2} and \eqref{pairing}
	\begin{align*} E(z) \begin{pmatrix} 
			\widetilde{e}_k(z, \lambda_l(z)) \end{pmatrix}_{j,k=1}^3
		& = \begin{pmatrix} 
			e_j(z, \lambda_l(z)) \end{pmatrix}_{j,l=1}^3
		\begin{pmatrix} 
			\widetilde{e}_k(z, \lambda_l(z)) \end{pmatrix}_{l,k=1}^3 \\
		& = \begin{pmatrix} \langle e_j, \widetilde{e}_k \rangle \end{pmatrix}_{j,k=1}^3.
	\end{align*}
	Hence we are looking for $\widetilde{e}_k$ such that
	\begin{align} \label{dualpair} \langle e_j, \widetilde{e}_k \rangle = \delta_{j,k} \end{align}
	for $j, k = 1,2,3$, and then \eqref{Einvdef} follows.
		
	In Table \ref{table1} we are specifying zeros and poles
	of $\widetilde{e}_1$, $\widetilde{e}_2$, and $\widetilde{e}_3$.
	The table also includes the zeros and poles of
	$e_j$ for $j=1,2,3$, as determined in Lemma~\ref{lemma53} and its proof,
	except for the zero $Q_j^*$, which is not included in the table. 
	
	The table also gives one more pole than zero for each $\tilde{e}_j$ (counting multiplicities). The Abel theorem
	tells  us that there is one more zero for each $\tilde{e}_j$
	and this additional zero is on the bounded oval. It could
	coincide with one of $Q^*$, $B_{\min}$ or $B_{\max}$
	and in such a case there is a zero/pole cancellation.
	
	With this understanding $\widetilde{e}_1$, $\widetilde{e}_2$, and $\widetilde{e}_3$ are uniquely determined, up
	to a multiplicative constant.
	
	From the table it can be checked
	$e_j \widetilde{e}_k$ has at most double poles at $P_0$
	and $P_{\infty}$ and at most simple poles at $B_{\min}$
	and $B_{\max}$, and these are the only poles.
	Thus $\langle e_j, \widetilde{e}_k \rangle$ is
	a rational function with possible poles at $0$, $z_{\min}$,
	$z_{\max}$, and $\infty$ only.
	
	Since $z^{1/3}$ is a local coordinate at $P_0$,
	and $P_0$ is at most a double pole, we have
	$e_j(z, \lambda_l(z)) \widetilde{e}_k(z,\lambda_l(z))
	= \mathcal{O}(z^{-2/3})$ as $z \to 0$
	for every $l =1,2,3$. Thus by \eqref{pairing}
	we have $\langle e_j, \widetilde{e}_k \rangle 
	= \mathcal{O}(z^{-2/3})$ as $z \to 0$,
	which means that the singularity at $z=0$
	is removable. Similarly, the singularities
	at $z_{\min}$, $z_{\max}$ and $\infty$ are removable,
	and it follows that $\langle e_j, \widetilde{e}_k \rangle$
	is a constant. 
	
	For $j \neq k$ we find from Table \ref{table1}, that 
	$e_j \widetilde{e}_k$ has a zero at either $P_0$ or $P_{\infty}$, and hence  $\langle e_j, \widetilde{e}_k \rangle = 0$ for $j \neq k$.
	
	It also follows that $e_j \widetilde{e}_j$ has no
	zero or pole at both $P_0$ and $P_{\infty}$. Then
	we can choose the multiplicative constant for $\widetilde{e}_j$ such that 
	 $e_j \widetilde{e}_j  = \frac{1}{3}$ at $P_0$.
	 Then we obtain from \eqref{pairing} that
	 $\langle e_j, \widetilde{e}_j \rangle = 1$ at $P_0$,
	 and since it is constant we obtain 
	  $\langle e_j, \widetilde{e}_j \rangle = 1$ everywhere.
	  
	  Hence \eqref{dualpair} holds, and we have shown
	  that $E(z)^{-1}$ indeed has the form \eqref{Einvdef}.
\end{proof}

\subsection{First transformation $Y \mapsto X$}

Following \cite{DK21} and \cite{KP24+}, we use $E$ in the first transformation $Y \mapsto X$ of the RH problem.

\begin{definition}\label{YtoX} We define
\begin{equation}  \label{Xdef}
 X(z) = Y(z) \begin{pmatrix} E(z) & 0_3 \\ 0_3 & E(z) \end{pmatrix}. 
\end{equation}
where $E(z)$ is given in \eqref{Edef}.
\end{definition}

Then $X$ satisfies the following RH problem. It is immediate
from \eqref{Xdef}, the RH problem~\ref{rhpforY} for $Y$ (with $B=C=1$ and we choose $\gamma = \mathbb T$), 
and the jump matrix \eqref{Ejump} for $E$.
\begin{rhproblem} \label{rhpforX} $X$ defined by \eqref{Xdef}
	satisfies the following.
	\begin{description}
		\item[RHP-X1] $X : \mathbb C \setminus \Sigma_X \to \mathbb C^{6 \times 6}$ is analytic where $\Sigma_X = \mathbb T \cup \mathbb R$. 
		\item[RHP-X2] $X_+ = X_- J_X$ on $\Sigma_X$ where
			\begin{align} \label{Xjump1} 
			J_X(z) &= 	\begin{pmatrix} I_3 & z^{-2N} \Lambda(z)^{2N} \\
			0_3 & I_3 \end{pmatrix}, \quad z \in \mathbb T, \\
			\label{Xjump2} 
			J_X(z) & = 
				\begin{pmatrix} J_E(z) & 0_3 \\ 0_3 & J_E(z) \end{pmatrix},
				\qquad z \in \mathbb R,			
		\end{align} 
		with $\Lambda(z) = \diag(\lambda_1(z), \lambda_2(z),\lambda_3(z))$, and $J_E$ is given by \eqref{Ejump}.
	\item[RHP-X3] As $z \to \infty$,
	\begin{equation} \label{Xasymp} 
		X(z) = \left(I_6 + \mathcal{O}(z^{-1}) \right) \begin{pmatrix} z^N E(z) & 0_3 \\
		0_3 & z^{-N} E(z) \end{pmatrix}. \end{equation}
	\item[RHP-X4] $X \begin{pmatrix} E^{-1} & 0_3 \\ 0_3 & E^{-1} \end{pmatrix}$
	remains bounded near each of
	the values $0$, $z_{\min}$, $z_{\max}$, and $z^*$,
	where $E^{-1}$ becomes unbounded. (See Lemma \ref{lemma53}
	and \eqref{zstar} for the special point $z^*$.)
	\end{description}
	\end{rhproblem}
The condition {\bf RHP-X4} is added 
in order to guarantee unique solvability of the RH problem.
Due to \eqref{Ydet}, \eqref{Edet}, \eqref{zstardef} and \eqref{Xdef} one has
\begin{align} \nonumber
	\det X(z)  & = \det Y(z) \left[\det E(z) \right]^2 \\
	& = w_{13}^4 w_{21}^2  
	(z-z^*)^2  \prod_{1\leq j < k \leq 3}
	(\lambda_k(z) - \lambda_j(z))^2,  \label{Xdet} 
\end{align}
which is a polynomial of degree $6$ with a double zero at $z^*$,
simple zeros at $z_{\min}$ and $z_{\max}$, and a double zero
at $0$. Thus $X$ is not invertible at these special points, and the
same is true for $E$. However, the singularities disappear
in the product $X \begin{pmatrix} E^{-1} & 0_3 \\ 0_3 & E^{-1}
\end{pmatrix}$ according to {\bf RHP-X4}.  
There will be a similar condition in the RH problems that follow.

The fact that $\det X$ is not identically one, presents
a number of complications in the analysis that follows. The construction
of the global parametrix will be done in  such a way that
after the final transformation we are back at a RH problem
whose solution has determinant identically one.

\begin{remark}
The upper right block in the jump matrix \eqref{Xjump1} on the unit circle
is in diagonal form. It means that only rows and columns 1-4, 2-5, and 3-6 are
connected via the jump on $\mathbb T$. We may alternatively write
\eqref{Xjump1} as 
\begin{multline} \label{Xjump3} J_X(z) \\
= \Pi \left[  \begin{pmatrix} 1 & \ds \frac{\lambda_1(z)^{2N}}{z^{2N}} \\ 0 & 1 \end{pmatrix} \oplus \begin{pmatrix} 1 & \ds \frac{\lambda_2(z)^{2N}}{z^{2N}} \\ 0 & 1 \end{pmatrix} \oplus \begin{pmatrix} 1 & \ds \frac{\lambda_3(z)^{2N}}{z^{2N}} \\ 0 & 1 \end{pmatrix}	 \right] \Pi^{-1}, \quad z \in \mathbb{T} \end{multline}
with the permutation matrix 
\begin{equation} \label{Pidef} \Pi = \begin{pmatrix} \begin{smallmatrix} 1 & 0 & 0 & 0 & 0 & 0 \\
	0 & 0 & 1 & 0 & 0 & 0 \\ 0 & 0 & 0 & 0 & 1 & 0 \\
	0 & 1 & 0 & 0 & 0 & 0 \\ 0 & 0 & 0 & 1 & 0 & 0 \\
	0 & 0 & 0 & 0 & 0 & 1  
	\end{smallmatrix} \end{pmatrix} \end{equation} 
associated with the permutation $(2 3 5 4)$. 
In \eqref{Xjump3} we use $\oplus$ to  denote the direct sum
of square matrices,
i.e, $A \oplus B \oplus C$ is the block diagonal matrix $\begin{pmatrix} 
	\begin{smallmatrix} A & 0 & 0 \\ 0 & B & 0 \\ 0 & 0 & C \end{smallmatrix} \end{pmatrix}$. We continue to use this notation in what follows.
\end{remark}

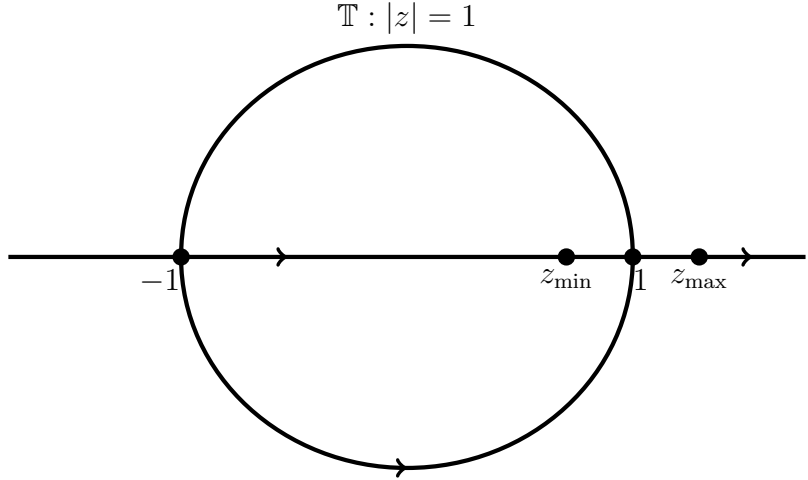
\begin{figure}
	\centering
	\begin{tikzpicture}[scale=3.5]
		\begin{scope}[ultra thick,decoration={
				markings,
				mark=at position 0.5 with {\arrow{>}}}
			] 
			
			\draw [postaction={decorate}] (-1.5,0)--(0.6,0);	
			\draw [postaction={decorate}] (1.1,0)--(1.5,0);
			\draw (0.6,0)--(1.1,0);
			
			\filldraw (0.6,0)  circle (0.7pt);	
			\filldraw (-0.85,0) circle (0.7pt);
			\filldraw  (1.1,0) circle (0.7pt);
			\filldraw (0.85,0) circle (0.7pt);
			\draw (0.6,0) node [below]{$z_{\min}$};
			\draw (1.1,0) node [below]{$z_{\max}$};
			\draw (-0.93,0) node [below]{$-1$};
			\draw (0.88,0) node [left, below]{$1$};
			\draw (0,0.8)  node[above]{$\mathbb T : |z|=1$};
			\coordinate (a) at (0.85,0);
			\coordinate (b) at (0,0.8);
			\coordinate (c) at (-0.85,0);
			\coordinate (d) at (0,-0.8);
			\coordinate (e) at (0.85,0);               
			\path[draw,use Hobby shortcut,closed=true]  (a)..(b)..(c)..(d)..(e); 
			\draw [postaction={decorate}] (0,-0.8) node{};
		\end{scope}
	\end{tikzpicture}
	
	\caption{Contour $\Sigma_X$ for the RH problem \ref{rhpforX} for $X$.
		\label{fig:SigmaX}} 
\end{figure}

\section{$g$-functions and the second transformation}
\label{section6}

In the second transformation we use the equilibrium measure $\mu$
in the external field, with  
the properties stated in Theorem \ref{theorem17}.

\subsection{The $g$-function}
\label{section61}
The equilibrium measure gives rise to a $g$-function,
an analytic function that will be such that $\Re g = - U^\mu$.
We define $g$ in terms of the complexified bipolar Green's kernel
that we denote by $G^{\mathbb C}_{P_{\infty}}$. 

Recall that we use the bipolar Green's kernel as in
Definition \ref{definition16} with the additional
normalization \eqref{bGfatPinfty2} at infinity.
Thus in view of \eqref{G0uv} and the
fact that $\mathcal A(P_\infty) \equiv \frac{1}{6}$,
we have with $u = \mathcal A(p)$, $v = \mathcal A(q)$,
\begin{align} 
	G_{P_{\infty}}(p,q) 
	& = \log \left| \frac{\theta_1(u- \frac{1}{6}) \theta_1(v- \frac{1}{6})}{\theta_1(u-v)} \right|- \frac{2\pi}{\Im \tau} (\Im u) (\Im v)
	- c_0. \label{Gdef2} 
\end{align}
The constant $c_0$ is such that \eqref{bGfatPinfty2} holds.

\begin{definition}
For $p \in \mathcal R \setminus \left( \mathbf{a} \cup \mathbf{b} \right)$,  $q \in {\mathbf{b}} = \Gamma_1 \cup \Gamma_2$, we define 
\begin{align}  \label{GCdef2} 
	G^{\mathbb C}_{P_{\infty}}(p,q) 
	& = \log \frac{\theta_1(u- \frac{1}{6}) \theta_1(v- \frac{1}{6})}{\theta_1(u-v)}- \frac{2\pi u}{\tau} (\Im v)
	- c_0, 
\end{align}
where $u = \mathcal{A}(p) \in (0,1) \times (0,\tau)$, $v  = \mathcal{A}(q) \in (0,\tau]$, 
and  the logarithm in \eqref{GCdef2} is 
to be interpreted as
\begin{align} \label{GCdef3} \log \theta_1(u- \tfrac{1}{6}) + \log \theta_1(v- \tfrac{1}{6})  - \log \theta_1(u-v) \end{align}
with principal branches of the logarithm.
\end{definition}
We emphasize that in \eqref{GCdef2} we use the
specific representatives of $\mathcal A(u)$
and $\mathcal A(v)$ in $\mathbb C \setminus L$
as indicated. The logarithms in \eqref{GCdef3}
are then well-defined, as none of 
$\theta_1(u-\frac{1}{6})$, $\theta_1(v-\frac{1}{6})$, 
and $\theta_1(u-v)$ can be real and negative.\footnote{We thank  Mateusz Piorkowski for this remark.}

We clearly have
\begin{equation} \label{GCdef1} 
	\Re G^{\mathbb{C}}_{P_{\infty}}(p,q) = G_{P_{\infty}}(p,q),
	\quad p \in \mathcal R \setminus (\mathbf{a} \cup \mathbf{b}),
		\quad q \in \mathbf{b}. 
	\end{equation}
 
\begin{definition}
We define 
\begin{equation} \label{gpdef} 
	\widehat{g}(p) = - 3\int G^{\mathbb C}_{P_{\infty}}(p,q) d\mu(q),
		\quad p \in \mathcal R \setminus (\mathbf{a} \cup \mathbf{b}). \end{equation}
Its restrictions to the various sheets are denoted by $\widehat{g}_1$, $\widehat{g}_2$, $\widehat{g}_3$,
and these are viewed as functions in the $z$-plane.
\end{definition}

The functions $\widehat{g}_j$ are closely related to, but not identical to,
the functions $g_j$ that appear in our main Theorems \ref{theorem13} and \ref{theorem14}. They differ by certain purely imaginary constants. See \eqref{gjdef} below for the definition of the functions $g_j$.

 The properties of the functions $\widehat{g}_j$
are collected in the following lemma.

\begin{lemma} \label{lemma63}
\begin{enumerate}
	\item[\rm (a)]  $\widehat{g}_1$ is defined and analytic in $\mathbb C \setminus
	(\mathbb R \cup \mathbb T)$,  $\widehat{g}_2$ is defined and analytic in $\mathbb C \setminus
	((-\infty, z_{\min}] \cup [z_{\max}, \infty) \cup \mathbb T)$,
	and  $\widehat{g}_3$ is defined and analytic in $\mathbb C \setminus
	((-\infty, z_{\min}] \cup [z_{\max}, \infty))$,
		\item[\rm (b)] on the real line we have the
	following jump conditions 
	\begin{align} \label{g1g2g3pm1} 
		\widehat{g}_{1,\pm} = \widehat{g}_{2,\mp}, & \quad \text{ on } (-\infty,0], \\
		\label{g1g2g3pm2} 
		\widehat{g}_{2,\pm} = \widehat{g}_{3,\mp}, & \quad \text{ on } [0,z_{\min}]\cup [z_{\max},\infty),
	\end{align}
	and for some integers $k_1, k_2 \in \mathbb Z$,
	\begin{align} 
		\label{g1g2g3pm3} 
		\widehat{g}_{3,+} - \widehat{g}_{3,-} =  \pi i + 6 \pi i k_1, & \quad \text{ on } (-\infty,0], \\
		\label{g1g2g3pm4} 
		\widehat{g}_{1,+} - \widehat{g}_{1,-} =  \pi i + 6 \pi i k_1, & \quad \text{ on } [0,1), \\
		\label{g1g2g3pm5} 
		\widehat{g}_{1,+} - \widehat{g}_{1,-} =  \pi i + 6 \pi i k_2, &  \quad \text{ on } (1,\infty), 
	\end{align}
	\item[\rm (c)] there are 
	$\ell_j^{\pm} \in i \mathbb R$	for $j=1,2,3$, such that
	\begin{equation} \label{hatgjasymp} 
	\widehat{g}_j(z) 
	=  \log z  + \ell_j^{\pm} + \mathcal{O}(z^{-1/3}), \quad \text{ as } z \to \infty
		\text{ with } \pm \Im z > 0, \end{equation} 
	where we use the principal branch of the logarithm,
	i.e., $\log z = \log |z| + i \arg z$, with $-\pi < \arg z < \pi$,
	and modulo $2 \pi i$ we have 
	\begin{equation} \label{elljpm} 
		\ell_1^+ = -\frac{\pi i}{2}, \, \ell_1^- = -\frac{3\pi i}{2}, \,
		\ell_2^+= \frac{\pi i}{2}, \, \ell_2^- = \frac{3\pi i}{2}, \,
		\ell_3^+ = \frac{3\pi i}{2}, \, \ell_3^- = \frac{\pi i}{2}, 
		\end{equation}
\item[\rm (d)] there is a constant $\widehat{\ell}$ with $\Re \widehat{\ell} = \ell$ (with 
$\ell$ as in \eqref{muELvarcon}) such that for $z \in \mathbb T$,
\begin{align} \label{g1ongamma}
	\widehat{g}_{1,+}(z) + \widehat{g}_{1,-}(z)  - 2 \log z + 2 \log \lambda_1(z)  & = -2 \widehat{\ell}, \\
	\label{g2ongamma}
	\widehat{g}_{2,+}(z) + \widehat{g}_{2,-}(z)  - 2 \log z + 2 \log \lambda_2(z) & = -2\widehat{\ell}, \\
	\label{g3ongamma} 
	\Re\left(2 \widehat{g}_3(z) -  2 \log z + 2 \log \lambda_3(z) \right) & < -2\Re \widehat{\ell} = -2 \ell.
\end{align}
\end{enumerate}
\end{lemma} 
\begin{proof}
(a) follows immediately from the sheet structure of
	the Riemann surface and fact that $g$ is analytic
	in $\mathcal R \setminus (\mathbf{a} \cup \mathbf{b})$.
	This also implies the jump conditions \eqref{g1g2g3pm1}
	and \eqref{g1g2g3pm2} of item (c).

\medskip

(b) We already proved \eqref{g1g2g3pm1}
and \eqref{g1g2g3pm2}. To prove the other identities 
we take $p \in \mathbf{a} \setminus \{P_\infty, B_0\}$. Then
by \eqref{gpdef}
\begin{equation} \label{gplusgminus} \widehat{g}_+(p) - \widehat{g}_-(p) = 
	3\int \left( G_{P_{\infty}}^{\mathbb C}(p,q)_- - G_{P_{\infty}}^{\mathbb C}(p,q)_+ \right) d\mu(q). \end{equation}
We evaluate the integrand using \eqref{GCdef2}.
For $p \in \mathbf{a}$, the $+$-boundary value 
corresponds to $u = \mathcal A(p)_+ \in (0,1)$ and
the $-$-boundary value to $u + \tau$. Then if $q \in \mathbf{b}$ and 
$v = \mathcal A(q) \in (0,\tau)$, 
\begin{multline*} 
	G_{P_{\infty}}^{\mathbb C}(p,q)_- - G_{P_{\infty}}^{\mathbb C}(p,q)_+ \\
	= \log \frac{\theta_1(u +\tau - \frac{1}{6}) \theta_1(v - \frac{1}{6})}{\theta_1(u + \tau-v)} 
	- \log \frac{\theta_1(u  - \frac{1}{6}) \theta_1(v - \frac{1}{6})}{\theta_1(u -v)}
	- 2 \pi \Im v.  \end{multline*}
From the quasi-periodicity property \eqref{theta1period}, it then follows that
this is equal to $\frac{\pi i}{3} + 2\pi ik$, for some integer $k$.
By continuity, the integer $k$ does not depend on $q \in \mathbf{b}$
and thus by \eqref{gplusgminus}
\[ \widehat{g}_+(p) - \widehat{g}_-(p) =  \pi i + 6 \pi i k,
	\quad p \in \mathbf{a} \setminus \{P_{\infty}, B_0\}. \]
The constant $k$ could depend on $p \in \mathbf{a} \setminus \{P_\infty,B_0\}$, but by continuity, it is constant on each
of the two connected components of $\mathbf{a}  \setminus \{ P_\infty, B_0\}$.  This proves the identities \eqref{g1g2g3pm3}--\eqref{g1g2g3pm5}.

\medskip
	
(c)  From \eqref{GCdef2} and \eqref{gpdef} 
we see that $\Re \widehat{g}_j(z) =  \log |z| + \mathcal{O}(z^{-1/3})$ as $z \to \infty$ for every $j=1,2,3$. For the imaginary parts we pick up
certain constants depending on $j$, but also depending on whether
we go to infinity in the upper or lower half planes. We denote these
constants by $\ell_j^{\pm}$ and \eqref{hatgjasymp} follows.

Let $p$ be on the positive real axis of the first sheet of $\mathcal R$
with $z(p) > 1$.
Then $\mathcal A_+(p) \in (0,\frac{1}{6})$ and $\mathcal A_+(p)$ 
increases to $\frac{1}{6}$ as $z(p) \to \infty$. From \eqref{GCdef2}
we then get with $u = \mathcal A_+(p)$ and $v = \mathcal A(q)$,
\[ G_{P_\infty}^{\mathbb C}(p,q)_+ = \log (- \theta_1(u-\tfrac{1}{6}))
	+ \log \frac{\theta_1(v-\tfrac{1}{6})}{\theta_1(v-u)}
	- \frac{2 \pi (u-\tfrac{1}{6})}{\tau} (\Im v) - \frac{\pi}{3\tau} (\Im v) -c_0. 
	\] 
The first term on the right is real, the second and third terms are 
$\mathcal{O}(u- \frac{1}{6})$ as $u \to \frac{1}{6}$, uniformly for $v \in \mathbf{b}$ and the last term is a real constant.
The fourth term is purely imaginary and  contributes to
the constant $\ell_1^+$. Hence
\[ \ell_1^+ = - 3\lim_{p \to P_{\infty}} \int \Im G_{P_\infty}^{\mathbb C}(p,q)_+ d\mu(q)
=   \frac{\pi}{\tau} \int (\Im v) d \mathcal A_*(\mu) (v),   \]
where $\mathcal A_*(\mu)$ is the push-forward of $\mu$ under the Abel map. It is a probability measure on $[0,\tau]$
that is symmetric around $\frac{\tau}{2}$.
(This comes from the symmetry of $\mu$ under the
anti-holomorphic involution $(z,\lambda) \mapsto 
(\overline{z}, \overline{\lambda})$ on $\mathcal R$)
Hence  $\int (\Im v) d\mathcal A_*(\mu) (v) = \frac{\Im \tau}{2}$
and we obtain  $\ell_1^+ = -\frac{\pi i}{2}$.

From  \eqref{g1g2g3pm5} we have $\ell_1^+ - \ell_1^- = \pi i$
modulo $2\pi i$. Hence $\ell_1^- = - \frac{3\pi i}{2}$.
The
identities \eqref{g1g2g3pm1} and \eqref{g1g2g3pm2} imply that
$\ell_1^{\pm} \pm \pi i = \ell_2^{\mp} \mp \pi i$
and $\ell_2^{\pm} = \ell_3^{\mp}$.
The other expressions in \eqref{elljpm}
follow from  this and from the values we already have for 
$\ell_1^{\pm}$.

\medskip
(d) The real parts of the left-hand sides of \eqref{g1ongamma}-\eqref{g2ongamma} are equal to the constant $-2\ell$ because of the variational
	equality \eqref{muELvarcon} associated with the equilibrium problem
	for $\mu$. The imaginary part is constant because of the
	$S$-property \eqref{GreenSprop}.
	
	Indeed, the tangential derivative of 
	$\Im(\widehat{g}_- - \log z + \log \lambda)$ along $\Gamma_1 \cup \Gamma_2$
	is equal to one of the normal derivatives of its real part, by Cauchy Riemann equations.
	The tangential derivative of
	$\Im(\widehat{g}_+ - \log z + \log \lambda)$ along $\Gamma_1 \cup \Gamma_2$
	is equal to minus one times the other normal derivative, because
	of different orientation. Since the normal derivatives agree by \eqref{GreenSprop}, the two tangential derivatives add up
	to zero, and thus the imaginary part is constant indeed.
	
	The inequality in \eqref{g3ongamma} follows
	from the strict inequality in \eqref{muELvarcon}.
\end{proof}
In view of \eqref{hatgjasymp} we define the functions $g_j$
as follows.
\begin{definition} \label{definition64}
	We define for $j=1,2,3$,
\begin{equation} \label{gjdef}  g_j(z) = \widehat{g}_j(z) - \ell_j^{\pm} \quad \text{ for } \pm \Im z > 0, \end{equation}
where the constants and $\ell_j^{\pm}$ are
as in Lemma \ref{lemma63} (c).
\end{definition}
We note that $ g_j(z) = \log z + \mathcal{O}(z^{-1/3}) 
$  as $z \to \infty$
for every $j=1,2,3$. Other properties follows from 
the definition \eqref{gjdef} and the properties of the $\widehat{g}_j$ functions listed in Lemma \ref{lemma63}.

For example
\begin{equation} \label{g1jumps1}
\begin{aligned} 
	 g_{1,\pm}  = g_{2,\mp} + 2\pi i, & \quad \text{on } (-\infty,0], \\
	 g_{2,\pm}  = g_{3,\mp} + 2\pi i, & \quad \text{on } [0,z_{\min}] \cup [z_{\max},\infty). 
\end{aligned}
\end{equation}

\subsection{Second transformation $X \mapsto T$}

We use the functions $g_1, g_2, g_3$
from \eqref{gjdef}, and the constant $\widehat{\ell}$ from Lemma \ref{lemma63}  (d) in the definition of the second transformation
of the RH problem.
\begin{definition} \label{XtoT}  We define
	\begin{multline} \label{Tdef} 
	T = \diag \left(e^{2N \widehat{\ell}}, e^{2N \widehat{\ell}}, e^{2N \widehat{\ell}}, 1, 1,1 \right) \\
	\times  X   \diag \left( e^{-N g_1}, e^{-N g_2}, e^{-N g_3},
	e^{N g_1}, e^{N g_2}, e^{N g_3} \right) \\ 
	\times 
	\diag \left(e^{-2N \widehat{\ell}}, e^{-2N \widehat{\ell}}, e^{-2N \widehat{\ell}}, 1, 1, 1 \right). 
	\end{multline}
\end{definition}

Then $T$ is defined and analytic for $\mathbb C \setminus (\mathbb R \cup \mathbb T)$
with the following properties.
In \eqref{Tjump1} we use the following $\varphi$-functions
\begin{equation} \label{phijdef} 
	\varphi_j(z)  = -  \widehat{g}_j(z) +  \log z - \log \lambda_j - \widehat{\ell}, \quad j=1,2,3, \end{equation}
and we note that by \eqref{g1ongamma}--\eqref{g2ongamma}
\begin{equation} \label{phijongamma}
	\varphi_{j+}(z) + \varphi_{j-}(z) = 0, \quad z \in \mathbb T, \quad j=1,2.
	\end{equation}  

\begin{comment}
\begin{figure}
	\centering
	\begin{tikzpicture}[scale=4]
		\begin{scope}[ultra thick,decoration={
				markings,
				mark=at position 0.5 with {\arrow{>}}}
			] 
			
			\draw [postaction={decorate}] (-1.5,0)--(0.6,0);	
			\draw (0.6,0)--(1.1,0);
			\draw [postaction={decorate}] (1.1,0)--(1.5,0);
			
			\filldraw (0.6,0)  circle (0.7pt);	
			\filldraw (-0.85,0) circle (0.7pt);
			\filldraw  (1.1,0) circle (0.7pt);
			\filldraw (0.85,0) circle (0.7pt);
			\draw (0.6,0) node [below]{$z_{\min}$};
			\draw (1.1,0) node [below]{$z_{\max}$};
			\draw (-0.93,0) node [below]{$-1$};
			\draw (0.88,0) node [left, below]{$1$};
			\draw (0,0.8)  node[above]{$\mathbb T : |z|=1$};
			\coordinate (a) at (0.85,0);
			\coordinate (b) at (0,0.8);
			\coordinate (c) at (-0.85,0);
			\coordinate (d) at (0,-0.8);
			\coordinate (e) at (0.85,0);               
			\path[draw,use Hobby shortcut,closed=true]  (a)..(b)..(c)..(d)..(e); 
			\draw [postaction={decorate}] (0,-0.8) node{};
		\end{scope}
	\end{tikzpicture}
	
	\caption{Contour $\Sigma_T = \mathbb R \cup \mathbb T$ for the RH problem \ref{rhpforT} for $T$.
		\label{fig:SigmaT}} 
\end{figure}
\end{comment}

\begin{rhproblem} \label{rhpforT} $T$ satisfies the following.
	\begin{description} 
		\item[RHP-T1]
$T : \mathbb C \setminus \Sigma_T \to \mathbb C^{6 \times 6}$ 
is analytic, where $\Sigma_T =  \mathbb R \cup \mathbb T$,
\item[RHP-T2] $T_+ = T_- J_T$ on $\Sigma_T$ where $J_T$ is given  on the circle by
\begin{multline}
	\label{Tjump1} 
	J_T = \Pi \left[ \begin{pmatrix} e^{2N \varphi_{1+}} & (-1)^N \\
			0 & e^{2N \varphi_{1-}} \end{pmatrix} \oplus
			\begin{pmatrix}	e^{2N \varphi_{1+}} & (-1)^N \\
				0 & e^{2N \varphi_{1-}} \end{pmatrix} \oplus
			\begin{pmatrix} 1 & (-1)^N e^{-2N \varphi_3} \\
				0 & 1 \end{pmatrix} \right] \Pi^{-1}, \\
				\quad z \in \mathbb T,			
\end{multline}
%\begin{multline} \label{Tjump1}
%	J_T = \begin{pmatrix} e^{2N \varphi_{1+}}  & 0 & 0 &
%		(-1)^N & 0 & 0 \\
%		0 & e^{2N \varphi_{2+}} & 0 & 0 & (-1)^N & 0 \\
%		0 & 0 & 1 & 0 & 0 & (-1)^N e^{-2 N \varphi_3}  \\
%		0 & 0 & 0 & e^{2N \varphi_{1-}}  & 0 & 0  \\
%		0 & 0 & 0 & 0 & e^{2N \varphi_{2-}}  & 0  \\
%		0 & 0 & 0 & 	0 & 0 & 1 
%		\end{pmatrix}, \quad |z| = 1,\end{multline}
with $\varphi_j$ for $j=1,2,3$ given by \eqref{phijdef}, and
permutation matrix $\Pi$ as in \eqref{Pidef},
while on the real line,
\begin{align} \label{Tjump2} 
	J_T = \begin{cases} J_E \oplus J_E, & \text{ on } (-\infty, z_{\min}) \cup (z_{\max},\infty), \\
	\diag\left(1,(-1)^N,(-1)^N,1,(-1)^N,(-1)^N\right), & \text{ on } [z_{\min},z_{\max}], 
	\end{cases} \end{align}
\item[RHP-T3] as $z \to \infty$  
\begin{multline} \label{Tasymp}
	T(z) = \left(I_6 + \mathcal{O}(z^{-1})\right)
		\begin{pmatrix} z^N E(z)  & 0_3 \\
			0_3 & z^{-N} E(z) \end{pmatrix} \\ \times
		\diag\left(e^{-N g_1(z)}, e^{-N g_2(z)}, e^{-N g_3(z)},
		e^{N g_1(z)}, e^{N g_2(z)}, e^{N g_3(z)} \right),
		 \end{multline}
\item[RHP-T4] $T \begin{pmatrix} E^{-1} & 0_3 \\ 0_3 & E^{-1} \end{pmatrix}$ remains bounded near $0$, $z_{\min}$, $z_{\max}$, and $z^*$.
\end{description} 
\end{rhproblem}
\begin{proof}
{\bf RHP-T1}  
	The analyticity of $T$ away from $\mathbb R \cup \mathbb T$ is clear. 

{\bf RHP-T2} 
Because of \eqref{Xjump3} and \eqref{Tdef}, the jump matrix $J_T$ on the circle $\mathbb T$ has the block form
\begin{equation} \label{JTcircle1} 
	J_T = \Pi \left[ A_1 \oplus A_2 \oplus A_3 \right] \Pi^{-1}
		\qquad \text{ on } \mathbb T, \end{equation}
with
\begin{align} \nonumber A_j & = \begin{pmatrix} e^{Ng_{j-} + 2 N \widehat{\ell}}
	& 0 \\ 0 & e^{-Ng_{j-}} \end{pmatrix}
	\begin{pmatrix} 1 & \frac{\lambda_j^{2N}}{z^{2N}} \\ 0 & 1 \end{pmatrix}
	\begin{pmatrix} e^{-Ng_{j-} - 2 N \widehat{\ell}}
	& 0 \\ 0 & e^{Ng_{j-}} \end{pmatrix} \\
	& =  \begin{pmatrix} e^{-N(g_{j+}-g_{j-})} & 
		e^{N(g_{j+}+g_{j-} +2\widehat{\ell})} \frac{\lambda_j^{2N}}{z^{2N}} \\
		0 & e^{N(g_{j+}-g_{j-})} \end{pmatrix}, \qquad j=1,2,3.
		\label{JTcircle2}
	\end{align}
	
The diagonal entries in \eqref{JTcircle2} are
equal to $e^{\mp N (\widehat{g}_{j+}-\widehat{g}_{j-})}$ 
since according to  \eqref{gjdef}, $g_j - \widehat{g}_j$
is constant in both the upper and lower half-planes.
The identities \eqref{g1ongamma}, \eqref{g2ongamma}
and the definition \eqref{phijdef}  then lead to
the diagonal entries $e^{2N \varphi_{j \pm}}$ for $j=1,2$,
while for $j=3$, it is $1$ as $\widehat{g}_{3+} = \widehat{g}_{3-}$ on $\mathbb T$.
Because of \eqref{gjdef} and the fact that $e^{-2 \ell_j^{\pm}} = -1$, see \eqref{elljpm}, the $12$ entry in \eqref{JTcircle2} is
equal to
\begin{equation} \label{JT12entries} (-1)^N e^{N \widehat{g}_{j+} + N \widehat{g}_{j-} + 2N\widehat{\ell}}\left(	\frac{\lambda_j}{z} \right)^{2N}, \quad j=1,2,3. \end{equation}
For $j=1,2$, \eqref{JT12entries} is equal to $(-1)^N$ because of \eqref{g1ongamma}
	and \eqref{g2ongamma}, while for $j=3$, we have
	$g_{3+} = g_{3,-}$ and  \eqref{JT12entries}
	is equal to	$(-1)^N e^{-2N \varphi_3}$ with $\varphi_3$ given by \eqref{phijdef}. Thus the matrices $A_j$ are indeed
	as given in \eqref{Tjump1}.
	
The jumps on the real line have the block diagonal form
\[ J_T = B_1 \oplus B_2 \]
with the $3 \times 3$ blocks, that
we calcuate using \eqref{Xjump2}, \eqref{Tdef}, and \eqref{gjdef}
\begin{align*} 
	B_1 & = \begin{pmatrix} \begin{smallmatrix} e^{N(\widehat{g}_{1-}-\ell_1^-)} & 0 & 0 \\
	0 & e^{N(\widehat{g}_{2-}-\ell_2^-)} & 0 \\ 0 & 0 & e^{N(\widehat{g}_{3-}-\ell_3^-)} \end{smallmatrix} \end{pmatrix} J_E
		\begin{pmatrix} \begin{smallmatrix} e^{-N(\widehat{g}_{1+}-\ell_1^+)} & 0 & 0 \\
			0 & e^{-N(\widehat{g}_{2+}-\ell_2^+)} & 0 \\ 0 & 0 & e^{-N(\widehat{g}_{3+}-\ell_3^+)} \end{smallmatrix} \end{pmatrix},  \\
	B_2 & = \begin{pmatrix} \begin{smallmatrix} e^{-N(\widehat{g}_{1-}-\ell_1^-)} & 0 & 0 \\
			0 & e^{-N(\widehat{g}_{2-}-\ell_2^-)} & 0 \\ 0 & 0 & e^{-N(\widehat{g}_{3-}-\ell_3^-)} \end{smallmatrix} \end{pmatrix} J_E
	\begin{pmatrix} \begin{smallmatrix} e^{N(\widehat{g}_{1+}-\ell_1^+)} & 0 & 0 \\
			0 & e^{N(\widehat{g}_{2+}-\ell_2^+)} & 0 \\ 0 & 0 & e^{N(\widehat{g}_{3+}-\ell_3^+)} \end{smallmatrix} \end{pmatrix}. \end{align*}
From the jump properties
\eqref{g1g2g3pm1}--\eqref{g1g2g3pm5} of the $g$-functions on the real line, and the values \eqref{elljpm}, 
we then find the jump matrices $J_T$ as in \eqref{Tjump2}.
Note that $e^{\ell_1^{\pm}} = e^{\ell_2^{\mp}} = e^{\ell_3^{\pm}}$ and $e^{\ell_j^+-\ell_j^-} = -1$ for $j=1,2,3$.
Observe that also
$e^{\pm  (\widehat{g}_{1,+}- \widehat{g}_{1,-})} = -1$ on $(0,1) \cup (1,\infty)$
and  $e^{\pm (\widehat{g}_{3,+}- \widehat{g}_{3,-})} = -1$ on $(-\infty,0)$.

{\bf RHP-T3}  
The asymptotics \eqref{Tasymp} is clear from
\eqref{Xasymp} and \eqref{Tdef}.

{\bf RHP-T4} Using  {\bf RHP-X4} and the definition \eqref{Tdef},
we see that it suffices to show
that $E D^{\pm N} E^{-1}$ remains bounded near 
$0$, $z_{\min}$, $z_{\max}$, and $z^*$, 
where
$D$ is the diagonal matrix 
\begin{equation} \label{Ddef} 
	D = \diag(e^{g_1}, e^{g_2}, e^{g_3}). \end{equation}

From \eqref{g1jumps1} it follows that
$J_E D_+ J_E = D_-$ on $(-\infty,0) \cup (0,z_{\min}) \cup (z_{\max}, \infty)$, which implies that
$(E D E^{-1})_+ = (E D E^{-1})_-$ on these intervals.
Thus $0$ is an isolated singularity of $EDE^{-1}$.
The entries of $E^{-1}(z)$ are $\mathcal{O}(z^{-2/3})$ as $z \to 0$, while $E$ and $D^{\pm 1}$ remain bounded near $0$.
Hence $0$ is a removable singularity, and $ED^{\pm N} E^{-1}$ remains bounded near $0$.

Due to \eqref{Edef}, \eqref{Einvdef}, and \eqref{Ddef}
the entries of $E D^{\pm N} E^{-1}$ are equal to
\begin{equation} \label{sumEDEinv} 
	\sum_{l=1}^3 e_j(z, \lambda_l(z)) e^{\pm N g_l(z)} 
	\widehat{e}_k(z, \lambda_l(z)),
		\qquad j,k=1,2,3. \end{equation}
At $Q^*$ there is a zero/pole cancellation in
the product $e_j \widehat{e}_k$, see Table \ref{table1}. 
Therefore
$e_j(z, \lambda_l(z)) \widehat{e}_k(z, \lambda_l(z))$ 
for $l=1,2,3$ remains bounded as $z \to z^*$. Also $e^{\pm N g_l(z)}$ remains bounded and therefore
$E D^{\pm N}E^{-1}$ remains bounded near $z^*$.

At $z_{\min}$ the term with $l=1$
in \eqref{sumEDEinv} remains bounded.
For $\pm \Im z > 0$ and $l=2,3$ we have $e^{g_l} = \mp i e^{\widehat{g}_l}$ by \eqref{gjdef} and \eqref{elljpm}. 
Hence the terms with $l=2,3$ in \eqref{sumEDEinv} add up to
\begin{equation} \label{sumEDEinv2} (- i)^N \sum_{l=2}^3  e_j(z, \lambda_l(z) e^{\pm N \widehat{g}(z, \lambda_l(z))} e_k(z,\lambda_l(z)), \quad \Im z > 0. \end{equation}
From the definition \eqref{gpdef} we recall that
$\widehat{g}$ is analytic in a neighborhood of $B_{\min}$,
as $B_{\min}$ is not on ${\bf a} \cup {\bf b}$. 
Then $e_j \, e^{\pm \widehat{g}}\, \widetilde{e}_k$ 
is an analytic function in a neighorhood of $B_{\min}$, 
except for a simple pole at	$B_{\min}$ due to
the factor $\widetilde{e}_k$, see Table \ref{table1}.
Then the sum \eqref{sumEDEinv2} is analytic 
in $z$ in a neighboorhood of $z_{\min}$ 
except for an isolated singularity at $z_{\min}$.
Since it behaves like $\mathcal{O}((z-z_{\min})^{-1/2})$
as $z \to z_{\min}$, the singularity is removable. 
It follows that \eqref{sumEDEinv2} and thus \eqref{sumEDEinv}
remains bounded as $z \to z_{\min}$ with $\Im z > 0$.
For $\Im z < 0$, there is only a possible sign of change in \eqref{sumEDEinv2} and we also find that 
\eqref{sumEDEinv} remains bounded as $z \to z_{\min}$
with $\Im z< 0$. This shows the boundedness
of $ED^{\pm N}E^{-1}$ near $z_{\min}$.

The reasoning for $z_{\max}$ is the same,
and we completed the proof of {\bf RHP-T4}.
\end{proof}

Note that $\Re \varphi_3 \geq c > 0$ on $\mathbb T$, for some $c > 0$,
due to \eqref{g3ongamma} and \eqref{phijdef}.
Therefore the $(3,6)$ entry in the jump matrix \eqref{Tjump1} is
exponentially small as $N \to \infty$. The other non-constant
entries in \eqref{Tjump1} are highly oscillatory 
as $N \to \infty$. 

\subsection{The asymptotic condition for $T$}

The asymptotic condition \eqref{Tasymp} is not suitable to work with later on,
and we are going to rewrite it.

\begin{lemma} \label{lemma67}
	There is a unit lower triangular matrix $L$
	(size $3 \times 3$) such that
	\begin{equation} \label{Tasymp2} T(z) = \left(I_6 + \mathcal{O}(z^{-1})\right)
		\begin{pmatrix} L^N E(z) & 0_3 \\ 0_3 & L^{-N} E(z) \end{pmatrix} \end{equation}
		as $z \to \infty$.
\end{lemma}
%Unit lower triangular matrix means that the matrix is lower triangular
%with ones on the diagonal.

\begin{proof}
Because of \eqref{gjdef} and \eqref{hatgjasymp} one has 
$g_j(z) = \log z +  \mathcal{O}(z^{-1/3})$ as $z \to \infty$,
and so $z^N$ and $z^{-N}$ terms in \eqref{Tasymp} gets cancelled.

We should be more precise here. The functions $g_j - \log z$ are
the restrictions of a single function on the Riemann surface to the various sheets. 
That function is analytic near $P_{\infty}$
with a simple zero at $P_{\infty}$. It is also real-valued on
the positive real axis on the first sheet. Since $z^{-1/3}$ is 
a local coordinate at $P_{\infty}$ we find for some real
coefficients $a_j$,
\begin{equation} \label{gexpansion}
\begin{aligned}
	g_1(z) - \log z & = \sum_{j=1}^{\infty} a_j z^{-j/3}, \\
	g_2(z) - \log z & = \sum_{j=1}^{\infty} a_j \omega^{\pm j}  z^{-j/3}, \\
	g_3(z) - \log z & = \sum_{j=1}^{\infty} a_j \omega^{\mp j} z^{-j/3}, 
\end{aligned}
\end{equation}
as $z \to \infty$ with $\pm \Im z > 0$, where $\omega = e^{2\pi i/3}$ (as before). 

We use \eqref{gexpansion}  to obtain  the following expansion 
\begin{multline} \label{EgEinvexpansion}
	 z^N E(z) \diag\left(e^{-Ng_1(z)},  e^{-Ng_2(z)}, e^{-Ng_3(z)} \right) E(z)^{-1} \\
= I_3 - a_1 N z^{-1/3}
	E(z) \begin{pmatrix} 1 & 0 & 0 \\
		0 & \omega^{\pm 1} & 0 \\
		0 & 0  & \omega^{\mp 1} \end{pmatrix}
		E(z)^{-1} \\ 
	+ (\tfrac{a_1^2}{2} N^2 - a_2 N) z^{-2/3}
		E(z) \begin{pmatrix} 1 & 0 & 0 \\
		0 & \omega^{\mp 1} & 0 \\
		0 & 0  & \omega^{\pm 1} \end{pmatrix}
	E(z)^{-1}  \\
	(-\tfrac{a_1^3}{6} N^3 + a_1 a_2 N^2 -a_3 N) z^{-1} I_3 	+ \cdots 
\end{multline}

Using the expansions \eqref{Easymp} and \eqref{Einvasymp}
for $E(z)$ and $E(z)^{-1}$, one finds that each term
in \eqref{EgEinvexpansion} has a Laurent
series around infinity (i.e., all fractional exponents disappear). 
The terms with $z^{-1/3}$ and $z^{-2/3}$
in \eqref{EgEinvexpansion} contribute also to leading order since
\begin{align}  \label{L1asymp}
	z^{-1/3} E(z) \begin{pmatrix} 1 & 0 & 0 \\
	0 & \omega^{\pm 1} & 0 \\
	0 & 0  & \omega^{\mp 1} \end{pmatrix}
E(z)^{-1} 
	& = L_1 + \mathcal{O}(z^{-1})  \\ \label{L2asymp}
	z^{-2/3} E(z) \begin{pmatrix} 1 & 0 & 0 \\
		0 & \omega^{\pm 1} & 0 \\
		0 & 0  & \omega^{\mp 1} \end{pmatrix}
	E(z)^{-1} 
	& = L_2 + \mathcal{O}(z^{-1})  \end{align}
with strictly lower triangular matrices
\begin{align} \label{L1L2def}
	L_1 & = 
	\begin{pmatrix} 0 &  0 & 0 \\ \frac{w_{21}}{c_\lambda} & 0 & 0 \\ 
		\frac{\frac{2}{3} c_\lambda^3 - w_{12} w_{21}}{w_{13} c_\lambda}
		- \frac{w_{23} c_\lambda^2}{w_{13}^2w_{21}} & \frac{c_\lambda^2}{w_{13}w_{21}} & 0  \end{pmatrix} &
	L_2 & = L_1^2 = \begin{pmatrix} 0 & 0 & 0 \\ 0 & 0 & 0 \\ \frac{c_\lambda}{w_{13}} & 0 & 0  \end{pmatrix}.
	\end{align}
All other terms in \eqref{EgEinvexpansion} are
$\mathcal O(z^{-1})$ as $z \to \infty$.

Then
\begin{equation} \label{Ldef}
	L = \exp(-a_1 L_1 - a_2 L_1^2)\end{equation}
is unit lower triangular.  
It follows that \eqref{EgEinvexpansion} can be written as
\begin{multline} \label{EgEinvexpansion2}
	z^N E(z) \diag\left(e^{-Ng_1(z)},  e^{-Ng_2(z)}, e^{-Ng_3(z)} \right) E(z)^{-1} \\
	= I_3 - a_1N L_1 + \frac{a_1^2}{2} N^2 L_1^2
		- a_2 N L_1^2 + \mathcal{O}(z^{-1}) \\
	= 	L^N + \mathcal{O}(z^{-1}) \quad \text{ as } z \to \infty.
\end{multline}
Similarly (we can just replace $N$ by $-N$), 
	\begin{equation}  z^{-N} E(z) \diag\left(e^{Ng_1(z)}, e^{Ng_2(z)},  e^{Ng_3(z)} \right) E(z)^{-1} 
		= L^{-N} 	
		+ \mathcal{O}(z^{-1}).
	\end{equation}
Then \eqref{Tasymp} leads to \eqref{Tasymp2} and the lemma follows.
\end{proof}

\section{Third transformation $T \mapsto S$}

The opening of lenses $T \mapsto S$ is standard.  We open an
annular lens around $\mathbb T$, that
we choose to be bounded by the circles $|z| = 1 \pm \eta$ for some small $\eta >0$. We take it such that
\[ z_{\min} < 1 - \eta < 1 + \eta < z_{\max}. \]

\begin{definition} \label{TtoS}
	We define
	\begin{align} \nonumber
		S & =  
		\begin{pmatrix} L^{-N} & 0_3 \\
			0_3 & L^N \end{pmatrix} T \\  \label{Sdef}
			& \quad
			\times \begin{cases} 
			\Pi \left[ \begin{pmatrix} 1 & 0 \\ (-1)^N e^{2N\varphi_1} & 1 \end{pmatrix} \oplus
			\begin{pmatrix} 1 & 0 \\ (-1)^N e^{2N\varphi_2} & 1 \end{pmatrix}  \oplus I_2  \right] \Pi^{-1}, \\
			\hfill\text{ for } 1-\eta < |z| < 1, \, \Im z \neq 0, \\
		\Pi \left[ \begin{pmatrix} 1 & 0 \\ -(-1)^N e^{2N\varphi_1} & 1 \end{pmatrix} \oplus
		\begin{pmatrix} 1 & 0 \\ -(-1)^N e^{2N\varphi_2} & 1 \end{pmatrix}  \oplus I_2  \right] \Pi^{-1}, \\
		 \hfill\text{ for } 1 < |z| < 1+\eta, \, \Im z \neq 0, \\
			I_6 \quad \text{ elsewhere in $\mathbb C \setminus \Sigma_T$.}
			\end{cases} \end{align}
\end{definition}

Note that we also premultiplied $T$ throughout
with the constant matrix $\begin{pmatrix} L^{-N} & 0_3 \\
	0_3 & L^N \end{pmatrix}$. This does not
	affect any of the jumps, but it will simplify
	the asymptotic behavior of $S$ due to  Lemma \ref{lemma67}.

\begin{figure}
	\centering
	\begin{tikzpicture}[scale=3]
		\begin{scope}[ultra thick,decoration={
				markings,
				mark=at position 0.5 with {\arrow{>}}}
			] 
			
			\draw [postaction={decorate}] (-1.5,0)--(0.6,0);	
			\draw (0.6,0)--(1.1,0);
			\draw [postaction={decorate}] (1.1,0)--(1.5,0);
			
			\filldraw (0.6,0)  circle (0.7pt);	
			\filldraw (-0.85,0) circle (0.7pt);
			\filldraw  (1.1,0) circle (0.7pt);
			\filldraw (0.85,0) circle (0.7pt);
			\draw (0.6,0) node [below]{$z_{\min}$};
			\draw (1.1,0) node [below]{$z_{\max}$};
			\draw (-0.93,0) node [below]{$-1$};
			\draw (0.88,0) node [left, below]{$1$};
			\draw (0,0.74)  node[above]{$\mathbb T$};
			\coordinate (a) at (0.85,0);
			\coordinate (b) at (0,0.8);
			\coordinate (c) at (-0.85,0);
			\coordinate (d) at (0,-0.8);
			\coordinate (e) at (0.85,0);             
			\path[draw,use Hobby shortcut,closed=true]  (a)..(b)..(c)..(d)..(e); 
			\draw (0,0.5)  node[above]{$|z|=1-\eta$};
			\coordinate (a) at (0.75,0);
			\coordinate (b) at (0,0.7);
			\coordinate (c) at (-0.75,0);
			\coordinate (d) at (0,-0.7);
			\coordinate (e) at (0.75,0);               
			\path[draw,use Hobby shortcut,closed=true]  (a)..(b)..(c)..(d)..(e); 	
			\draw (0,0.88)  node[above]{$|z|=1+\eta$};
			\coordinate (a) at (0.95,0);
			\coordinate (b) at (0,0.9);
			\coordinate (c) at (-0.95,0);
			\coordinate (d) at (0,-0.9);
			\coordinate (e) at (0.95,0);               
			\path[draw,use Hobby shortcut,closed=true]  (a)..(b)..(c)..(d)..(e); 
			\draw [postaction={decorate}] (0,-0.7) node{};
			\draw [postaction={decorate}] (0,-0.8) node{};
			\draw [postaction={decorate}] (0,-0.9) node{};
		\end{scope}
	\end{tikzpicture}
	
	\caption{Contour $\Sigma_S = \mathbb R \cup \{ |z|  = 1\} \cup \{|z| = 1\pm \eta \}$ for the RH problem \ref{rhpforS} for $S$.
		\label{fig:SigmaS}} 
\end{figure}
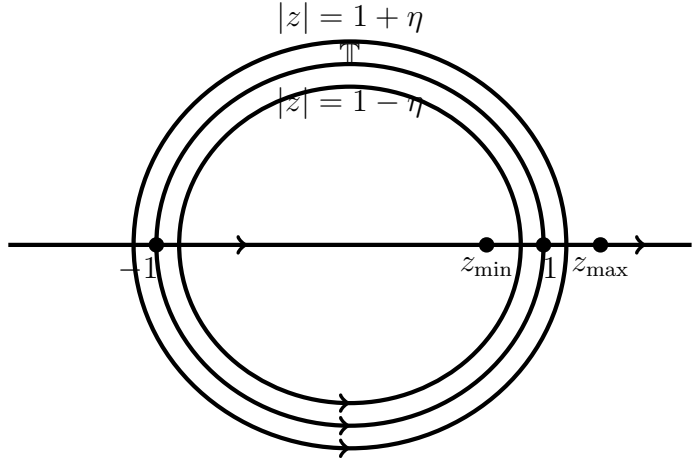

%Then $S$ satisfies the following.

\begin{rhproblem} \label{rhpforS} $S$ satisfies the following
	RH problem.
	\begin{description} \item[RHP-S1]
		$S : \mathbb C \setminus \Sigma_S \to \mathbb C^{6 \times 6}$ 
		is analytic, where $\Sigma_S = \mathbb R \cup \mathbb T \cup \{ |z| = 1 \pm \eta \}$. 
		\item[RHP-S2] $S_+ = S_-J_S$ on $\Sigma_S$ where
		\begin{align} \label{Sjump1} 
		J_S & = \Pi \left[ \begin{pmatrix} 0 & (-1)^N \\ -(-1)^N & 0 \end{pmatrix} \oplus \begin{pmatrix} 0 & (-1)^N \\ -(-1)^N & 0 \end{pmatrix} \oplus \begin{pmatrix} 1 & (-1)^N e^{-2N \varphi_3} \\
			0 & 1 \end{pmatrix} \right] \Pi^{-1} 
		\end{align} on $\mathbb T$, and
		\begin{align}
		 \label{Sjump2} 
		J_S & = \Pi \left[ \begin{pmatrix} 1 & 0 \\ (-1)^N e^{2N\varphi_1} & 1 \end{pmatrix} \oplus \begin{pmatrix} 1 & 0 \\ (-1)^N e^{2N\varphi_2} & 1 \end{pmatrix} \oplus I_3 \right] \Pi^{-1},
		\quad |z| = 1 \pm \eta, \\
			 \label{Sjump3}
		J_S & = 
			J_E \oplus J_E \quad \text{ on } (-\infty,z_{\min}) \cup (z_{\max},\infty) \\
		J_S & = \label{Sjump4}
			\diag\left(1,(-1)^N, (-1)^N, 1, (-1)^N, (-1)^N \right)
			\quad \text{ on } [z_{\min}, z_{\max}].  \end{align}
		\item[RHP-S3] As $z \to \infty$  
		\begin{align} \label{Sasymp}	
			S(z) = \left(I_6 + \mathcal{O}(z^{-1})\right)
		\begin{pmatrix} E(z) & 0_3 \\ 0_3 &  E(z) \end{pmatrix}.
		\end{align}
	\item[RHP-S4] $S \begin{pmatrix} E^{-1} & 0_3 \\ 0_3 & E^{-1} \end{pmatrix}$ 
	remains bounded near $0$, $z_{\min}$, $z_{\max}$, and $z^*$.
	\end{description} 
\end{rhproblem}

\begin{proof}
	{\bf RHP-S1} 
	The analyticity of $S$ is clear.
	
	{\bf RHP-S2}  
	The jump matrix \eqref{Sjump1} on $\mathbb T$ comes from the following factorization of the $2 \times 2$ matrices that appear
	in the jump matrix \eqref{Tjump1}. For $j=1,2$ we have
\[ \begin{pmatrix} e^{2N \varphi_{j+}} & (-1)^N \\
		0 & e^{2N \varphi_{j-}} \end{pmatrix}
		= \begin{pmatrix} 1 & 0 \\
			(-1)^N e^{2N \varphi_{j-}}	& 1 \end{pmatrix}  
		\begin{pmatrix} 0 & (-1)^N \\ -(-1)^N & 0 \end{pmatrix}
		\begin{pmatrix} 1 & 0 \\
			(-1)^N e^{2N \varphi_{j+}}	& 1 \end{pmatrix}.
		\]	
	The jump matrix \eqref{Sjump2} on $|z| = 1\pm \eta$ are immediate
	from the definition \eqref{Sdef}, while the jump matrices on
	$\mathbb R$ remain unchanged if one goes from $T$ to $S$.
	This requires a little calculation on the
	interval $(-1-\eta, -1 + \eta)$ that is inside
	the annulus $1-\eta < |z| < 1+\eta$. 
	Here one has to use the  
	property that $\varphi_{1,\pm} - \varphi_{2,\mp} \in 2\pi i \mathbb Z$, which follows from the definition \eqref{phijdef}
	and the fact that $\widehat{g}_{1,\pm} = \widehat{g}_{2,\mp}$ and
	$\lambda_{1,\pm} = \lambda_{2,\mp}$ there.

	{\bf RHP-S3}  The asymptotic condition \eqref{Sasymp} follows
	from \eqref{Tasymp2} and \eqref{Sdef}.
	
	{\bf RHP-S4} 
	The statement about the boundedness near $0, z_{\min}$ and $z_{\max}$ is immediate from {\bf RHP-T4},
	as these points are outside the lens.
	The boundedness near $z^*$ is also immediate if $z^*$ is outside the lens. If $z^*$ is inside the lens,
	then given {\bf RHP-T4} and the definition \eqref{Sdef},
	it comes down to proving that
	$E 
		\diag \left(e^{2N \varphi_1}, e^{2N \varphi_2}, 0
		\right)  E^{-1}$ 
		remains bounded near $z^*$. 
	This follows in the same way as we showed
	the boundedness near $z^*$ in {\bf RHP-T4} in
	the RH problem for $T$, since $\varphi_1$ and $\varphi_2$
	remain bounded.
\end{proof}

\begin{lemma} \label{lemma73} 
	For $\eta > 0$ sufficiently small, we can find
	a constant $c > 0$ such that 
	$2\Re \varphi_1 < -c$ and $2\Re \varphi_2  < -c$ on
	the circles $|z| = 1 \pm \eta$.
\end{lemma}

\begin{proof}
	By \eqref{phijdef} and $\Re \widehat{\ell} = \ell$, we have for $j=1,2$,
	\[ \Re \varphi_j = -3 \Re \widehat{g}_j + \frac{\Re V}{2} - \ell. \]
	This is the restriction to the $j$th sheet of the function $h$
	that was used in the proof of Theorem \ref{theorem13},
	see \eqref{hdef}. It was shown at the end of the proof
	of Theorem \ref{theorem13} (b), that $h(p) < 0$ 
	if $0 < \Re \mathcal A(p) < \frac{1}{3}$ or $\frac{2}{3} < \Re \mathcal A(p) < 1$, where $\mathcal A$ is the Abel map.  
	
	We choose $\eta$ small enough so that points $p$ on the first and second
	sheets with $1 < |z(p)| \leq 1 + \eta$ are mapped by the Abel map into
	$0 < \Re u < \frac{1}{3}$ and points on the first and second
	sheets with $1- \eta \leq |z(p)| < 1$ ar mapped into $\frac{2}{3} < \Re u < 1$,
	and in both cases $h(p) < 0$. Then by compactness, we can find
	a constant $c >0$ as in the lemma.
\end{proof}

\section{Global parametrix} \label{section8}

There is no need for local parametrices. But we do need to find a
global parametrix.

\subsection{Statement of the RH problem} \label{section81} 
We ignore the exponentially small
entries in the jump matrices \eqref{Sjump1}-\eqref{Sjump2} for $S$ and we look for $M$ satisfying
the following RH problem.

\begin{rhproblem} \label{rhpforM} \
	\begin{description} \item[RHP-M1]
		$M : \mathbb C \setminus \Sigma_M \to \mathbb C^{6 \times 6}$ 
		is analytic, where $\Sigma_M = \mathbb R \cup \mathbb T$.
		\item[RHP-M2] $M_+ = M_- J_M$ on $\Sigma_M$ where
		\begin{align} \label{Mjump1} 
		J_M & = \Pi \left[ \begin{pmatrix} 0 & (-1)^N \\ -(-1)^N & 0 \end{pmatrix} \oplus \begin{pmatrix} 0 & (-1)^N \\ -(-1)^N & 0 \end{pmatrix} \oplus I_2 \right] \Pi^{-1}, \quad \text{on } \mathbb T,  \\
			\label{Mjump2}
			J_M & =
			J_S = \begin{cases} J_E \oplus J_E, 
				& \hspace*{-2.2cm} \text{on } (-\infty,z_{\min}) \cup (z_{\max},\infty), \\
			\diag (1, (-1)^N, (-1)^N, 1, (-1)^N, (-1)^N),
			& \text{on } (z_{\min}, z_{\max}).
			\end{cases}
		\end{align}
		
		\item[RHP-M3] As $z \to \infty$  
		\begin{align} \label{Masymp}
			M(z) = \left(I_6 + \mathcal{O}(z^{-1}) \right)
			\begin{pmatrix} E(z) & 0_3 \\
				0_3 & E(z)  \end{pmatrix}
	 \end{align}	 
	 \item[RHP-M4] $M \begin{pmatrix} E^{-1} & 0_3 \\ 0_3 & E^{-1} \end{pmatrix}$ remains bounded near $0$, $z_{\min}$, $z_{\max}$ and $z^*$.	
		\end{description} 
\end{rhproblem}

Let us show that the solution is unique (if it exists).
Suppose $M$ solves the RH problem \ref{rhpforM} and put
\begin{equation} \label{Mhatdef} 
	\widehat{M} = M \begin{pmatrix} E^{-1} & 0_3 \\ 0_3 & E^{-1} \end{pmatrix} \end{equation}
Then it is not difficult to show that
$\widehat{M}$ is a solution of the following RH problem.

\begin{rhproblem} \label{rhpforMhat}
	$\widehat{M}$ satisfies the following.
	\begin{description} \item[RHP-$\bf \widehat{M}$1]
		$\widehat{M} : \mathbb C \setminus (\mathbb T \cup [z_{\min},z_{\max}])$ is analytic,
		\item[RHP-$\bf \widehat{M}$2] $\widehat{M}_+ = \widehat{M}_- 
		\begin{pmatrix} E  & 0_3 \\ 0_3 & E \end{pmatrix}
		J_M \begin{pmatrix} E^{-1} & 0_3 \\ 0_3 & E^{-1}
			\end{pmatrix} $
			on $\mathbb T \cup [z_{\min},z_{\max}]$,
		\item[RHP-$\bf \widehat{M}$3] $\widehat{M}(z) = I_6 + \mathcal{O}(z^{-1})$
		as $z \to \infty$.  
	\end{description} 
\end{rhproblem}
\begin{proof}
	$M$ and $\begin{pmatrix} E & 0_3 \\ 0_3 & E_3 \end{pmatrix}$
	have the same jumps on $(-\infty,0) \cup (0,z_{\min}) \cup (z_{\max}, \infty)$, and therefore 
	$\widehat{M}$ has  analytic continuation 
	across these real intervals.
	Because of {\bf RHP-M4} the singularity at $0$ is
	removable and {\bf RHP-$\bf \widehat{M}$1} follows.
	 	
	The jump condition  {\bf RHP-$\bf \widehat{M}$2} on $\mathbb T$ is immediate from
	\eqref{Mhatdef}, since $E$ and $E^{-1}$ are analytic across $\mathbb T$.
	The jump condition on $(z_{\min}, z_{\max}) \setminus \{z^*\}$ is also immediate. From {\bf RHP-M4} we have that 
	$\widehat{M}$ remains bounded near the special
	points $z_{\min}, z_{\max}, z^*$, and we 
	conclude that
	the jump condition extends to these points as well.
	The asymptotic condition {\bf RHP-$\bf \widehat{M}$3}
	is immediate from {\bf RHP-M3} and the definition
	\eqref{Mhatdef}.
\end{proof}

The RH problem \ref{rhpforMhat} is normalized at infinity, with a jump matrix whose
determinant is identically one. By standard arguments
\cite{Dei99}, one concludes that $\det \widehat{M} \equiv 1$, hence $\widehat{M}$ is invertible. From \eqref{Mhatdef}
we then also get
\begin{equation} \label{Mdet} 
	\det M = (\det E)^2 \quad \text{ on } \mathbb C \setminus (\mathbb T \cup [z_{\min},z_{\max}]). 
	\end{equation}
Also for any other solution of RH problem \ref{rhpforMhat}, 
say $\widetilde{M}$, one can put $H =  \widetilde{M} \left(\widehat{M}\right)^{-1}$
and show that $H$ extends to an entire function such
that $H \to I_6$ at infinity, and then $H = I_6$
by Liouville's theorem. Hence the solution
of the RH problem for $\widehat{M}$ is unique,
and then  the solution to the RH problem \ref{rhpforM}
for $M$ is unique as well. We also conclude that
$ M \begin{pmatrix} E^{-1} & 0_3 \\ 0_3 & E^{-1}\end{pmatrix}$
not only remains bounded near $z^*$, but it has an
analytic continuation to $z^*$ with determinant equal to $1$.

Due to the entries $(-1)^N$ in the jump matrices \eqref{Mjump1}
and \eqref{Mjump2},
the  global parametrix depends on the parity of $N$.
There will be $M_e$ for $N$ even, and $M_o$ for $N$ odd.

\subsection{Solution strategy}

We look for $M$ in the block form
\begin{equation} \label{Mansatz}
	M = K
	\begin{pmatrix} E \circ \Psi & -E \circ \Phi \\
		E \circ \Phi & E \circ \Psi \end{pmatrix} 
	\end{equation}
	where $\circ$ denotes the Hadamard (i.e., entrywise) product of matrices. $K$ is a constant invertible matrix,
	and $\Phi$ and $\Psi$ are  matrix valued
	functions 
on $\mathbb C \setminus \Sigma_M$ of size
$3 \times 3$ with the property that
\begin{equation} \label{Psijump1}  \begin{pmatrix} \Psi & -\Phi \\
		\Phi & \Psi \end{pmatrix}_+
	=\begin{pmatrix} \Psi & -\Phi \\
		\Phi & \Psi \end{pmatrix}_- J_M 
		\quad \text{ on } \mathbb R \cup \mathbb T,
\end{equation}
Then, by the special form \eqref{Mansatz} and taking
note of the jump matrices \eqref{Ejump}, \eqref{Mjump1}, \eqref{Mjump2}, and \eqref{Psijump1}, 
we will have that $M_+ = M_- J_M$.

To be able to obtain the asymptotic condition \eqref{Masymp} we 
impose that
\begin{equation} \label{Psiasymp} 
		\Psi(z) = \begin{pmatrix} 1 & 1 & 1 \\
			1 & 1 & 1 \\
			1 & 1 & 1 \end{pmatrix} + \mathcal{O}(z^{-1/3})
		 \quad \text{ and } \quad \Phi(z) = \mathcal{O}(z^{-1/3}), \quad \text{ as } z \to \infty.
		\end{equation}
		
We cannot solve \eqref{Psijump1}--\eqref{Psiasymp} with
analytic $\Psi$ and $\Phi$. We have to allow certain poles.
Recall from Lemma \ref{lemma53} that one of the columns of $E$ vanishes 
identically at $z^*$.
For $j=1,2,3$, we also have $z_j^* = z(Q_j^*)$, see
Lemma \ref{lemma53}(d), and one of the entries of $E$
in row $j$, in either the second or third column, vanishes at $z_j^*$. Thus there is $k \in \{2,3\}$ such that $E_{j,k}(z_j^*)$. We supplement \eqref{Psijump1}--\eqref{Psiasymp}
with the following condition. 
For each $j=1,2,3$, and $k=2,3$ we have
\begin{equation} \label{Psipoles}
	\text{if $E_{j,k}(z_j^*) = 0$ then $\Psi_{j,k}$ and $\Phi_{j,k}$ have simple poles
	at $z_j^*$}.
\end{equation}
Under the condition \eqref{Psipoles} there
is a zero/pole cancellation at $z_j^*$ 
in the products $E_{j,k} \Psi_{j,k}$ and $E_{j,k} \Phi_{j,k}$
which means that the Hadamard products in \eqref{Mansatz}
have no poles, and remain bounded near each $z_j^*$.

\begin{remark}
	For $j =2,3$, we also have the option to put
	the simple poles at $0$, since $E_{j,k}$ vanishes
	at $0$ for every $k=1,2,3$. This would lead
	to different $\Psi$ and $\Phi$ matrices.
	Also the constant matrix $K$ in \eqref{Mansatz} would
	be different, but the result for $M$ would be the same.
\end{remark}

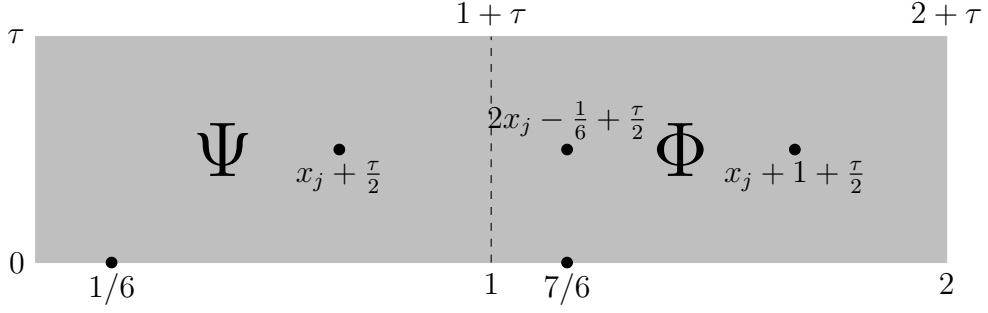
\begin{figure}[t]
	\begin{center}
		\begin{tikzpicture}[scale=0.5](15,10)(0,0)
			% rectangle
			\filldraw[gray!50!white] (-6,0) --(6,0) --(6,6) --(-6,6) --(-6,0);
			%\draw[very thick, black] (-6,0) --(6,0) --(6,6) --(-6,6) --(-6,0);	
			\filldraw[gray!50!white] (6,0) --(18,0) --(18,6) --(6,6) --(6,0);
			
			\draw (0,3) node[left] {$\scaleto{\Psi}{20pt}$};
			\draw (12,3) node[left] {$\scaleto{\Phi}{20pt}$};
			
			\draw (-6,0) node[left] {$0$};
			\draw (6,0) node[below] {$1$};
			\draw (-6,6) node[left] {$\tau$};
			\draw (6,6) node[above] {$1+\tau$};
			
			\draw (18,0) node[below] {$2$};
			\draw (18,6) node[above] {$2+\tau$};
				
			% images of special points
			\filldraw (-4,0)  circle (4pt);	 
			\filldraw (8,0)  circle (4pt);
			\draw (-4,0) node[below] {$1/6$};
			\draw (8,0) node[below] {$7/6$};
			
			% images of zeros and poles
			\filldraw (2,3) circle (4pt);
			\filldraw (14,3) circle (4pt);
			\draw (2,3) node[below] {$x_j + \frac{\tau}{2}$};
			\draw (14,3) node[below] {$x_j +1 + \frac{\tau}{2}$};
			
			\filldraw(8,3) circle (4pt);
			\draw (8,3) node[above] {$2x_j - \frac{1}{6} + \frac{\tau}{2}$};

			% image of circles
		%	\draw[dashed,black] (0,0)--(0,6);
		%	\draw[dashed,black] (-6,0)--(-6,6);
			\draw[dashed,black] (6,0)--(6,6);
		\end{tikzpicture}
	\end{center}
	\caption{Construction of $\Psi$ and $\Phi$ in 
		case $N$ is even. 
		For $j=1,2,3$ there is a meromorphic function $f_j$ on the rectangle $(0,2) \times (0,\tau)$
		with poles at $x_j+\frac{\tau}{2}$ and $x_j+1+\frac{\tau}{2}$, zeros at $\frac{7}{6}$
		and $2x_j - \frac{1}{6} + \frac{\tau}{2}$,
		and normalization such that $f_j(1/6) = 1$.
		The $j$th row of $\Psi$ is filled with the
		$f_j$ values in $(0,1) \times (0,\tau)$
		and the $j$th row of $\Phi$ is filled with
		the $f_j$ values in $(1,2) \times (0,\tau)$.
		 \label{fig:PsiPhieven}}
\end{figure}

\subsection{Construction of $\Psi$ and $\Phi$ (case $N$ is even)}

We construct $\Psi$ and $\Phi$ row by row. 
The $j$th rows come from two functions
come from two functions $\psi_j$ and $\phi_j$ on
the Riemann surface
with the properties  listed in the following lemma.
We assume $N$ is even.
\begin{lemma}	\label{lemma84}
	For each $j=1,2,3$, there is a unique pair of functions
	$(\psi_j, \phi_j)$ with the following properties
	\begin{enumerate}
		\item[\rm (a)] $\psi_j$ and $\phi_j$ are meromorphic
		on $\mathcal R \setminus \mathbf{b}$ with
		\begin{align} \label{psiphijump}
		\begin{array}{lll} \psi_{j,+}  = \phi_{j,-}, &  
			\phi_{j,+}  = - \psi_{j,-}, & \text{ on } 
			\mathbf{b}, %\\
		%	\psi_{j,+}  = (-1)^N \psi_{j,-}, &
		%	\phi_{j,+}  = (-1)^N \phi_{j,-}, & \text{ on } \mathbf{c}.
			\end{array}
	 \end{align}
		\item[\rm (b)] $\psi_j$ and $\phi_j$ have a simple
		pole at $Q_j^*$ on the bounded oval, and 
		no other poles,
		\item[\rm (c)] $\psi_j(P_\infty) = 1$ and $\phi_j(P_{\infty}) = 0$.
	\end{enumerate}
\end{lemma}
%Regarding item (b) we recall that $z_j^*$ is the $z$-coordinate
%of $Q_j^*$ and $Q_j^*$ is on the $k$th sheet
%if and only if $E_{j,k}(z_j^*) = 0$.

\begin{proof}
	We write $\mathcal A(Q_j^*) = x_j + \frac{\tau}{2}$ 
	with $0 \leq x_j < 1$.
Then we start from the function $f_j$ given as the following
	product and ratio of Jacobi theta functions
	\begin{multline} \label{fjeven} f_j(u) = 
	\frac{\theta_1 \left( \frac{\tau'}{2} (\frac{1}{6}-x_j - \frac{\tau}{2})  \mid \tau'\right) \
		\theta_1\left(\frac{\tau'}{2} (-\frac{5}{6}- x_j - \frac{\tau}{2}) 
		\mid \tau'\right)}
		{\theta_1\left(- \frac{\tau'}{2}   \mid \tau'\right) 
		\theta_1\left(\frac{\tau'}{2} (\frac{1}{3}- 2x_j - \frac{\tau}{2}) \mid \tau'\right)
		} \\ 
		\times
		\frac{\theta_1\left( \frac{\tau'}{2} (u-\frac{7}{6})  \mid \tau'\right) \
			\theta_1(\frac{\tau'}{2} (u- 2x_j + \frac{1}{6} - \frac{\tau}{2}) 
			\mid \tau')}
		{\theta_1\left( \frac{\tau'}{2} (u-x_j - \frac{\tau}{2})  \mid \tau'\right) \
			\theta_1\left(\frac{\tau'}{2} (u- x_j-1 - \frac{\tau}{2}) 
			\mid \tau'\right)},
		\quad \tau' = - \frac{2}{\tau}. 
\end{multline}
From the quasi-periodicity properties \eqref{theta1period}
we then get
\begin{equation} \label{fjperiod}
	f_j(u+2) = - f_j(u), \qquad f_j(u+\tau) = f_j(u).
\end{equation}
Hence $f_j$ is quasi-periodic with respect to the
lattice $2\mathbb Z + \tau \mathbb Z$.
It has simple zeros in $\frac{7}{6}$ and
$2x_j- \frac{1}{6}$, simple poles in $x_j + \frac{\tau}{2} = \mathcal A(Q_j^*)$,
and $x_j+1 + \frac{\tau}{2} = \mathcal A(Q_j^*) + 1$ 
and no other zeros and poles (modulo $2 \mathbb Z + \tau \mathbb Z$). The prefactor on the first line of \eqref{fjeven} 
is such that $f_j(\frac{1}{6}) = 1$.

Then we define $\psi_j$ and $\phi_j$ on the Riemann surface by
\begin{align} \label{psijdef}
	\psi_j(p) = f_j (\mathcal A(p)), \quad
	\phi_j(p)  = f_j (\mathcal A(p) + 1), 
\end{align}
where $\mathcal A$ is the Abel map, as before, with 
values in $(0,1) \times (0,\tau)$. 
Thus $\psi_j$ takes the values of $f_j$ from the
left half of its fundamental domain $[0,2] \times [0,\tau]$,
and $\phi_j$ takes the values from the right half,
see also Figure \ref{fig:PsiPhieven}.

The items (a), (b), (c) of the lemma now follow
by easy verification. 
\end{proof}

\begin{definition} \label{definition85}
	With $\psi_j$ and $\phi_j$ as in Lemma \ref{lemma84}, 
	we define in case $N$ is even
	\begin{equation} \label{PsiPhidef}
		\Psi := (\psi_{j,k})_{j,k=1}^3, \qquad \Phi := (\phi_{j,k})_{j,k=1}^3, 
	\end{equation}
	where $\psi_{j,k}$ denotes the restriction of $\psi_j$
	to the $k$th sheet of the Riemann surface,
	and similarly, $\phi_{j,k}$ denotes the
	restriction of $\phi_j$ to the $k$th sheet.
\end{definition}

\begin{corollary} \label{corollary86}
	If $N$ is even, then 
 $\Psi$ and $\Phi$ defined in \eqref{PsiPhidef} satisfy \eqref{Psijump1}, \eqref{Psiasymp}, and \eqref{Psipoles}. 	\end{corollary} 
 
 \begin{proof}
 In view of \eqref{Mjump1} with $N$ even, the jump condition \eqref{Psijump1} on $\mathbb T$	is satisfied if
 and only if for $j=1,2,3$,
 \begin{align} \label{PsiPhijump1}
 	\Psi_{jk,+} =  \Phi_{jk,-},  \quad & \Phi_{jk,+} = -\Psi_{jk,-}, \quad
 k=1,2,  \\ \label{PsiPhijump2}
 	\Psi_{j3,+} = \Psi_{j3,-},\quad & \Phi_{j3,+} = \Phi_{j3,-},
 	\end{align}
 on $\mathbb T$.
 The identities \eqref{PsiPhijump1} follow from
 \eqref{psiphijump}  and the definition \eqref{PsiPhidef}
 since $\mathbf{b} = \Gamma_1 \cup \Gamma_2$. 
 The identities \eqref{PsiPhijump2} come from the fact
 that $\psi_j$ and $\phi_j$ in Lemma \ref{lemma84} are
 analytic across $\Gamma_3$.
 
 On $\mathbb R$, we have from \eqref{Mjump2}
 with $N$ even, that the jump condition \eqref{Psijump1} on $\mathbb R$ is satisfied if and only if 
 \begin{equation} \label{PsiPhijump3} 
 	\Psi_+  = \Psi_- J_E , \quad \Phi_+ = \Phi_- J_E 
 	\quad \text{ on } \mathbb R. \end{equation}
 Recall that $J_E$ is given by \eqref{Ejump}.
 With the definition \eqref{PsiPhidef} these identities
 are indeed satisfied, as they reflect the fact that
 $\psi_j$ and $\phi_j$ from Lemma \ref{lemma84} 
 are analytic across the real ovals. 
 This implies, for example, that $\psi_{j,1,\pm} = \psi_{j,2,\mp}$
 and $\psi_{j,3,+} = \psi_{j,3,-}$ on $(-\infty,0)$,
 and similarly for $\phi_j$, which gives the
 jump properties \eqref{PsiPhijump3} on $(-\infty,0)$.   
 
 The asymptotic property \eqref{Psiasymp} follows
 from the normalization in part (c) of Lemma \ref{lemma84},
 together with the fact that the $\psi_j$ and $\phi_j$
 are holomorphic at $P_{\infty}$.
 Since $z^{-1/3}$ is a local
 coordinate at $P_{\infty}$ we obtain the error terms
 $\mathcal{O}(z^{-1/3})$ in \eqref{Psiasymp}. 

Finally, the property \eqref{Psipoles} follows
from part (b) of Lemma \ref{lemma84}. Indeed, $\psi_j$
and $\phi_j$ have a simple pole at $Q_j^*$ on the bounded
oval. Thus $Q_j^*$ is either on the second or the
third sheet. If $Q_j^*$ is on the $k$th sheet
then $\psi_{j,k}$ and $\phi_{j,k}$ have simple
poles at $z_j^* = z(Q_j^*)$, and this corresponds
to the property \eqref{Psipoles}.  
 \end{proof}

 \begin{figure}[t]
 	\begin{center}
 		\begin{tikzpicture}[scale=0.5](15,10)(0,0)
 			% rectangle
 			\filldraw[gray!50!white] (-6,0) --(6,0) --(6,6) --(-6,6) --(-6,0);
 			%\draw[very thick, black] (-6,0) --(6,0) --(6,6) --(-6,6) --(-6,0);	
 			\filldraw[gray!50!white] (6,0) --(18,0) --(18,6) --(6,6) --(6,0);
 			
 			\draw (0,3) node[right] {$\scaleto{\Phi}{20pt}$};
 			\draw (12,3) node[right] {$\scaleto{\Psi}{20pt}$};
 			
 			\draw (-6,0) node[left] {$-\frac{\tau}{2}$};
 			\draw (-6,3) node[left] {$0$};
 			\draw (-6,6) node[left] {$\frac{\tau}{2}$};
 			\draw (6,6) node[above] {$1+\frac{\tau}{2}$};
 			\draw (6,0) node[below] {$1-\frac{\tau}{2}$};
 			
 			\draw (18,0) node[below] {$2- \frac{\tau}{2}$};
 			\draw (18,6) node[above] {$2+ \frac{\tau}{2}$};
 			
 			% images of special points
 			\filldraw (-4,3)  circle (4pt);	 
 			\filldraw (8,3)  circle (4pt);
 			\draw (-4,3) node[below] {$1/6$};
 			\draw (8,3) node[below] {$7/6$};
 			
 			% images of zeros and poles
 			\filldraw (2,6) circle (4pt);
 			\filldraw (14,6) circle (4pt);
 			\draw (2,6) node[above] {$x_j + \frac{\tau}{2}$};
 			\draw (14,6) node[above] {$x_j +1 + \frac{\tau}{2}$};
 			
 			\filldraw(8,6) circle (4pt);
 			\draw (8,6) node[below] {$2x_j - \frac{1}{6} + \frac{\tau}{2}$};

 			% image of circles
 			%	\draw[dashed,black] (0,0)--(0,6);
 			%	\draw[dashed,black] (-6,0)--(-6,6);
 			\draw[dashed,black] (6,0)--(6,6);
 		\end{tikzpicture}
 	\end{center}
 	\caption{Construction of $\Psi$ and $\Phi$ in 
 		case $N$ is odd. 
 		For $j=1,2,3$ there is a meromorphic function $f_j$ on the rectangle $(0,2) \times (-\frac{\tau}{2}, \frac{\tau}{2})$
 		with poles at $x_j+\frac{\tau}{2}$ and $x_j+1+\frac{\tau}{2}$, zeros at $\frac{1}{6}$
 		and $2x_j - \frac{5}{6} + \frac{\tau}{2}$,
 		and normalization such that $f_j(7/6) = 1$.
 		The $j$th row of $\Phi$ is filled with the
 		$f_j$ values in $(0,1) \times (-\frac{\tau}{2}, \frac{\tau}{2})$
 		and the $j$th row of $\Psi$ is filled with
 		the $f_j$ values in $(1,2) \times (-\frac{\tau}{2},\frac{\tau}{2})$.
 		\label{fig:PsiPhiodd}}
 \end{figure}
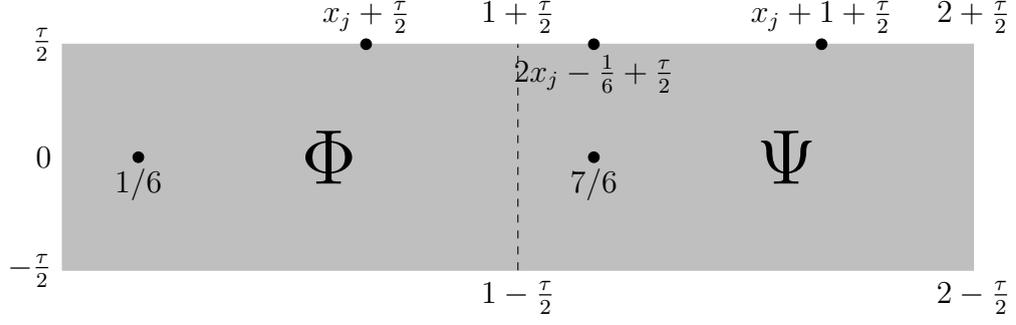
 
 \subsection{Construction of $\Phi$ and $\Psi$ (case $N$ is odd)}

The construction in the case $N$ is odd is similar, but the details are slightly different.

\begin{lemma}	\label{lemma87}
	For each $j=1,2,3$, there is a unique pair of functions
	$(\psi_j, \phi_j)$ with the following properties
	\begin{enumerate}
		\item[\rm (a)] $\psi_j$ and $\phi_j$ are meromorphic
		on $\mathcal R \setminus (\mathbf{b} \cup \mathbf{c})$,
		where $\mathbf{c}$ is the bounded oval on the Riemann surface, and
		\begin{align} \label{psiphijump2}
			\begin{array}{lll} \psi_{j,+}  = -\phi_{j,-}, &  
				\phi_{j,+}  =  \psi_{j,-}, & \text{ on } \mathbf{b}, \\
				\psi_{j,+}  = - \psi_{j,-}, &
				\phi_{j,+}  = - \phi_{j,-}, & \text{ on } \mathbf{c},
			\end{array}
		\end{align}
		\item[\rm (b)] $\psi_j$ and $\phi_j$ have a simple
		pole at $Q_j^*$ on the bounded oval, and have
		no other poles,
		\item[\rm (c)] $\psi_j(P_\infty) = 1$ and $\phi_j(P_{\infty}) = 0$.
	\end{enumerate}
\end{lemma}

\begin{proof} It is similar to the proof of Lemma \ref{lemma84}
	but instead of \eqref{fjeven} we now use 
	for $j=1,2,3$,
	\begin{multline} \label{fjodd} f_j(u) = 
		\frac{\theta_1 \left( \frac{\tau'}{2} (\frac{7}{6}-x_j - \frac{\tau}{2})  \mid \tau'\right) \
			\theta_1\left(\frac{\tau'}{2} (\frac{1}{6}- x_j - \frac{\tau}{2}) 
			\mid \tau'\right)}
		{\theta_1\left( \frac{\tau'}{2}   \mid \tau'\right) 
			\theta_1\left(\frac{\tau'}{2} (\frac{4}{3}- 2x_j - \frac{\tau}{2}) \mid \tau'\right)
		} \\ 
		\times
		\frac{\theta_1\left( \frac{\tau'}{2} (u-\frac{1}{6})  \mid \tau'\right) \
			\theta_1(\frac{\tau'}{2} (u- 2x_j + \frac{1}{6} - \frac{\tau}{2}) 
			\mid \tau')}
		{\theta_1\left( \frac{\tau'}{2} (u-x_j - \frac{\tau}{2})  \mid \tau'\right) \
			\theta_1\left(\frac{\tau'}{2} (u- x_j-1 - \frac{\tau}{2}) 
			\mid \tau'\right)},
		\quad \tau' = - \frac{2}{\tau}. 
	\end{multline}
We then have
\begin{equation} \label{fjperiododd}
	f_j(u+2) = - f_j(u), \qquad f_j(u+\tau) = -f_j(u),
\end{equation}
and $f_j$ has simple zeros $\frac{1}{6}$ and
$2x_j+ \frac{5}{6}$, simple poles in $x_j + \frac{\tau}{2} = \mathcal A(Q_j^*)$,
and $x_j+1 + \frac{\tau}{2} = \mathcal A(Q_j^*) + 1$ 
and no other zeros and poles (modulo $2 \mathbb Z + \tau \mathbb Z$).  The prefactor is such that $f_j(\frac{7}{6}) = 1$.

We use the branch of the Abel map
that is well-defind and analytic on $\mathcal R \setminus 
(\mathbf{b} \cup \mathbf{c})$
with values in $(0,1) \times (-\frac{\tau}{2}, \frac{\tau}{2})$,
and we define
\[ \phi_j(p) = f_j(\mathcal A(p)), \quad	
	\psi_j(p) =f_j(\mathcal A(p) + 1). \]
The lemma follows by straightforward verification.
Note that now $\psi_j$ takes the values of $f_j$ from
the right half of the fundamental domain, and $\phi_j$
from the left half, see Figure \ref{fig:PsiPhiodd}.
\end{proof}

In the case $N$ is odd we use the functions
$\psi_j$ and $\phi_j$ from  Lemma \ref{lemma87}.

\begin{definition} \label{definition88}
	With $\psi_j$ and $\phi_j$ as in Lemma \ref{lemma87}, 
	we define in case $N$ is odd
	\begin{equation} \label{PsiPhidef2}
		\Psi := (\psi_{j,k})_{j,k=1}^3, \qquad \Phi := (\phi_{j,k})_{j,k=1}^3, 
	\end{equation}
	where $\psi_{j,k}$ and $\phi_{j,k}$ denote the restrictions of $\psi_j$ and $\phi_j$ to the $k$th sheet of the Riemann surface. 
\end{definition}

\begin{corollary} \label{corollary89}
	If $N$ is odd, then 
	$\Psi$ and $\Phi$ defined in \eqref{PsiPhidef2}
	satisfy \eqref{Psijump1}, \eqref{Psiasymp}, and \eqref{Psipoles}. 	\end{corollary} 
\begin{proof}
This is similar to the proof of Corollary \ref{corollary86}.
The only difference comes from the entries $(-1)^N$ in the jump matrices in 
\eqref{Mjump1} and \eqref{Mjump2}.
The different signs in \eqref{psiphijump2}
on $\mathbf b$ when compared to \eqref{psiphijump}
give rise to the jump matrix \eqref{Mjump1} with $N$ odd.
The sign changes on $\mathbf{c}$ in \eqref{psiphijump2},
create the jump \eqref{Mjump2} on $(z_{\min}, z_{\max})$
that is non-trivial in case $N$ is odd. 
\end{proof}

\subsection{Construction of $M$}

\begin{proposition}
Let $\Phi$ and $\Psi$ be as in Definition \ref{definition85}
in case $N$ is even,
and as in Definition \ref{definition88} in case $N$ is odd.

Then there is a constant invertible matrix $K$ such that
\begin{equation} \label{Mdef} M = K \begin{pmatrix} E \circ \Psi & - E \circ \Phi \\
	E \circ \Phi & E \circ \Psi \end{pmatrix} \end{equation}
satisfies the RH problem \ref{rhpforM}.

\end{proposition}
\begin{proof} 
{\bf RHP-M1} Analyticity of $M$ is clear, for any choice of $K$.

{\bf RHP-M2} The jump conditions in {\bf RHP-M2} are satisfied
because of \eqref{Ejump}, \eqref{Psijump1} 
and the definition \eqref{Mdef} of $M$. Note that
the constant matrix $K$ has no influence on the jump condition.

{\bf RHP-M3} The choice of $K$ is important for the
asymptotic condition.
The functions $\psi_j$ and $\phi_j$
(from  Lemma \ref{lemma84} or Lemma \ref{lemma87}) 
are holomorphic at $P_{\infty}$ with $\psi_j(P_{\infty}) = 1$
and $\phi_j(P_{\infty}) = 0$. Since $z^{-1/3}$
is a local coordinate at $P_{\infty}$, there are expansions 
\begin{align*} 
	\psi_j & = 1 + c_{j,1} z^{-1/3} + c_{j,2} z^{-2/3} + \mathcal{O}(z^{-1}), \\
	\phi_j & =  d_{j,1} z^{-1/3} + d_{j,2} z^{-2/3} + \mathcal{O}(z^{-1}), 
	\end{align*}
	with certain coefficients $c_{1,j}, c_{2,j}, d_{1,j}, d_{2,j}$.
The restrictions to the various sheets then satisfy
\begin{align*} 
	\psi_{j,k}(z) & = 1 + c_{j,1} \omega^{k-1}  z^{-1/3} + c_{j,2} \omega^{1-k} z^{-2/3} + \mathcal{O}(z^{-1}), \\
	\phi_{j,k}(z) & =  d_{j,1} \omega^{k-1}  z^{-1/3} + d_{j,2}  \omega^{1-k} z^{-2/3} + \mathcal{O}(z^{-1}),
\end{align*}
as $z \to \infty$ with $\pm \Im z > 0$,
and due to \eqref{PsiPhidef} and \eqref{PsiPhidef2}
we have an expansion for $\Psi$ and $\Phi$.

Then, similar to the calculations that lead to
\eqref{L1asymp} and \eqref{L2asymp} in the proof of Lemma~\ref{lemma67},  one finds
\begin{align*} 
	(E \circ \Psi)(z) E(z)^{-1} 
	& = I_3 + C_1 + \mathcal{O}(z^{-1}), \\ 
	(E \circ \Phi)(z) E(z)^{-1} & =
	C_2 + \mathcal{O}(z^{-1}), \end{align*}
with  constant matrices $C_1$ and $C_2$ that are strictly lower triangular. 
Hence $M$ defined by \eqref{Mdef} satisfies
\[ M = K \left[ \begin{pmatrix} I_3 + C_1 & -C_2 \\
	C_2 & I_3 + C_1 \end{pmatrix}  + \mathcal O(z^{-1})\right]
	\quad \text{ as } z \to \infty, \]
and \eqref{Masymp} provided that $K$ is  given by
\begin{equation} \label{Kdef} 
	K = \left( I_6 +  \begin{pmatrix} C_1 & -C_2 \\ C_2 & C_1 \end{pmatrix}\right)^{-1}. \end{equation}
It is not difficult to see that the inverse matrix indeed exists, as $C_1$ and $C_2$
are strictly lower triangular. It can actually be shown that
$K$ takes the form
\[ K = I_6 + \begin{pmatrix} K_1 & K_2 \\
	-K_2 & K_1 \end{pmatrix} \]
	with certain strictly lower triangular matrices $K_1$ and $K_2$.
Note that $K$ depends on $\Psi$ and $\Phi$, and thus
on the parity of $N$. 

{\bf RHP-M4} 
	Because of \eqref{Mansatz}  we have to show that
	$(E \circ \Psi) E^{-1}$ and $(E \circ \Phi) E^{-1}$
	are  bounded near $0$, $z^*$ and $z_{\min}, z_{\max}$.
	We will prove it for 	$(E \circ \Psi) E^{-1}$
	as the proof for $(E \circ \Phi) E^{-1}$ is similar.
	
	By \eqref{Edef}, \eqref{Einvdef}, Definition \ref{definition85} and Definition \ref{definition88}, 
	$(E \circ \Psi)(z) E(z)^{-1}$ has entries
	\begin{equation} \label{EpsiEinv} \sum_{l=1}^3 e_j(z, \lambda_l(z))
		\psi_j(z, \lambda_l(z)) \widetilde{e}_k(z, \lambda_l(z)), \quad j,k=1,2,3. \end{equation}
		We recall that 
	$e_j$ and $\psi_j$ are analytic at $P_{0}$,
	while $\widetilde{e}_k$ has at most a double pole,
	Then \eqref{EpsiEinv} is analytic in a neighborhood of
	$0$, except for an isolated singularity at $0$.
	The singularity is removable as \eqref{EpsiEinv} behaves
	like $\mathcal{O}(z^{-2/3})$ as $z \to 0$.
	Thus $(E \circ \Psi) E^{-1}$ remains bounded near $0$. 
		
	At the other points $z_{\min}$, $z_{\max}$ and $z^*$,
	the term with $l=1$ in \eqref{EpsiEinv} remains
	bounded.  For the sum of the other two terms 
	we can argue as in the proof of {\bf RHP-T4}
	in the RH problem \ref{rhpforT} for $T$ and we obtain
	the boundedness as well.
\end{proof}

\section{Final transformation} \label{section9}

The final transformation is

\begin{definition} We define for $z \in \mathbb C \setminus
	\Sigma_S$,
\begin{equation} \label{Rdef} 
	R(z) =  S(z) M(z)^{-1}.
	\end{equation}
\end{definition}

\begin{figure}
	\centering
	\begin{tikzpicture}[scale=3]
		\begin{scope}[ultra thick,decoration={
				markings,
				mark=at position 0.5 with {\arrow{>}}}
			] 
			
	%		\draw [postaction={decorate}] (-1.5,0)--(0.6,0);	
	%		\draw (0.6,0)--(1.1,0);
	%		\draw [postaction={decorate}] (1.1,0)--(1.5,0);
			
	%		\filldraw (0.6,0)  circle (0.7pt);	
			\filldraw (-0.85,0) circle (0.7pt);
	%		\filldraw  (1.1,0) circle (0.7pt);
			\filldraw (0.85,0) circle (0.7pt);
	%		\draw (0.6,0) node [below]{$z_{\min}$};
	%		\draw (1.1,0) node [below]{$z_{\max}$};
			\draw (-0.93,0) node [below]{$-1$};
			\draw (0.88,0) node [left, below]{$1$};
			\draw (0,0.74)  node[above]{$\mathbb T$};
			\coordinate (a) at (0.85,0);
			\coordinate (b) at (0,0.8);
			\coordinate (c) at (-0.85,0);
			\coordinate (d) at (0,-0.8);
			\coordinate (e) at (0.85,0);             
			\path[draw,use Hobby shortcut,closed=true]  (a)..(b)..(c)..(d)..(e); 
			\draw (0,0.5)  node[above]{$|z|=1-\eta$};
			\coordinate (a) at (0.75,0);
			\coordinate (b) at (0,0.7);
			\coordinate (c) at (-0.75,0);
			\coordinate (d) at (0,-0.7);
			\coordinate (e) at (0.75,0);               
			\path[draw,use Hobby shortcut,closed=true]  (a)..(b)..(c)..(d)..(e); 	
			\draw (0,0.88)  node[above]{$|z|=1+\eta$};
			\coordinate (a) at (0.95,0);
			\coordinate (b) at (0,0.9);
			\coordinate (c) at (-0.95,0);
			\coordinate (d) at (0,-0.9);
			\coordinate (e) at (0.95,0);               
			\path[draw,use Hobby shortcut,closed=true]  (a)..(b)..(c)..(d)..(e); 
			\draw [postaction={decorate}] (0,-0.7) node{};
			\draw [postaction={decorate}] (0,-0.8) node{};
			\draw [postaction={decorate}] (0,-0.9) node{};
		\end{scope}
	\end{tikzpicture}
	
	\caption{Contour $\Sigma_R = \mathbb T \cup \{|z| = 1\pm \eta \}$ for the RH problem \ref{rhpforR} for $R$.
		\label{fig:SigmaR}} 
\end{figure}
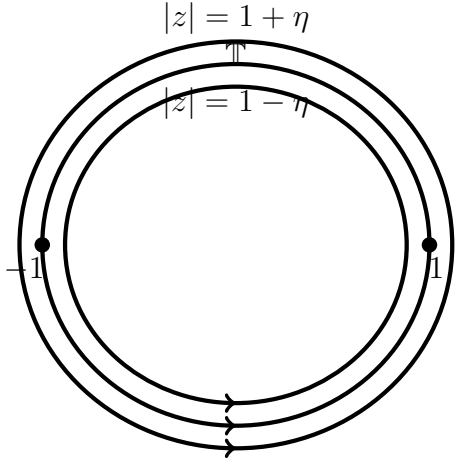

With the definition \eqref{Rdef} we restore the property that
$\det R = 1$. 
We have the following RH problem for $R$.

\begin{rhproblem} \label{rhpforR} $R$ satisfies the following
	\begin{description}
	\item[RHP-R1]
	$R$ has analytic continuation to $\mathbb C \setminus \Sigma_R$
	where $\Sigma_R$ is the union of the three circles
	$\mathbb T$ and $|z| = 1 \pm \eta$.
	\item[RHP-R2] $R_+ = R_- J_R$ on $\Sigma_R$ where
		\begin{equation} \label{Rjump} 
			J_R = M J_S M^{-1}. \end{equation}
	\item[RHP-R3] as $z \to \infty$
	\begin{equation} \label{Rasymp} 
		R(z) = I_6 + \mathcal{O}(z^{-1})  \end{equation}
	\end{description}
\end{rhproblem}
\begin{proof}
	{\bf RHP-R1}  By {\bf RHP-S2} and {\bf RHP-M2},  $S$ and $M$ have the same jumps
	on the real line, and therefore $R$ is analytic
	across the real line, except possibly at the points
	$0, z_{\min}, z_{\min}, z^*$ where $M$ is not invertible.
	At these points we have that $R$ remains bounded,
	due to {\bf RHP-S4} and {\bf RHP-M4}, and thus
	the singularities are removable.

{\bf RHP-R2} is immediate from \eqref{Rdef}, while
{\bf RHP-R3} is immediate from {\bf RHP-S3} and {\bf RHP-M3}.
\end{proof}
The jump matrices $J_R$ tend to the identity matrix
at an exponential rate 
as $n \to \infty$, and we obtain the following.

\begin{corollary} \label{corol93}
	For some $c > 0$, we have 
\begin{equation} \label{Restimate} R(z) = I_6 + \mathcal{O}\left( \frac{e^{-cN}}{1+|z|}\right) \end{equation}
	as $N \to \infty$, uniformly for $z \in \mathbb C \setminus \Sigma_R$.
\end{corollary}
\begin{proof}
	Because of Lemma \ref{lemma73} and the expressions
	\eqref{Sjump1} and \eqref{Sjump2} for $J_S$ we have
	\[ J_S = I_6 + \mathcal{O}(e^{-cN}), \quad\text{ as } N \to \infty, \]
	uniformly on $\Sigma_R$. Then by \eqref{Rjump}
	and the fact that $M$ remains uniformly on $\Sigma_R$
	we obtain the same estimate
\[ J_R = I_6 + \mathcal{O}(e^{-cN}) \quad \text{ as } N \to \infty, \]
as $N \to \infty$, uniformly on $\Sigma_R$.
The corollary follows from standard estimates
on small norm RH problems \cite{Dei99}.
\end{proof} 

\section{Remaining proofs} \label{section10}
\subsection{Proof of Theorem \ref{theorem13}} \label{section101}

Starting from \eqref{PNinY}, and using \eqref{Xdef} and \eqref{Tdef}, we get
\begin{multline} \label{PNinT} P_N(z) E(z)  = \begin{pmatrix} I_3 & 0_3 \end{pmatrix}
	X(z) \begin{pmatrix} I_3 \\ 0_3 \end{pmatrix} \\
	 =  \begin{pmatrix} I_3 & 0_3 \end{pmatrix}
	T(z) \begin{pmatrix} I_3 \\ 0_3 \end{pmatrix}
	\diag \left( e^{Ng_1(z)},
	 e^{Ng_2(z)}, e^{Ng_3(z)} \right).
	\end{multline}
Thus, with $G(z)$ as in \eqref{Gform} it is 
\begin{align} \label{PNinT2} 
	P_N(z)  =  \begin{pmatrix} I_3 & 0_3 \end{pmatrix}
	T(z) \begin{pmatrix} I_3 \\ 0_3 \end{pmatrix}
	E(z)^{-1} G(z)^N.
	\end{align}
For $z$ outside the lens around $\mathbb T$, that is,
for $|z| < 1- \eta$ or $|z| > 1+\eta$, we have \eqref{Sdef}
and thus 
\begin{equation} \label{PNinS} 
	P_N(z) = L^N \begin{pmatrix} I_3 & 0_3 \end{pmatrix}
		S(z) \begin{pmatrix} I_3 \\ 0_3 \end{pmatrix} E(z)^{-1} G(z)^N.
			\end{equation} 
Since $S = RM$ with $R$ satisfying the estimate from
Corollary \ref{corol93}, we obtain
\eqref{PNasymp} with  
\begin{equation} \label{ANdef}
	A_N(z) = \begin{pmatrix} I_3 & 0_3 \end{pmatrix}
	M(z) \begin{pmatrix} I_3 \\ 0_3 \end{pmatrix} E(z)^{-1}.
\end{equation} 
Since $\eta > 0$ can be taken arbitrarily close
to $0$, we obtain \eqref{PNasymp} uniformly
for $z$ in compact subsets of $\overline{\mathbb C} \setminus \mathbb T$.

The items (a)-(d) in Theorem \ref{theorem13} were also established along the way. Indeed $G$ is given
by \eqref{Gform}, with $g_1, g_2, g_3$
as in \eqref{gjdef} and \eqref{gjasymp} is satisfied because of \eqref{hatgjasymp}. The expansion 
\eqref{GNasymp} holds because of \eqref{EgEinvexpansion}. 

Since $M$ depends on the parity of $N$, see RH problem \ref{rhpforM}, there are $M_e$ and $M_o$
such that
\[ M(z) = \begin{cases} M_e(z), & \text{ if $N$ is even}, \\
	M_o(z), & \text{ if $N$ is odd}.
	\end{cases} \]
Thus, by \eqref{ANdef}, also $A_N$ depends only on 
the parity of $N$. The limit \eqref{ANasymp} follows from \eqref{Masymp} and \eqref{ANdef}.

By its definition \eqref{ANdef}, $A_N$ is analytic in $\mathbb C \setminus (\mathbb R \cup \mathbb T)$.
Because of \eqref{Ejump} and \eqref{Mjump1},
$A_N$ has analytic continuation across
$(-\infty,z_{\min})$ and $(z_{\max},\infty)$. 
On $[z_{\min}, z_{\max}]$ we have the jump
\eqref{ANjump1} because of \eqref{Ejump} and \eqref{Mjump2}, which shows that $A_e$ is actually
also analytic across the full real line.

Finally, to show that $G$ has the jump \eqref{Gjump1} on $[z_{\min}, z_{\max}]$,
we use \eqref{Gform}, \eqref{Ejump} and
the fact that by \eqref{gjdef} and 
\eqref{elljpm}
\[ g_{j+} - g_{j-} = \widehat{g}_{j+}
	- \widehat{g}_{j-} - (\ell_j^+ - \ell_j^-) 
	= \widehat{g}_{j+} - \widehat{g}_{j-}
	\pm \pi i. \]
Since $\widehat{g}_1$ satisfies
\eqref{g1g2g3pm4} and \eqref{g1g2g3pm5},
while  $\widehat{g}_2$ and $\widehat{g}_3$
are analytic across $[z_{\min}, z_{\max}]$
we find that $e^{g_{1+}} = e^{g_{1-}}$,
$e^{g_{2+}} = - e^{g_{2-}}$ and
$e^{g_{3+}} = - e^{g_{3-}}$ on  
$[z_{\min}, z_{\max}]$ and
\eqref{Gjump1} follows.

\subsection{Proof of Theorem \ref{theorem14} (b)}
\label{section102}

For $z \in \mathbb C$, we use $z^{(1)}, z^{(2)}, z^{(3)}$
to denote the three points on the Riemann surface, whose
$z$-coordinate is equal to $z$. Then for $q \in \mathcal R \setminus \{ P_{\infty}\}$ one has  that
\[ G_{P_{\infty}}(z^{(1)}, q) + G_{P_{\infty}}(z^{(2)}, q)
	+
	G_{P_{\infty}}(z^{(3)}, q) \]
is harmonic for $z$ in $\mathbb C \setminus \{z(q) \}$,
with 
\[ G_{P_{\infty}}(z^{(1)}, q) + G_{P_{\infty}}(z^{(2)}, q) +
G_{P_{\infty}}(z^{(3)}, q)  = 
\begin{cases} - \log |z| + o(1), & \text{ as } z \to \infty, \\
	\log |z-z(q)| + O(1), &\text{ as } z \to z(q). 
\end{cases} \]
This implies 
\[ G_{P_{\infty}}(z^{(1)}, q) + G_{P_{\infty}}(z^{(2)}, q)
+
G_{P_{\infty}}(z^{(3)}, q) = - \log |z-z(q)|. \]

By \eqref{gpdef} and \eqref{gjdef} we then
have
\begin{multline*} 
 \Re \left(	g_1(z) + g_2(z) + g_3(z) \right) \\
	= -3  \int \left(
		G_{P_{\infty}}(z^{(1)}, q) + G_{P_{\infty}}(z^{(2)}, q)
		+
		G_{P_{\infty}}(z^{(3)}, q) \right) d\mu(q) \\
	 = 3 \int \log |z-z(q)| d\mu(q) 
	 = 3 \int \log |z- s| d\mu_*(s),
\end{multline*}
which is \eqref{Unu}.

\subsection{Proof of Theorem \ref{theorem14} (c)}
\label{section103}

Both $\det A_e$ and $\det A_o$ are analytic
functions in $\mathbb C \setminus \mathbb T$ 
and both tend to $1$ at infinity.
Thus they are not identically zero in $|z| > 1$,
and there can be no accumulation point of zeros in
the domain $|z| > 1$ exterior to the unit circle.

If $\det A_e$ or $\det A_o$ would be identically zero inside
the unit circle, then it follows from \eqref{detPNasymp}
that 
\begin{equation} \label{ChebyshevT} \sup_{|z| \leq 1} \left|\det P_N(z) \right|  \leq C e^{-cN}  \end{equation}
for some $c > 0, C > 0$, and for $N$ even or odd.
However, $z^{-3N} \det P_{N}(z)$ is analytic with
value $1$ at infinity, and hence by the maximum principle
applied to the domain $|z| > 1$ we have
\[ \sup_{z \in \mathbb T} \det |P_N(z) | \geq 1, \]
which is a contradiction with \eqref{ChebyshevT}
if $N$ is large enough. 

Thus $\det A_e$ and $\det A_o$ have only 
finitely many zeros outside of any ring domain 
$1-\eta < |z| < 1 + \eta$ around the unit circle.
Then by \eqref{detPNasymp} all but finitely many
(independent of $N$) zeros of $\det P_N$ are in
$1-\eta < |z| < 1 + \eta$ for $N$ large enough.
We also obtain from \eqref{detPNasymp}
and \eqref{Unu} that 
\begin{equation} \label{detPNasym1} 
	\lim_{N \to \infty} 
	\frac{1}{3N} \log | \det P_N(z)|  
 	= \int \log |z-s| d\mu_*(s)
	\end{equation}
for every $z \in \mathbb C \setminus \mathbb T$
that is not a zero of $A_e$ or $A_o$. 

Let $\nu_N = \nu(\det P_N)$ denote the
normalized zero counting measure as in \eqref{nudetPNdef}. Then \eqref{detPNasym1} can be equivalently stated as  
\begin{equation} \label{detPNasym2} 
	\lim_{N \to \infty} 
\int \log |z-s| d\nu_N(s) 
= \int \log |z-s| d\mu_*(s),
\end{equation}
for every $z \in \mathbb C \setminus \mathbb T$
that is not a zero of $A_e$ or $A_o$,
and in particular it is true almost everywhere
with respect to two-dimensional Lebesgue measure.

Let $\nu$ be a weak$^*$
limit of a subsequence of $(\nu_N)_N$. Then
$\nu$ is a probability measure on $\mathbb T$,
and by  the lower envelope theorem \cite[Theorem I.6.9]{ST97}
\[ \int \log|z-s| d\nu(s) = \int \log|z-s| d\mu_*(s),\]
almost everywhere. Then the unicity theorem 
\cite[Theorem II.2.1]{ST97} implies that $\nu = \mu_*$.

We proved part (d) of Theorem \ref{theorem14}.

\appendix

\section{Appendix: random tilings} \label{appendixA}

\begin{figure}[t]
	\begin{center}
		\begin{tikzpicture}[scale=0.3,transform shape](15,10)(0,0)
			% perfect matching
			
			% blue dimers
			\foreach \x/\y in {-3/1,-3/2,-3/3,-3/4, 
				-1.5/0.5,-1.5/1.5,-1.5/2.5,-1.5/3.5, -1.5/5.5, 
				0/0, 0/1, 1.5/-0.5, 1.5/7.5, 
				3/1, 3/3, 3/4, 4.5/-0.5, 4.5/6.5, 
				6/8, 7.5/-0.5,7.5/3.5,7.5/4.5, 7.5/6.5, 7.5/7.5, 
				9/2,9/3,9/4, 
				10.5/1.5,10.5/2.5,10.5/3.5, 10.5/6.5, 10.5/7.5,
				12/6,12/7,13.5/5.5, 13.5/6.5}
			{ \foreach \v in {1.732}	
				{ \draw[draw=blue, very thick] (\x,\y*\v) --++ (1/2,-\v/2);
					\draw[draw=black, fill=blue!50] (\x-1/2, \y*\v+\v/2) --++ (1.5,-\v/2) --++ (0,-\v) --++ (-1.5,\v/2) -- cycle;
				}
			}
			
			% red dimers
			\foreach \x/\y in {-3/5,-3/6, -1.5/6.5, 
				0/2, 0/3, 0/5, 0/7,
				1.5/0.5, 1.5/2.5, 1.5/3.5, 1.5/5.5,
				3/-1, 3/6, 3/7, 
				4.5/0.5, 4.5/2.5, 4.5/3.5, 4.5/7.5,
				6/-1, 6/1, 6/3, 6/4, 6/6, 
				7.5/1.5,  9/-1, 9/6, 9/7, 10.5/-0.5,
				12/0, 12/1, 12/2, 12/3,
				13.5/0.5, 13.5/1.5, 13.5/2.5, 13.5/3.5}
			{ \foreach \v in {1.732}	
				{
					\draw[draw=black, fill=red] (\x-1/2, \y*\v-\v/2) --++ (1.5,\v/2) --++ (0,\v) --++ (-1.5,-\v/2) -- cycle;
				}
			}		
			
			% yellow dimers
			\foreach \x/\y in {-3/4.5, -1.5/4,-1.5/6, 
				0/1.5, 0/4.5, 0/6.5, 1.5/0, 1.5/2, 1.5/5, 1.5/8,
				3/-1.5, 3/1.5, 3/4.5, 3/5.5, 3/8.5,
				4.5/-2, 4.5/0, 4.5/2, 4.5/5, 4.5/7, 4.5/9,  
				6/-1.5, 6/0.5, 6/2.5, 6/5.5, 6/8.5,
				7.5/0, 7.5/1, 7.5/5, 7.5/8, 
				9/0.5, 9/4.5, 9/5.5, 10.5/4, 10.5/5, 12/4.5}
			{ \foreach \v in {1.732}	
				{ \draw[draw=yellow!100!black, very thick] (\x+1/2,\y*\v) --++ (1,0);
					\draw[draw=black, fill=yellow!100!black] (\x-1/2, \y*\v) --++ (1.5,-\v/2) --++ (1.5,\v/2) --++ (-1.5,\v/2) -- cycle;
				}
			}		
			
	\end{tikzpicture} \end{center}
	\caption{Lozenge tiling of a hexagon \label{fig:tiling}} \end{figure}
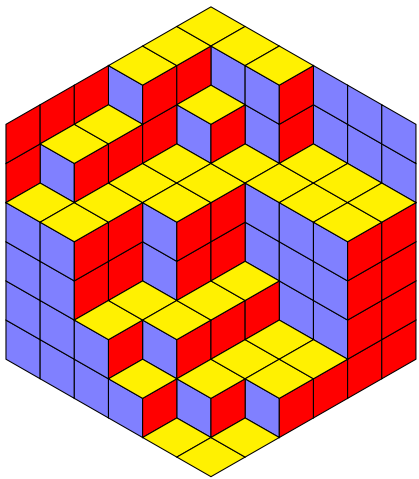

In this appendix we briefly describe how the
matrix valued orthogonal polynomials considered in this
paper arise in the study of random tiling models of a hexagon.

A regular hexagon can be covered with lozenges of
three types as shown in Figure~\ref{fig:tiling}.
A tiling is equivalent to a family of non-intersecting paths
going from left to right that follow the blue and red lozenges
(in the coloring of Figure~\ref{fig:tiling}).
The paths lie on an underlying directed graph $G=(V,E)$ with
vertex set $V = \mathbb Z^2$ and edges going
from $(j,k)$ to either $(j+1,k)$ or $(j+1,k+1)$.
The non-intersecting paths start at consecutive
vertices $(0,0), (0,1), \ldots, (0, N-1)$ and
end at $(2N,N), (2N,N+1), \ldots (2N,2N-1)$.
Here $N$ is the side length of the hexagon.

Given a weighting $w : E \to \mathbb R^+$ of the edges,
we assign a weight to a non-intersecting path system  $\mathcal P$
\[ w(\mathcal P) = \prod_{e\in \mathcal P} w(e) \]
as well as a probability 
\[ \Prob(\mathcal P) = \frac{1}{Z} w(\mathcal P) \]
with $Z = \sum_{\mathcal P'} w(\mathcal P')$. 
This random model is a determinantal point process on the 
set of vertices $V$, in the sense that
there is a kernel $K : V \times V \to \mathbb R$
such that for any set of distinct vertices $v_1, \ldots, v_k$
one has that 
\[ \det \left[ K(v_i,v_j) \right]_{i,j=1}^k \]
is equal to the probability that the non-intersecting
path system passes through all of the vertices $v_1, \ldots, v_k$. 
This is a consequence of the Lindstr\"om-Gessel-Viennot lemma \cite{GV85, Lin73} which says  that the probability 
that the non-intersecting paths cover a given set of vertices 
is equal to a product of determinants
of a certain type, combined with the Eynard-Mehta theorem \cite{EM98}
that says any such probability measure is determinantal.
Eynard and Mehta \cite{EM98} also gave a formula for $K$.

Given a weighting $w$ we have two functions $a: V \to \mathbb R^+$,
$b: V \to \mathbb R^+$ defined on the vertices,
namely
\[ a_{jk} = w((j-1,k-1), (j,k-1)), \quad b_{jk} = w((j-1,k-1), (j,k)). \] 
Note that we write $a_{jk}$ and $b_{jk}$ instead of $a(j,k)$ and $b(j,k)$. The weighting is periodic with periods $p$ and $q$
if 
\begin{equation} \label{abper} 
	a_{j+np,k+mq} = a_{jk}, \quad b_{j+np,k+mq} = b_{jk} 
	\end{equation}
for every $n, m \in \mathbb Z$.   See Figure \ref{fig:paths}
for $3\times 3$ doubly periodic weights on the tiling from Figure \ref{fig:tiling}.  As indicated in the figure, given
 a tiling of the hexagon, we may also consider the 
 weights to be on the lozenges, where a blue
 lozenge carries the weight $a_{jk}$ and a red lozenge
 carries the weight $b_{jk}$ if 
 the vertex $(j-1,k-1)$ is on its left boundary.
In this setup the horizontal (yellow) tiles have weight $1$.

\begin{figure}[t]
	\begin{center}
		\begin{tikzpicture}[scale=0.5,transform shape](15,10)(0,0)
			% perfect matching
			
			% blue dimers
			\foreach \x/\y in {-3/1,-3/2,-3/3,-3/4, 
				-1.5/0.5,-1.5/1.5,-1.5/2.5,-1.5/3.5, -1.5/5.5, 
				0/0, 0/1, 1.5/-0.5, 1.5/7.5, 
				3/1, 3/3, 3/4, 4.5/-0.5, 4.5/6.5, 
				6/8, 7.5/-0.5,7.5/3.5,7.5/4.5, 7.5/6.5, 7.5/7.5, 
				9/2,9/3,9/4, 
				10.5/1.5,10.5/2.5,10.5/3.5, 10.5/6.5, 10.5/7.5,
				12/6,12/7,13.5/5.5, 13.5/6.5}
			{ \foreach \v in {1.732}	
				{ 
					\draw[draw=blue!50, very thick] (\x-1/2,\y*\v) --++ (1.5,-\v/2);
					\draw[draw=black] (\x-1/2, \y*\v+\v/2) --++ (1.5,-\v/2) --++ (0,-\v) --++ (-1.5,\v/2) -- cycle;
				}
			}
			{ \foreach \h/\v in {0.2/1.732}
				{
					\draw (-3+\h,1*\v) node{\Huge $a_{11}$};
					\draw (-3+\h,2*\v) node{\Huge $a_{12}$};
					\draw (-3+\h,3*\v) node{\Huge $a_{13}$};
					\draw (-3+\h,4*\v) node{\Huge $a_{11}$};
					\draw (-1.5+\h,0.5*\v) node{\Huge $a_{21}$};
					\draw (-1.5+\h,1.5*\v) node{\Huge $a_{22}$};
					\draw (-1.5+\h,2.5*\v) node{\Huge $a_{23}$};
					\draw (-1.5+\h,3.5*\v) node{\Huge $a_{21}$};
					\draw (-1.5+\h,5.5*\v) node{\Huge $a_{23}$};
					\draw (0+\h,0*\v) node{\Huge $a_{31}$};
					\draw (0+\h,1*\v) node{\Huge $a_{32}$};
					\draw (1.5+\h,-0.5*\v) node{\Huge $a_{11}$};
					\draw (1.5+\h,7.5*\v) node{\Huge $a_{13}$};
					\draw (3+\h,1*\v) node{\Huge $a_{23}$};
					\draw (3+\h,3*\v) node{\Huge $a_{22}$};
					\draw (3+\h,4*\v) node{\Huge $a_{23}$};
					\draw (4.5+\h,-0.5*\v) node{\Huge $a_{32}$};
					\draw (4.5+\h,6.5*\v) node{\Huge $a_{33}$};
					\draw (6+\h,8*\v) node{\Huge $a_{12}$};
					\draw (7.5+\h,-0.5*\v) node{\Huge $a_{23}$};
					\draw (7.5+\h,3.5*\v) node{\Huge $a_{21}$};
					\draw (7.5+\h,4.5*\v) node{\Huge $a_{22}$};
					\draw (7.5+\h,6.5*\v) node{\Huge $a_{21}$};
					\draw (7.5+\h,7.5*\v) node{\Huge $a_{22}$};
					\draw (9+\h,2*\v) node{\Huge $a_{33}$};
					\draw (9+\h,3*\v) node{\Huge $a_{31}$};
					\draw (9+\h,4*\v) node{\Huge $a_{32}$};
					\draw (10.5+\h,1.5*\v) node{\Huge $a_{13}$};
					\draw (10.5+\h,2.5*\v) node{\Huge $a_{11}$};
					\draw (10.5+\h,3.5*\v) node{\Huge $a_{12}$};
					\draw (10.5+\h,6.5*\v) node{\Huge $a_{12}$};
					\draw (10.5+\h,7.5*\v) node{\Huge $a_{13}$};
					\draw (12+\h,6*\v) node{\Huge $a_{22}$};
					\draw (12+\h,7*\v) node{\Huge $a_{23}$};
					\draw (13.5+\h,5.5*\v) node{\Huge $a_{32}$};
					\draw (13.5+\h,6.5*\v) node{\Huge $a_{33}$};
				}
			}
			
			% red dimers
			\foreach \x/\y in {-3/5,-3/6, -1.5/6.5, 
				0/2, 0/3, 0/5, 0/7,
				1.5/0.5, 1.5/2.5, 1.5/3.5, 1.5/5.5,
				3/-1, 3/6, 3/7, 
				4.5/0.5, 4.5/2.5, 4.5/3.5, 4.5/7.5,
				6/-1, 6/1, 6/3, 6/4, 6/6, 
				7.5/1.5,  9/-1, 9/6, 9/7, 10.5/-0.5,
				12/0, 12/1, 12/2, 12/3,
				13.5/0.5, 13.5/1.5, 13.5/2.5, 13.5/3.5}
			{ \foreach \v in {1.732}	
				{\draw[draw=red, very thick] (\x-1/2,\y*\v) --++ (1.5,\v/2);
					\draw[draw=black] (\x-1/2, \y*\v-\v/2) --++ (1.5,\v/2) --++ (0,\v) --++ (-1.5,-\v/2) -- cycle;
				}
			}		
			
			{ \foreach \h/\v in {0.2/1.732}
				{
					\draw (-3+\h,5*\v) node{\Huge $b_{12}$};
					\draw (-3+\h,6*\v) node{\Huge $b_{13}$};
					\draw (-1.5+\h,6.5*\v) node{\Huge $b_{21}$};
					\draw (0+\h,2*\v) node{\Huge $b_{33}$};
					\draw (0+\h,3*\v) node{\Huge $b_{31}$};
					\draw (0+\h,5*\v) node{\Huge $b_{33}$};
					\draw (0+\h,7*\v) node{\Huge $b_{32}$};
					\draw (1.5+\h,0.5*\v) node{\Huge $b_{12}$};
					\draw (1.5+\h,2.5*\v) node{\Huge $b_{11}$};
					\draw (1.5+\h,3.5*\v) node{\Huge $b_{12}$};
					\draw (1.5+\h,5.5*\v) node{\Huge $b_{11}$};
					\draw (3+\h,-1*\v) node{\Huge $b_{21}$};
					\draw (3+\h,6*\v) node{\Huge $b_{22}$};
					\draw (3+\h,7*\v) node{\Huge $b_{23}$};
					\draw (4.5+\h,0.5*\v) node{\Huge $b_{33}$};
					\draw (4.5+\h,2.5*\v) node{\Huge $b_{32}$};
					\draw (4.5+\h,3.5*\v) node{\Huge $b_{33}$};
					\draw (4.5+\h,7.5*\v) node{\Huge $b_{32}$};
					\draw (6+\h,-1*\v) node{\Huge $b_{12}$};
					\draw (6+\h,1*\v) node{\Huge $b_{11}$};
					\draw (6+\h,3*\v) node{\Huge $b_{13}$};
					\draw (6+\h,4*\v) node{\Huge $b_{11}$};
					\draw (6+\h,6*\v) node{\Huge $b_{13}$};
					\draw (7.5+\h,1.5*\v) node{\Huge $b_{22}$};
					\draw (9+\h,-1*\v) node{\Huge $b_{33}$};
					\draw (9+\h,6*\v) node{\Huge $b_{31}$};
					\draw (9+\h,7*\v) node{\Huge $b_{32}$};
					\draw (10.5+\h,-0.5*\v) node{\Huge $b_{11}$};
					\draw (12+\h,0*\v) node{\Huge $b_{22}$};
					\draw (12+\h,1*\v) node{\Huge $b_{23}$};
					\draw (12+\h,2*\v) node{\Huge $b_{21}$};
					\draw (12+\h,3*\v) node{\Huge $b_{22}$};
					\draw (13.5+\h,0.5*\v) node{\Huge $b_{33}$};
					\draw (13.5+\h,1.5*\v) node{\Huge $b_{31}$};
					\draw (13.5+\h,2.5*\v) node{\Huge $b_{32}$};
					\draw (13.5+\h,3.5*\v) node{\Huge $b_{33}$};
			}}	
			
			% yellow dimers
			\foreach \x/\y in {-3/4.5, -1.5/4,-1.5/6, 
				0/1.5, 0/4.5, 0/6.5, 1.5/0, 1.5/2, 1.5/5, 1.5/8,
				3/-1.5, 3/1.5, 3/4.5, 3/5.5, 3/8.5,
				4.5/-2, 4.5/0, 4.5/2, 4.5/5, 4.5/7, 4.5/9,  
				6/-1.5, 6/0.5, 6/2.5, 6/5.5, 6/8.5,
				7.5/0, 7.5/1, 7.5/5, 7.5/8, 
				9/0.5, 9/4.5, 9/5.5, 10.5/4, 10.5/5, 12/4.5}
			{ \foreach \v in {1.732}	
				{ 
					\draw[draw=black] (\x-1/2, \y*\v) --++ (1.5,-\v/2) --++ (1.5,\v/2) --++ (-1.5,\v/2) -- cycle;
				}
			}		
	\end{tikzpicture} \end{center}
	\caption{Non-intersecting paths and $3 \times 3$ doubly periodic weights
		on the tiling from Figure \eqref{fig:tiling}.
		 \label{fig:paths}} \end{figure}
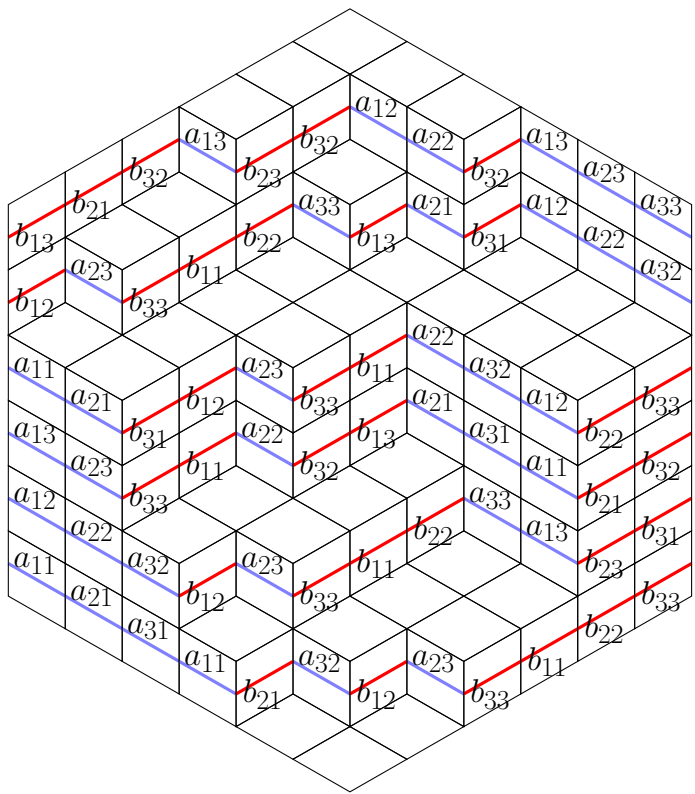

For doubly periodic weightings there is an alternative formula
for the correlation kernel $K$ given by Duits and Kuijlaars \cite{DK21}.
We present it here for the hexagon of size $qN$.
Then we have $2qN$ transition matrices of size $q \times q$
of the form
\[ T_j(z) = \begin{pmatrix} a_{j1} & b_{j1} & 0 & \cdots & \cdots & 0 \\
		0 & a_{j2} & b_{j2} & 0 & \cdots & \vdots \\
		\vdots & \ddots  & \ddots & \ddots & \ddots   & \vdots  \\
		\vdots &  & \ddots & \ddots & \ddots & 0  \\
		0 & \cdots & \cdots  & 0 & a_{j,q-1} &  b_{j,q-1} \\
		b_{jq} z & 0 & \cdots & \cdots & 0 & a_{jq} \end{pmatrix},
		\quad j=1,\ldots,2qN. \]
The matrix valued function
$T_j(z)$ models the transition from horizontal level $j-1$
to level $j$, and it takes into account the periodicity
\eqref{abper} in the second argument. For each $j, j' \in \{0,\ldots, 2qN\}$ 
with $j < j'$ we also write
\[ T_{j \to j'}(z)  = T_{j+1}(z) \cdot T_{j+2}(z) \cdot \cdots \cdot T_{j'-1}(z) \cdot T_{j'}(z) \]
and this models the transition from level $j$ to $j'$.

Take two vertices $v_1$, $v_2$ with coordinates
\[ v_1 = (j, qNy + k),  \qquad v_2 = (j', q Ny'+ k')\]
with integer $Ny,Ny'$ and $j, j' \in \{0,1, \ldots 2N \}$,
$k,k'\in \{0, \ldots, q-1\}$. 
Then $K(v_1,v_2)$ is equal to the $(k+1,k'+1)$th entry of the matrix
\begin{multline} \label{DKkernel} - \frac{\chi_{j>j'}}{2\pi i} \oint_{\gamma}
	\frac{T_{j' \to j}(z)}{z^{N(y-y')}} \frac{dz}{z} 
	+ \frac{1}{(2\pi i)^2} \oint_\gamma \oint_\gamma
	\frac{T_{j' \to 2qN}(z_1)}{z_1^{2N-Ny'}}  
	R_N(z_1,z_2)  \frac{T_{0\to j}(z_2)}{z_2^{Ny}}  \frac{dz_1 d z_2}{z_2},
\end{multline}
see \cite[Theorem 4.7]{DK21}, where $\gamma$ is a closed
contour going once around the origin in the positive direction.
The factor $R_N(z_1,z_2)$ in the double integral in \eqref{DKkernel} 
is the reproducing kernel for the MVOP with matrix weight
function $\frac{T_{0 \to 2qN}(z)}{z^{2N}}$. It is a bivariate matrix polynomial in two variables, of size $q \times q$
and degree $\leq N-1$
in both variables satisfying,
for every matrix valued polynomials $P$ and $Q$ of degree $\leq N-1$,
\begin{align*} 
	\frac{1}{2\pi i}
	\oint_{\gamma} P(z_1) \frac{T_{0 \to 2qN}(z_1)}{z_1^{2N}}  R_N(z_1,z_2) dz_1 & = P(z_2), \\
 \frac{1}{2\pi i}
	\oint_{\gamma} R_N(z_1,z_2) \frac{T_{0 \to 2qN}(z_2)}{z_2^{2N}} Q(z_2) dz_2 & = Q(z_1).
	\end{align*}
The reproducing kernel $R_N$ can be expressed directly in
terms of the solution of the RH problem for MVOP, see
\cite{Del10} and \cite[Proposition 4.9]{DK21}, as
\begin{equation} \label{RNinY} R_N(z_1, z_2) = \frac{1}{z_2-z_1} \begin{pmatrix} 0_3 & I_3 \end{pmatrix}  Y^{-1}(z_1) Y(z_2) \begin{pmatrix} I_3 \\ 0_3 \end{pmatrix}. 
	\end{equation}
	
The formula \eqref{DKkernel} is valid without any assumption
 on periodicity in the first components in \eqref{abper}. In the case
of periodicity with period $p$, and assuming $2q$ is a multiple of $p$, then
\[ T_{0 \to 2qN}(z) = W(z)^{\frac{2qN}{p}}, \]
with 
\[ W(z) = T_{0\to p}(z) =  T_1(z) \cdot T_2(z) \cdot \cdots \cdots T_p(z).  \]
For vertices of the form
\[ v_1 = (pNx,qNy), \quad v_2 = (pNx', qNy'),\]
with integer $Nx, Ny, Nx', Ny'$, 
the correlation kernel $K(v_1,v_2)$ is then
equal to the $(1,1)$ entry of the matrix   
\begin{multline} \label{DKkernel2} 
	- \frac{\chi_{x> x'}}{2\pi i} \oint_{\gamma}
	\frac{W^{N(x-x')}(z)}{z^{N(y-y')}} \frac{dz}{z} \\
	+ \frac{1}{(2\pi i)^2} \oint_\gamma \oint_\gamma
	\frac{W^{N(2\frac{q}{p} - x')}(z_1)}{z_1^{2N-Ny'}}  
	R_N(z_1,z_2)  \frac{W^{N x}(z_2)}{z_2^{Ny}}  
	\frac{dz_1 d z_2}{z_2}.
\end{multline}

The situation \eqref{Tj3} and \eqref{Wz} considered
here corresponds to $p=q=3$. 
In a follow-up paper we aim to study the large $N$
behavior of \eqref{DKkernel2} for the special class
of weightings considered in this paper. 
The large $N$ behavior of \eqref{RNinY}  will then follow from
the asymptotic analysis for $Y$ that we performed.
This part of the analysis is independent of the coordinates
$(x,y)$ and $(x',y')$ of the vertices inside the hexagon.
We expect that it can be followed by a steepest descent analysis of the
integrals in \eqref{DKkernel2} that do depend on these
coordinates, which would in particular explain the 
different asymptotic
phases as illustrated in Figure \ref{fig:hexagon}.  

The three phases (solid, rough and smooth) were first described
in \cite{KOS06}. The smooth phase is a special feature of 
doubly periodic models. In recent years there has been a lot of
activity in the analysis of such models, in particular 
for doubly
periodic domino tilings of the Aztec diamond, see 
\cite{Bai23, BCJ18, B21, BB23+, BD19, BB24+, CD23, CJ16, DK21, JM23, Pio24+}.
We expect that the techniques developed in this paper will be helpful to 
study analogous properties of doubly periodic lozenge tilings of
the hexagon.

\begin{figure}[t]
	\begin{center}
\includegraphics[trim= 0cm 0cm 0cm 0cm, clip,scale=0.25]{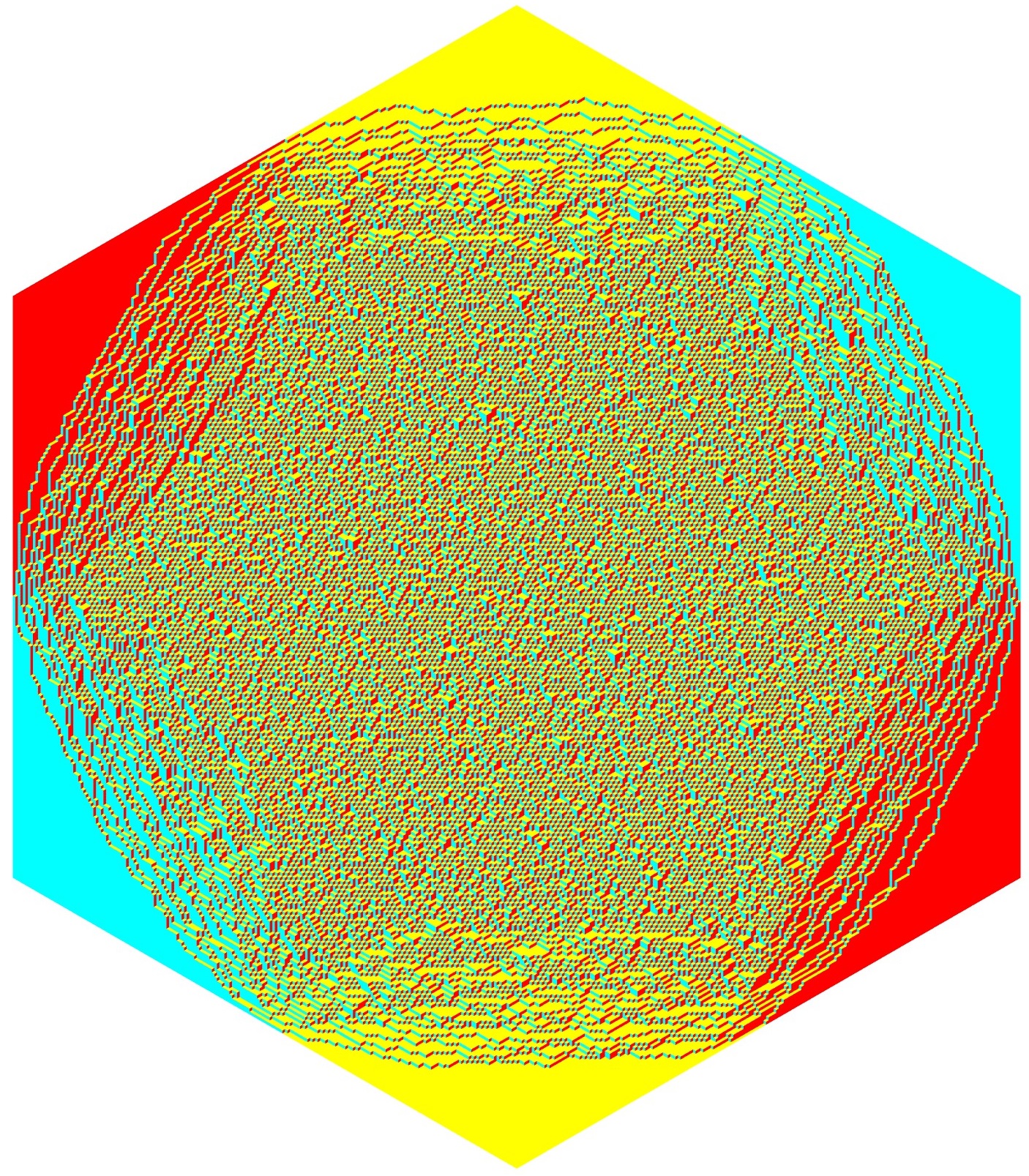}
\caption{Random tiling of a hexagon with
$3 \times 3$ doubly periodic weights of the form considered in this paper. The parameters are $N=200$ and $\alpha_1 = \alpha_2 = 0.3$, see \eqref{Tjconcrete}. The model exhibits three asymptotic phases: solid, rough and smooth. 
The  smooth phase in the middle has six cusps.
% The figure is due to Julian Mauersberger with software of Christophe Charlier. 
\label{fig:hexagon}}
	\end{center}
	
	\end{figure}
	
	\subsection*{Acknowledgement}
	
	I thank Tomas Berggren and Mateusz Piorkowski for stimulating discussions. I am very grateful to Christophe Charlier and 
	Julian Mauersberger for providing me with high quality
	figures of random tilings, such as Figure \ref{fig:hexagon}.
	
	The author is supported by long term structural funding-Methusalem grant of the Flemish Government, and 
	by FWO Flanders project G.0910.20.


\begin{thebibliography}{99}
	\bibitem{AS24+} V. Alves and G.L.F. Silva,
	The P\'olya-Tchebotarev problem with semiclassical external fields, 
	arXiv:2403.00719.
	
	
%	\bibitem{AK11}
%	A.I. Aptekarev and A.B.J. Kuijlaars,
%	Hermite-Pad\'e approximations and  multiple orthogonal polynomials ensembles, 
%	Russian Math. Surveys 66 (2011), no. 6, 1133--1199.

	\bibitem{Bai23}
	E. Bain, 
	Local correlation functions of the two-periodic weighted Aztec diamond in mesoscopic limit,
	J. Math. Phys.  64 (2023), paper 023301, 64 pp.
	
	\bibitem{BCJ18}
	V. Beffara, S. Chhita, K. Johansson, 
	Airy point process at the liquid-gas boundary, 
	Ann. Probab. 46 (2018), 2973--3013. 
	
	\bibitem{B21}
	T. Berggren, Domino tilings of the Aztec diamond with doubly periodic weightings, 
	Ann. Probab. 49 (2021), 1965--2011. 
	
	\bibitem{BB23+} T. Berggren and A. Borodin, 
	Geometry of the doubly periodic Aztec dimer model, arXiv:2306.07482.
	
	\bibitem{BD19} 	T. Berggren and M. Duits, 
	Correlation functions for determinantal processes defined by infinite block Toeplitz minors, 
	Adv. Math. 356 (2019) 106766. 
	
	\bibitem{Ber21} M. Bertola, 
	Pad\'e approximants on Riemann surfaces and KP tau functions,
	Anal. Math. Phys. 11 (2021), no. 4, Paper No. 149, 38 pp. 
	
	\bibitem{Ber22} M. Bertola, 	
	Nonlinear steepest descent approach to orthogonality on elliptic curves,
	J. Approx. Theory 276 (2022), Paper No. 105717, 33 pp. 
	
	\bibitem{Ber23} M. Bertola, 
	Abelianization of matrix orthogonal polynomials,
	Int. Math. Res. Not. IMRN 2023 (2023), no. 10, 8544--8595. 
	
	\bibitem{BGK23} M. Bertola, A. Groot, and A.B.J. Kuijlaars,
	Critical measures on higher genus Riemann surfaces,
	Comm. Math. Phys. 404 (2023), 51--95.
	
	\bibitem{BI99} 	P. Bleher and A. Its, 
	Semiclassical asymptotics of orthogonal polynomials, Riemann-Hilbert problem, and universality in the matrix model, 
	Ann. Math. 150 (1999) 185--266.
	
	\bibitem{BB24+} A.I. Bobenko and N. Bobenko, 
	Dimers and $M$-curves: limit shapes from Riemann surfaces, arXiv:2407.19462.
	
	\bibitem{BD23} A. Borodin and M. Duits,
	Biased $2 \times 2$ periodic Aztec diamond and an elliptic curve,
	Prob. Theory Rel. Fields 187 (2023), 259--315.
	
	\bibitem{BCT23} C. Boutillier, D. Cimasoni, and B. de Tilli\`ere,
	Minimal bipartite dimers and higher genus Harnack curves,
	Probab. Math. Phys. 4 (2023), 151--208.
	
	\bibitem{CM12} 
	G.A. Cassatella-Contra and M. Ma\~nas, 
	Riemann-Hilbert problems, matrix orthogonal polynomials and discrete matrix equations with singularity confinement, 
	Stud. Appl. Math. 128 (2012), 252--274.
	
	\bibitem{Cha21a} C. Charlier, 
	Doubly periodic lozenge tilings of a hexagon and matrix valued orthogonal polynomials,
	Stud. Appl. Math. 146 (2021), 3--80.
	
	\bibitem{Cha21b}  C. Charlier, 
	Matrix orthogonality in the plane versus scalar orthogonality in a Riemann surface, 
	Trans. Math. Appl. 5 (2021), tnab004, 35 pp.
	
	\bibitem{CDKL20} C. Charlier, M. Duits, A.B.J. Kuijlaars, and J. Lenells,
	A periodic hexagon tiling model and non-hermitian orthogonal polynomials,
	Comm. Math. Phys. 378 (2020), 401--466.
	
	\bibitem{CD23} S. Chhita and M. Duits,
	On the domino shuffle and matrix refactorizations,
	Comm. Math. Phys. 401 (2023), 1417--1467.
	
	\bibitem{CJ16} S. Chhita and K. Johansson,
	Domino statistics of the two-periodic Aztec diamond,
	Adv. Math. 294 (2016), 37--149.   
	
	\bibitem{Chi18} E.M. Chirka, 
	Potentials on a compact Riemann surface, 
	Proc. Steklov Inst. Math. 301 (2018), 272--303. 
	
	\bibitem{Chi19}	E.M. Chirka, 
	Equilibrium measures on a compact Riemann surface, 
	Proc. Steklov Inst. Math. 306 (2019), 296--334.
	
	\bibitem{Chi20} E.M. Chirka, 
	Capacities on a compact Riemann surface, 
	Proc. Steklov Inst. Math. 311 (2020), 36--77.

	\bibitem{DPS08} D. Damanik, A. Pushnitski, and B. Simon, 
	The analytic theory of matrix orthogonal polynomials, 
	Surv. Approx. Theory 4 (2008) 1--85. 

	\bibitem{DKR23} A. Dea\~no, A.B.J. Kuijlaars, and P. Rom\'an,
	Asymptotics of matrix valued orthogonal polynomials on $[-1,1]$, 
	Adv. Math. 423 (2023), Paper 109043, 61 pp. 
	
	\bibitem{Dei99} P. Deift,
	Orthogonal Polynomials and Random Matrices: a Riemann–Hilbert Approach, Courant Lecture
	Notes in Mathematics 3. Amer. Math. Soc, Providence, RI, 1999.
	
	\bibitem{DKMVZ99} P. Deift, T. Kriecherbauer, K.T-R McLaughlin, S. Venakides, and X. Zhou,
	Uniform asymptotics for polynomials orthogonal with respect to varying exponential weights and applications to universality questions in random matrix theory,
	Comm. Pure Appl. Math. 52 (1999), 1335--1425.
	
	\bibitem{DZ93} 	P. Deift and X. Zhou, 
	A steepest descent method for oscillatory Riemann-Hilbert problems. Asymptotics for the MKdV equation, 
	Ann. Math. 137 (1993), 295--368.
	
	\bibitem{Del10} S. Delvaux, 
	Average characteristic polynomials for multiple orthogonal polynomial ensembles, 
	J. Approx. Theory 162 (2010), 1033--1067
	
	\bibitem{DIP24+} H. Desiraju, A.R. Its, and A. Prokhorov,
	Nonlinear steepest descent on a torus: A case study of the Landau-Lifshitz equation, preprint arXiv:2405.17662.
	
	\bibitem{DLR23+} H. Desiraju, T.L. Latimer, and
	P. Roffelsen,
	On a class of elliptic orthogonal polynomials and their integrability, to appear in Constructive Approximation,
	preprint arXiv:2305.04404.
	
	\bibitem{DK21} 	M. Duits and A.B.J. Kuijlaars, 
	The two-periodic Aztec diamond and matrix valued orthogonal polynomials, 
	J. Eur. Math. Soc. 23 (2021), 1075--1131.
	
	
%	\bibitem{DG05} A. Dur\'an and F.A. Gr\"unbaum, 
%	A survey on orthogonal matrix polynomials satisfying second order differential equations, 
%	J. Comput. Appl. Math. 178 (2005), 169--190.

%	\bibitem{DLS99} A.J. Dur\`an, P. L\'opez-Rodr\'iguez, and E.B. Saff, 
%	Zero asymptotic behaviour for orthogonal matrix polynomials, J. d'Analyse Math. 78 (1999), 37--60. 
	
	\bibitem{EM98} B. Eynard and M.L. Mehta, 
	Matrices coupled in a chain. I. Eigenvalue correlations. J. Phys. A 31 (1998),  4449--4456.
	
	\bibitem{FOX23} M. Fasondini, S. Olver, and Y. Xu,
	Orthogonal polynomials on planar cubic curves,
	Found. Comput. Math. 23 (2023), 1--31.
	
	\bibitem{FIK92} A.S. Fokas, A.R. Its, and A.V. Kitaev, 
	The isomonodromy approach to matrix models in 2D quantum gravity,
	Comm. Math. Phys. 147 (1992), 395--430.
	
	\bibitem{GV85} I. Gessel and G. Viennot, 
	Binomial determinants, paths, and hook length formulae. Adv. Math. 58 (1985), 300--321.
	
	\bibitem{GK21} 	A. Groot and A.B.J. Kuijlaars, 
	Matrix-valued orthogonal polynomials related to hexagon tilings,
	J. Approx. Theory 270 (2021), 105619, 36 pp. 
	
	\bibitem{GIM11} F.A. Gr\"unbaum, M.D. de la Iglesia, and A. Mart\'nez-Finkelshtein, 
	Properties of matrix orthogonal polynomials
	via their Riemann-Hilbert characterization, 
	SIGMA Symmetry Integr. Geom. Methods Appl. 7 (2011) 098,
	31 pages.
	
%	\bibitem{GKL24+} A. Guionnet, K. Kozlowski and A. Little,
%	Asymptotic expansion of the partition function for $\beta$-ensembles with complex potentials, preprint arXiv:2411.10610.

	\bibitem{JM23} K. Johansson and S. Mason,
	Dimer-dimer correlations at the the rough-smooth boundary,
	Comm. Math. Phys. 400 (2023), 1255--1315.

    \bibitem{KM17+} N.-G. Kang and N. Makarov, 
	Calculus of conformal fields on a compact Riemann surface, arXiv:1708.07361, 86 pp.

	
	\bibitem{KO06} R. Kenyon and A. Okounkov, 
	Planar dimers and Harnack curves, 
	Duke Math.~J. 131 (2006), 499--524.
	 
	\bibitem{KOS06} R. Kenyon, A. Okounkov, and S. Sheffield, 
	Dimers and amoebae, 
	Ann. Math. 163 (2006), 1019--1056.
	
%	\bibitem{Kui10a} 
%	A.B.J. Kuijlaars, Multiple orthogonal polynomial ensembles,
%	In: Recent trends in orthogonal polynomials and approximation theory,
%	(J. Arves\'u, F. Marcell\'an, and A. Mart\'inez-Finkelshtein, eds.) 
%	Contemp. Math., 507, Amer. Math. Soc., Providence, RI, 2010,
%	pp. 155--176.
	
%	\bibitem{Kui10b}
%	A.B.J. Kuijlaars, Multiple orthogonal polynomials in random matrix theory,
%	In: Proceedings of the International Congress of Mathematicians, Volume III (R. Bhatia, et al., eds.), 
%	Hindustan Book Agency, New Delhi, 2010, pp. 1417--1432. 
	
	
	\bibitem{KP24+} A.B.J. Kuijlaars and M. Piorkowski,
	Wiener-Hopf factorization and matrix-valued orthogonal polynomials, arXiv:2402.07706.
	
	\bibitem{KS15} 	A.B.J. Kuijlaars and G.L.F. Silva, 
	S-curves in polynomial external fields, 
	J. Approx. Theory 191 (2015), 1--37.
	
	
	\bibitem{Lin73} B. Lindstr\"om,
	On the vector representations of induced matroids,
	Bull. London Math. Soc. 5 (1973), 85--90.
	
	\bibitem{MMS06} A. Mart\'inez-Finkelshtein, 
	K.T.-R. McLaughlin and E.B. Saff,
	Szeg\H{o} orthogonal polynomials with respect to an analytic weight: canonical representation and strong asymptotics,
	Constr. Approx. 24 (2006), 319--363.
	
	\bibitem{MR11} A. Mart\'inez-Finkelshtein and E.A. Rakhmanov, 
	Critical measures, quadratic differentials and weak limits of zeros of Stieltjes polynomials, 
	Comm. Math. Phys. 302 (2011), 53--111.
	
	\bibitem{MR16} A. Mart\'inez-Finkelshtein and E.A. Rakhmanov, 
	Do orthogonal polynomials dream of symmetric curves?
	Found. Comput. Math. 16 (2016), 1697--1736. 
	
	\bibitem{MS16} A. Mart\'inez-Finkelshtein and G.L.F. Silva, 
	Critical measures for vector energy: global structure of trajectories of quadratic differentials,
	Adv. Math. 302 (2016), 1137--1232.
	
	\bibitem{Pio24+} M. Piorkowski,
	Arctic curves of periodic dimer models and generalized discriminants,
	arXiv:2410.17138.
	
	\bibitem{Rak12} E.A. Rakhmanov, 
	Orthogonal polynomials and $S$-curves, 
	in: Recent Advances in Orthogonal Polynomials, Special Functions, and their Applications, 
	Contemp. Math. 578, Amer.Math. Soc., Providence RI, pp. 195--239, 2012.
	
	\bibitem{ST97} E.B. Saff and V. Totik, 
	Logarithmic Potentials with External Fields, in: Grundlehren der mathematischen
	Wissenschaften, vol. 316, Springer-Verlag, Berlin, 1997, second edition 
	Springer Nature Switzerland, 2024.
	
	\bibitem{Ski15} B. Skinner,
	Logarithmic Potential Theory on Riemann Surfaces,Dissertation (Ph.D.), California Institute
	of Technology, 
	 doi:10.7907/Z9Q52MK8,
	 \verb|https://resolver.caltech.edu/CaltechTHESIS:05292015-072640484|
	
	\bibitem{Sta85} H. Stahl, 
	Extremal domains associated with an analytic function. I, II. Complex Variables Theory
	Appl. 4 (1985), 311--324, 325--338.
	
	\bibitem{Sta86}
	H. Stahl, Orthogonal polynomials with complex-valued weight function. I, II. Constr. Approx. 2 (1986), 225--240, 241--251.
\end{thebibliography}
\end{document}